\documentclass[11pt,reqno,a4paper]{amsart}

\usepackage{graphicx}
\usepackage{color}
\usepackage[USenglish]{babel}
\usepackage{latexsym}
\usepackage{amsmath}
\usepackage{amsfonts}
\usepackage{amssymb}
\usepackage{comment}
\usepackage{esint}
\usepackage[all]{xy}

\newcommand{\grad}{\nabla}

\renewcommand{\to}{\rightarrow}
\newcommand{\pa}{\partial}

\newcommand{\rife}[1]{(\ref{#1})}
\newcommand{\ov}[1]{\overline{#1}}

\newcommand{\dsp}{\displaystyle}

\newcommand{\intbar}{\mathop{\int\makebox(-15.5,0){\rule[6pt]{.7em}{0.3pt}}\kern-6pt}\nolimits}

\newcommand{\ii}{\infty}

\newcommand{\al}{\alpha}

\renewcommand{\a }{\alpha }

\newcommand{\D }{\Delta }

\newcommand{\e}[1]{{\,\dsp e^{\dsp #1}}}

\newcommand{\rh }{\rho }

\renewcommand{\o }{\omega }
\renewcommand{\O }{\Omega }

\def\o{\omega}
\def\p{\partial}

\newcommand{\be}{\begin{equation}}
\newcommand{\ee}{\end{equation}}
\newcommand{\beq}{\begin{equation}}
\newcommand{\eeq}{\end{equation}}

\newcommand{\R}{\mathbb{R}}

\newcommand{\N}{\mathbb{N}}

\newcommand{\dis}{\displaystyle}

\newcommand{\vn}{v_n}

\newcommand{\Ks}{\overline{K}_n}

\newcommand{\xndn}{\frac{x_n}{\delta_n}}
\newcommand{\txndn}{\tfrac{x_n}{\delta_n}}

\newcommand{\Otn}{\O_{t}^n}
\newcommand{\Ovn}{\O_{v_n^{\ast}(r)}^n}
\newcommand{\vns}{v_n^{\ast}}
\newcommand{\ton}{t_{0}}
\newcommand{\brx}{B_r(x)}

\newcommand{\sing}{\big|\tfrac{\tau_n}{|x_n|}y+\tfrac{x_n}{|x_n|}\big|^{2\alpha_n}}
\newcommand{\singx}{\big|\tfrac{\tau_n}{|x_n|}x+\tfrac{x_n}{|x_n|}\big|^{2\alpha_n}}
\newcommand{\Rns}{\overline{R}_n}

\newtheorem{theorem}{Theorem}[section]
\newtheorem{proposition}[theorem]{Proposition}
\newtheorem{definition}{Definition}[section]
\newtheorem{corollary}[theorem]{Corollary}
\newtheorem{remark}[theorem]{Remark}
\newtheorem{example}[theorem]{Example}

\newtheorem{lemma}[theorem]{Lemma}

\newcommand{\bpr}{\begin{proposition}}
\newcommand{\epr}{\end{proposition}}
\newcommand{\bex}{\begin{example}\rm}
\newcommand{\eex}{\end{example}}
\newcommand{\brm}{\begin{remark}\rm}
\newcommand{\erm}{\end{remark}}
\newcommand{\bdf}{\begin{definition}\rm}
\newcommand{\edf}{\end{definition}}
\newcommand{\bte}{\begin{theorem}}
\newcommand{\ete}{\end{theorem}}
\newcommand{\ble}{\begin{lemma}}
\newcommand{\ele}{\end{lemma}}
\newcommand{\bco}{\begin{corollary}}
\newcommand{\eco}{\end{corollary}}
\newcommand{\mycomment}[1]{}

\DeclareMathOperator*{\esssup}{ess.\,sup}
\DeclareMathOperator*{\essinf}{ess.\,inf}

\begin{document}

\newtheorem{lem}{Lemma}[section]
\newtheorem{pro}[lem]{Proposition}
\newtheorem{thm}[lem]{Theorem}
\newtheorem{rem}[lem]{Remark}
\newtheorem{cor}[lem]{Corollary}
\newtheorem{df}[lem]{Definition}

\title[A Harnack type inequality for singular Liouville type equations]
{A Harnack type inequality\\  for singular Liouville type equations}

\author{P. Cosentino}

\address{Paolo Cosentino, Department of Mathematics, Ph.D. School, University of Rome {\it ``Tor Vergata"} \\  Via della ricerca scientifica n.1, 00133 Roma, Italy. }
\email{cosentino@mat.uniroma2.it}


\thanks{P.C. is partially supported by the MIUR Excellence Department Project MatMod@TOV
awarded to the Department of Mathematics, University of Rome Tor Vergata, and by  INdAM-GNAMPA Project ''Mancanza di regolarit\`{a} e spazi non lisci: studio di autofunzioni e autovalori``, codice CUP E53C23001670001.}

\begin{abstract}
We obtain a Harnack type inequality for solutions of the Liouville type equation,
\begin{equation}\nonumber
 -\D u=|x|^{2\alpha}K(x)e^{\displaystyle u} \qquad\text{in} \,\,\, \Omega,
\end{equation}
where $\a\in(-1,0)$, $\O$ is a bounded domain in $\mathbb{R}^2$ and $K$ satisfies,
\begin{equation}\nonumber
 0<a\leq K(x)\leq b<+\infty.
\end{equation} 
This is a generalization to the singular case of a result by C.C. Chen and C.S. Lin [Comm. An. Geom. 1998], which considered the regular case $\alpha=0$. \\
%The singular problem is more delicate and needs a different approach. Part of the arguments of Chen-Lin can be adapted to the singular case by means of an isoperimetric inequality for surfaces with conical singularities. The rest of the proof actually requires a different approach, due to the loss of translation invariance of the problem.
Part of the argument of Chen-Lin can be adapted to the singular case by means of an isoperimetric inequality for surfaces with conical singularities. However, the case $\alpha\in(-1,0)$ turns out to be more delicate, due to the lack of traslation invariance of the singular problem, which requires a different approach.

\end{abstract}

\maketitle


\

\section{Introduction}
We are concerned with solutions of
the following singular Liouville type equation
\begin{equation}\label{Liouville}
 -\D u=|x|^{2\alpha}K(x)e^{\displaystyle u} \qquad\text{in} \,\,\, \Omega
\end{equation}
where $\a\in(-1,0)$, $\O$ is a bounded domain in $\mathbb{R}^2$ and $K$ satisfies,
\begin{equation}\label{Kbasic}
 0<a\leq K(x)\leq b<+\infty
\end{equation}
We will say that $u$ is a solution of (\ref{Liouville}) if $u$ is a distributional solution of (\ref{Liouville}),
$u\in L^1_{loc}(\O)$ and $|x|^{2\alpha}K(x)e^{\displaystyle u}\in L^1_{loc}(\O)$. As a consequence of the results of Brezis and Merle (\cite{bm}) and standard elliptic regularity
theory (\cite{GT}), such solutions satisfy $u\in C_{loc}^{1,\gamma}(\O\backslash \{0\})\cap C_{loc}^{0,\kappa}(\O)\cap W_{loc}^{2,p}(\O\backslash\{0\})\cap W_{loc}^{2,q}(\O)$, for any $\gamma\in(0,1)$, $k\in (0, k_0)$, $p\geq1$  and $q\in [1,\tfrac{1}{|\alpha|})$, where $ k_0\leq1$ is a constant which depends on $\alpha$. \\

The study of the equation (\ref{Liouville}) has been motivated by many problems arising in different fields.
We mention in particular the conformal geometry of surfaces with conical singularities (\cite{Picard},\cite{troyanov}),
the statistical mechanics description of point vortices in 2d-turbulence and of self-gravitating systems (\cite{clmp1},\cite{clmp2},\cite{wolansky}),
the Electroweak theory of Glashow-Salam-Weinberg (\cite{bt}, \cite{sy2}) and Gauge Field Theories (\cite{Tar}, \cite{tarantello}, \cite{yang}). 
%At last, although we will not discuss these facts here, let us mention that a series of groundbreaking results have been recently pushed forward about the case of the flat two torus with singular sources (\cite{Zchen-lin-1}, \cite{ lin22}, \cite{linwang}), about constant mean curvature immersions of surfaces in hyperbolic 3-manifolds (\cite{tarantello-23}, \cite{TS}), about the concentration phenomenon at singular sources (\cite{wei-zhang-19}, \cite{wei-zhang-24}) and about the so called ``sphere covering inequality'' (\cite{GM}, \cite{GHM}). 
\\

We are going to analyze the so called ``$\sup+\inf$'' inequality (also known as ``Harnack type" inequality (\cite{yy})) for singular Liouville
type equations of the form \eqref{Liouville}.\\
We first recall few facts about the ``$\sup+\inf$" inequality for the so called ``regular case'' , i.e. $\alpha=0$. It was first conjectured in the work of Brezis and Merle (\cite{bm}) that for any  compact set $A\Subset\O$ there exists a constant $C_1\geq 1$, which depends only on $a$ and $b$ and a positive constant $C_2$, which depends also on the distance dist$(A,\partial \O)$, such that,
\begin{equation}\label{sup+inf.intro}
 \underset{A}\sup \,u + C_1 \underset{\O}\inf\, u\leq C_2,
\end{equation}
for any solution of \eqref{Liouville} and \eqref{Kbasic}, with $\alpha=0$. The conjecture was in fact proved by Shafrir (\cite{sh}) in 1992.
Further assuming that,
\begin{equation}\label{KLip}
 ||\nabla K||_\infty\leq C,
\end{equation}
Brezis, Li and Shafrir (\cite{bls}) showed that in fact \eqref{sup+inf.intro} holds with $C_1=1$ and $C_2$ depending also on $C$. An interesting
open problem was left in \cite{bls} which was whether or not \eqref{sup+inf.intro} still holds true with $C_1=1$ if $K$ just satisfies a uniform H\"{o}lder condition.
This question was in fact addressed by Chen and Lin (\cite{CL}) in 1998, which actually proved that \eqref{sup+inf.intro}
with $C_1=1$ holds under even weaker assumptions on $K$. \\ Indeed, let $u$ be a solution  of \eqref{Liouville} with $\alpha=0$ and $K$ satisfy \eqref{Kbasic}. Also, let us assume that there exist $0\leq \overline\rho \leq \frac{1}{2}$, $\overline\sigma\geq 1$ and $B\geq 0$ such that if $|x-y|\leq \overline\rho$, then
 \begin{equation} \label{K1}
  \frac{K(x)}{K(y)}\leq \overline\sigma + \frac{B}{|\log|x-y||}.
 \end{equation}
Then, it has been shown in \cite{CL} that \eqref{sup+inf.intro} holds with $C_1=\sqrt{\overline\sigma}$ and
$C_2$ depending on $a,b,A, \O, \overline{\rho}, \overline{\sigma},$ $ B$. In particular, in case $\overline{\sigma}=1$, we see that the answer to the question of Brezis-Li-Shafrir is affirmative whenever $K$ admits just a uniform modulus of continuity of logarithmic type.\\

Concerning this point, our aim is to generalize this result to solutions of the singular Liouville equation \eqref{Liouville} with $\alpha\in (-1,0)$.
 Actually, this is a non trivial variation of the case $\alpha=0$, mainly because the singularity in the equation breaks
 the translation invariance of the problem. This issue is well known and in fact both the results of Shafrir (\cite{sh}) and Brezis-Li-Shafrir (\cite{bls})
 has been generalized to the case $\al\in (-1,0)$ in \cite{b0}, \cite{b2} respectively.
 At least to our knowledge no generalization of this sort is at hand for the result in \cite{CL} as we describe hereafter.
 We cannot discuss here the many subtle aspects related to the same problem in case $\alpha>0$ and refer to the works of Tarantello (\cite{Tar3},
 \cite{tarantello2005}) and Bartolucci-Tarantello (\cite{BT-2}) and references quoted therein for further details.

%Due to this fact, the blow-up analysis for (\ref{Liouville}) is more delicate and the bubbling behavior of the problem is richer. For example, by arguing as in (\cite{tarantello2005}), it is possible to show that non-radial (with respect to the singularity) blow-up sequences of solutions of \eqref{Liouville}, \eqref{Kbasic} and \eqref{KLip} exist such that their asymptotic profile can only be defined in a shrinking ball  around the local maxima and far away from the singularity. \\
%In the context of Liouville-type equations with conical singularities, the lack of traslation invariance has been faced by
%Bartolucci (\cite{b0}, \cite{b1}, \cite{b2}).
%In particular, in (\cite{b0}), Bartolucci generalized the ``$\sup +C_1\inf$'' result, obtained by Shafrir in (\cite{sh}), in the case of conical singularities and for general bounded potentials. Later, in (\cite{b1}), Bartolucci showed that the $C_1$ constant must be bounded from below by 1 and it is precisely 1 assuming a uniform bound on the Lipschitz norm of the potentials (\cite{b2}), generalizing the work of Brezis, Li and Shafrir (\cite{bls}).\\
%In our paper, following the techniques implemented by Chen and Lin in \cite{CL} and emulating the resolution of the loss of traslation invariance described in \cite{b2}, we extend the class of potentials (see property (\ref{K1})) and find out a new proof for the lower bound of the constant $C$ for which the ``$\sup +C\inf$'' holds.%

%This lower bound depends on the particular inequality (\ref{K}) which defines this larger class of potentials.%

\medskip

%Following the techniques implemented by Chen and Lin in \cite{CL}, we extend the class of potentials (see property (\ref{K1})) and find out a new proof for the lower bound of the constant $C$ for which the ``$\sup +C\inf$'' holds.
%Actually, for a certain class of solutions of (\ref{Liouville}),(\ref{K}) and (\ref{K1}), which will be described later, we are able to obtain a finer and sharper %result. \\
%In order to solve the lack of traslation invariance we emulate the procedure described in \cite{b2}. Unfortunately, we cannot use the moving plane tecnique, due to %the low regularity of the potential, so we need to combine non-trivially Chen-Lin procedure with Bartolucci's scheme. In the following, we present the main steps %of the proof and the most significant difficulties which naturally arise going through the procedure.\\

Our main result is the following,
\begin{theorem}\label{teo1.intro} \hfill \\
Let $\O$ be a bounded domain in $\mathbb{R}^2$. Assume $u$ is a solution of
 \begin{equation}\label{maineq}
 -\D u=|x|^{2\alpha}K(x)e^{\displaystyle u} \qquad\text{in} \,\,\, \Omega
\end{equation}
and $K(x)$ satisfies \eqref{Kbasic} and \eqref{K1}. Then, for any compact set $I\subset (-1,0]$, for any $\alpha\in I$, for any compact subset $A\Subset \O$ and for any $\sigma>\overline\sigma$, there exists a constant $C=C(a,b,A,\O,\rho,\sigma,B,I)$ such that,
 \begin{equation}\label{supinf}
 \underset{A}{\sup} \, u+\sqrt{\sigma}\,\underset{\O}{\inf} \, u\leq C.
 \end{equation}
\end{theorem}

The previous theorem easily implies the following corollary:

\begin{corollary} \label{cor1}
Let $\O$ be a bounded domain in $\mathbb{R}^2$ and $\a\in(-1,0]$. \\
Assume $u$ is a solution of (\ref{maineq}) and $K(x)$ satisfy (\ref{Kbasic}) and is H\"{o}lder continuous, that is,
\[
 |K(x)-K(y)|\leq \tilde C|x-y|^{\theta},
\]
for $x,y\in\overline{\O}$ and some constants $\tilde C>0$ and $0<\theta<1$. \\ Then, for any compact set $I\subset (-1,0]$ with $\alpha\in I$, any compact subset $A$ of $\O$ and for any $\sigma>1$, there exists a constant $C=C(a,b,A,\O,\tilde C,\theta,\sigma,I)$ such that it holds the following inequality
 \begin{equation}\nonumber
 \underset{A}{\sup} \, u+\sqrt{\sigma}\,\underset{\O}{\inf} \, u\leq C.
 \end{equation}
\end{corollary}

The proof of this corollary follows directly from theorem \ref{teo1.intro}, just observing that if $K$ is H\"{o}lder continuous then it satisfies (\ref{K1}) with $\overline\sigma=1$.\\
\hfill \\

Let us illustrate the main ideas behind the proof of Theorem \ref{teo1.intro}. Arguing by contradiction, we assume the existence of a sequence of solutions $u_n$ and potentials $K_n$ satisfying
$$
 -\D u_n=|x|^{2\alpha_n}K_ne^{\dis u_n} \quad \text{in} \,\, \O \quad\qquad \text{with} \,\,\alpha_n\rightarrow \alpha_\infty\in(-1,0]
 $$
 $$
 0<a\leq K_n \leq b<+\infty,
 $$
 the existence of $\overline\rho\in(0,\frac{1}{2}]$, $\overline\sigma\geq1$ and $B\geq0$ such that, for $|x-y|\leq \overline\rho$
 \[
  \frac{K_n(x)}{K_n(y)}\leq\overline\sigma+\frac{B}{|\log|x-y||}
 \]
for every $n$, for which the quantity
\begin{equation}
 \label{assurdo1} \tfrac{1}{\sqrt{\sigma}}\,\underset{A}{\sup} \, u_n+\underset{\O}{\inf} \, u_n\to+\infty, \qquad \text{as} \,\, n\rightarrow+\infty,
\end{equation}
for some $\sigma>\overline \sigma$. In particular, this implies that
\begin{equation}\label{assurdo}
 \tfrac{1}{\sqrt{\overline\sigma}}\,\underset{A}{\sup} \, u_n+\underset{\O}{\inf}\, u_n\rightarrow +\infty  \qquad \text{as} \,\, n\rightarrow+\infty.
 \end{equation}
At this point, let us denote with $x_n$ any maximum point of $u_n$ inside the compact set $A$, namely
\[
 M_n:=u_n(x_n)=\underset{A}\sup\, u_n
\]
so that in particular $M_n\to+\infty$. Let us consider the case in which $x_n\to 0$, otherwise the contradiction follows by
the theorem of Chen-Lin (\cite{CL}) for $\alpha=0$. Moreover, by dilation invariance, we can assume without loss of generality
that $\{x_n\}\subset\subset B_1\subset\subset A$, where $B_1$ denotes the disk in the plane of radius $1$. We run a
well known blow-up argument around this local maximum (\cite{b2}, \cite{bt}, \cite{ls}) which is defined using the following quantity,
$$
\delta_n=\exp({-\tfrac{M_n}{2(1+\alpha_n)}})\to 0.
$$
Let $\rho=\tfrac{\overline{\rho}}{4}$ and let us define,
\[
 v_n(y):=u_n(x_n+\delta_n y)-u_n(x_n), \,\,\,\,\,\,\,\, y\in B_{\frac{\rho}{\delta_n}}(0),
\]
which solves
\[
 -\D {v}_n=|y+\tfrac{x_n}{\delta_n}|^{2\alpha_n}K(x_n +\delta_ny)e^{\dis  v_n(y)}, \,\,\,\,\,\,\,\, y\in B_{\frac{\rho}{\delta_n}}(0).
\]

At this point the lack of translation invariance plays a crucial role, because it becomes fundamental to understand how fast the sequence of local maximum
points $x_n$ is converging to $0$ compared with the ``blow-up rate" $\delta_n$. Therefore we are naturally led to analyze two different cases:
\medskip
\begin{itemize}
 \item[(I)] there exists a constant $C>0$ such that $\tfrac{|x_n|}{\delta_n}\leq C$, for all $n\in\N$,
 \medskip
 \item[(II)] there exists a subsequence such that $\tfrac{|x_n|}{\delta_n}\to+\infty$, as $n\to+\infty$.
\end{itemize}
\medskip
If (I) holds, we can prove the existence of a subsequence of $v_n$ which locally converges to a solution of an entire singular Liouville problem, which has the minimal possible total curvature, that is $4\pi(1+\alpha_\infty)(1+{\tfrac{1}{\sqrt{\overline\sigma}}})$ (\cite{b}).
In particular, in this case we can work out a careful adaptation of the Chen-Lin argument (\cite{CL}), based on a weighted rearrangement and
isoperimetric inequalities. We discuss the details of this part of the proof in Section \ref{Case I}. Remark that in this case we obtain a sharper estimate, namely we prove
that \eqref{supinf} holds with $\sigma=\overline{\sigma}$, see Remark \ref{remark1}.

\medskip

On the other hand, if (II) holds,
then we consider a different blow-up rate, that is,
\[
 \tau_n=\tfrac{\delta_n^{(1+\alpha_n)}}{|x_n|^{\alpha_n}}.
\]
It is still possible to run a blow-up argument, namely defining,
\[
 \hat{v}_n(y):=u_n(x_n+\tau_n y)-u_n(x_n), \,\,\,\,\,\,\,\, y\in B_{\frac{\rho}{\tau_n}}(0),
\]
which solves
\[
 -\D \hat{v}_n=\big|\tfrac{\tau_n}{|x_n|}y+\tfrac{x_n}{|x_n|}\big|^{2\alpha_n}K(x_n +\tau_ny)e^{\dis\hat v_n(y)}, \,\,\,\,\,\,\,\, y\in B_{\frac{\rho}{\tau_n}}(0).
\]
We notice that the rescaled equation contains a singularity in $-\tfrac{x_n}{\tau_n}$, which satisfies $\tfrac{|x_n|}{\tau_n}\to\infty$.
Thus, in this case we are able to prove the existence of a subsequence of $\hat v_n$ which actually converges to a solution of the Liouville problem in $\R^2$ without the conical singularity, whose minimal total curvature is $4\pi(1+{\tfrac{1}{\sqrt{\overline\sigma}}})$ (\cite{CL}).\\
Although the situation could appear easier than the previous one, here we have to deal with a genuine new difficulty. In particular, let us define
\[
 l_n:=\sup\Big\{l\leq \tfrac{\rho}{\tau_n}|\int_{|y|\leq l}\sing K(x_n +\tau_ny)e^{\dis \hat v_n(y)}\,dy\leq 4\pi\Big(1+{\tfrac{1}{\sqrt{\overline\sigma}}}\Big)\Big\}.
\]

This quantity plays a crucial role in the contradiction argument in both cases. The subtle point here is that we do not know how
the location of the singularity (which is $-\tfrac{x_n}{\tau_n}$) behaves with respect to $l_n$.
This forces us again to analyze two different cases separately:
\begin{itemize}
 \item [(i)] there exists $\epsilon_0>0$ such that $\epsilon_0l_n\leq\tfrac{|x_n|}{\tau_n}$,
 \item [(ii)] $\tfrac{|x_n|}{\tau_n}=o(l_n)$.
\end{itemize}
The subcase (i) can be discussed essentially as in \cite{CL}, although again we need a careful adaptation of the
rearrangement argument due to the singularity. However, in this case, once again we can prove that \eqref{supinf} holds with $\sigma=\overline\sigma$, see Remark \ref{remark2}.\\
On the contrary, this is not anymore possible for the subcase (ii), for which a subtle point arise in the argument in \cite{CL},
because for $n$ large we are not able to obtain a good estimate of the decay of the solutions in the ``neck'' region $\tfrac{|x_n|}{\tau_n}\leq |x|\leq l_n$.
Actually, even if we had such an improved estimates, a step by step adaptation of the rearrangement argument in \cite{CL} would not end up with a contradiction,
since, the singularity being contained in the superlevel sets of the solutions for $|x|$ large, we would have a much worse isoperimetric constant, see Proposition \ref{Huber.App.prop} in the appendix. As a consequence, we would miss a suitable version of \eqref{diffeq2} below, which is
a subtle differential inequality for the total ``Gaussian curvature" in the domain. This is why we have to attack the problem by a different approach. The underlying idea is to recover the needed estimates by
using some argument in \cite{ls} together with the ``$\sup+C\inf$'' inequality in \cite{b0}. The argument is not straightforward as it
requires a sort of glueing of different techniques at different scales.\\
Unfortunately, we are not able to prove the sharper result in general, although the set of functions for which the inequality \eqref{supinf} for $\sigma=\overline\sigma$ fails seems to be very thin, see Remark \ref{goodclass}. \\
A natural sufficient condition which guarantees that \eqref{supinf} holds for $\sigma=\overline\sigma$ is for example that,
\[
 \int_{|x-x_n|\leq \rho}|x|^{2\alpha_n}{K}_{n}(y)e^{\displaystyle u_n(x)}\,dx\geq 4\pi\Big(1+{\tfrac{1}{\sqrt{\overline\sigma}}}\Big),
\]
but the problem is still open in general.\\

This paper is organised as follows. \\
In Section \ref{Case (II)} we give a detailed proof of Case (II), while in Section \ref{Case I} we give a sketch of the proof of Case (I). In Section \ref{Remarks and} we make some final remarks on the proof, whereas in Section \ref{Example} we give an explicit example of the inequality and shortly discuss a geometric application of the result. Lastly, in the appendix, Section \ref{Appendixsection}, we present some technical results. In particular, we prove a generalization of a lemma in \cite{suz2}, we state the Huber inequality (\cite{hu}) and prove a lemma about the measure of level sets of solutions to Liouville type equations.

\hfill\\

\section{Case (II)}\label{Case (II)}
In this section we analyze the second possibility, that is, in this case there exists a subsequence such that $\frac{|x_n|}{\delta_n}\to+\infty$. Let us set
\[
\tau_n:=\frac{\delta_n^{1+\alpha_n}}{|x_n|^{\alpha_n}}.
 \]
We notice that, $\tau_n\to 0$ and  $\frac{\tau_n}{|x_n|}=(\frac{\delta_n}{|x_n|})^{(1+\alpha_n)}\to 0$. \\
Let us define
\[
 v_n(y)=u_n(\tau_ny+x_n)-u_n(x_n), \,\,\,\,\,\,\,\, y\in B_{\frac{1}{2\tau_n}}(0)
\]
and consider $n$ sufficiently large so that $|x_n|\leq\frac{1}{2}$ and $\tau_n y+x_n\in A$, whenever $y\in B_{\frac{1}{2\tau_n}}(0)$.
The function $v_n$ satisfies the following
\begin{equation} \label{vequation2}
 -\D v_n(y)=\big|\tfrac{\tau_n}{|x_n|}y+\tfrac{x_n}{|x_n|}\big|^{2\alpha_n}\Ks(y)e^{\displaystyle v_n(y)}, \,\,\,\,\,\,\,\, y\in B_{\frac{1}{2\tau_n}}(0),
\end{equation}
\[
 v_n(y)\leq 0=v_n(0),
\]
where $\Ks(y)=K_n(\tau_ny+x_n)$.  \\
Now, by using Green representation's formula, for a fixed $l$ and for every $|y|\leq l$, it holds that
\begin{align} \nonumber
 v_n(y)&=\int_{|\zeta|\leq 2l}\big|\tfrac{\tau_n}{|x_n|}\zeta+\tfrac{x_n}{|x_n|}\big|^{2\alpha_n} G(y,\zeta)\Ks(\zeta)e^{\dis v_n(\zeta)}\,d\zeta - \int_{|\zeta|=2l}\frac{\partial G}{\partial r}(y,\zeta)\vn(\zeta)\,d\sigma(\zeta) \\ \label{Green4}
 &\geq -C\Big[1 + \int_{|\zeta|=2l}|\vn(\zeta)|\,d\sigma(\zeta)\Big],
\end{align}
where $G(y,\zeta)$ is the Green's function of $-\D$ on $B_{2l}(0)$ and $\frac{\partial G}{\partial r}(y,\zeta)$ its radial derivative calculated in $|\zeta|=2l$, whose formulas are respectively
\[G(y,\zeta)=-\frac{1}{2\pi}\log|y-\zeta|+\frac{1}{2\pi}\log\Big|\tfrac{\tilde{y}-\zeta}{2l}|y|\Big|,
\]
\[\frac{\partial G}{\partial r}(y,\zeta)=\frac{1}{4\pi l}\frac{|y|^2-(2l)^2}{|y-\zeta|^2},
\]
where $\tilde{y}=\frac{(2l)^2}{|y|^2}y$ is the inverse of $y$  with respect to $B_{2l}(0)$.\\
At this point, we want to estimate the integral in (\ref{Green4}).
Firstly, let us define,
for every open set $\omega$, 
\[
\mu_n(\omega):=\int_{\omega}\singx\,dx,
\]
and estimate $\mu_n(B_r(0))$ from below and from above.
We notice that, if $-\frac{x_n}{\tau_n}\notin B_{2r}(0)$, then $\sing\in\big[(\frac{3}{2})^{2\alpha_n},(\frac{1}{2})^{2\alpha_n}\big]$ for every $y\in B_r(0)$,
and this implies that,
\begin{equation}
\nonumber
\big(\tfrac{3}{2})^{2\alpha_n}\pi r^{2}\leq\mu_n(B_r(0))\leq(\tfrac{1}{2})^{2\alpha_n}\pi r^{2}.
\end{equation}
Otherwise, if $-\frac{x_n}{\tau_n}\in B_{2r}(0)$, then we deduce that
\begin{align*}
 \int_{B_r(0)}\big|\tfrac{\tau_n}{|x_n|}x&+\tfrac{x_n}{|x_n|}\big|^{2\alpha_n}\,dx=(\tfrac{\tau_n}{|x_n|})^{2\alpha_n}\int_{B_r(0)}|x+\tfrac{x_n}{\tau_n}|^{2\alpha_n}\,dx \\&\leq (2r)^{-2\alpha_n}\int_{B_r(\frac{x_n}{\tau_n})}|y|^{2\alpha_n}\,dy \\ &\leq (2r)^{-2\alpha_n}\left(\int_{B_r(\frac{x_n}{\tau_n})\cap \{|y|>|y-\frac{x_n}{\tau_n}|\}}|y|^{2\alpha_n}\,dy+\int_{B_r(\frac{x_n}{\tau_n})\cap \{|y|\leq|y-\frac{x_n}{\tau_n}|\}}|y|^{2\alpha_n}\,dy\right) \\ &\leq (2r)^{-2\alpha_n}\left(\int_{B_r(\frac{x_n}{\tau_n})}|y-\tfrac{x_n}{\tau_n}|^{2\alpha_n}\,dy+\int_{B_r(0)}|y|^{2\alpha_n}\,dy\right) \\ &\leq \tfrac{2^{1-2\alpha_n}\pi}{1+\alpha_n}r^2
\end{align*}
and
\begin{align*}
 \int_{B_r(0)}\singx\,dx &=(\tfrac{\tau_n}{|x_n|})^{2\alpha_n}\int_{B_r(0)}|x+\tfrac{x_n}{\tau_n}|^{2\alpha_n}\,dx\\&=(\tfrac{\tau_n}{|x_n|})^{2\alpha_n}\int_{B_r(\frac{x_n}{\tau_n})}|y|^{2\alpha_n}\,dy  \geq
  3^{2\alpha_n} (\tfrac{\tau_n}{|x_n|}r)^{2\alpha_n}\pi r^{2},
\end{align*}
where in the last inequality we used the fact that $|y|^{2\alpha_n}\geq(3r)^{2\alpha_n}$.  \\
Hence, we have that,
\begin{align}
\label{raggibuoni2}
(\tfrac{3}{2})^{2\alpha_n}\pi r^{2}\leq&\mu_n(B_r(0))\leq(\tfrac{1}{2})^{2\alpha_n}\pi r^{2}\hspace{1,55cm}\text{if}\,\, r\leq\tfrac{|x_n|}{2\tau_n},
\\ \label{raggicattivi2}
3^{2\alpha_n}\pi (\tfrac{\tau_n}{|x_n|}r)^{2\alpha_n} r^{2}\leq&\mu_n(B_r(0))\leq (\tfrac{2}{2^{2\alpha_n}})\tfrac{\pi}{1+\alpha_n} r^{2} \hspace{1cm}\text{if}\,\, r>\tfrac{|x_n|}{2\tau_n}.
\end{align}

At this point, we can estimate \eqref{Green4}. Infact, recalling that $v_n\leq0$, then
\begin{eqnarray}\nonumber
 \int_{|\zeta|=2l}|\vn(\zeta)|\,d\sigma(\zeta)&=&\Big|\int_{|\zeta|=2l}\vn(\zeta)\,d\sigma(\zeta)\Big|.
\end{eqnarray}
Now, let us change variables $\zeta=(r\cos(\theta),r\sin(\theta))$ and define $$\hat{v}_n(r,\theta)=v_n((r\cos(\theta),r\sin(\theta))).$$
Thus, we have that
\begin{eqnarray}\nonumber
 \int_{|\zeta|=2l}|\vn(\zeta)|\,d\sigma(\zeta)&=&2l\Big|\int_0^{2\pi}\hat{v}_n(2l,\theta)\,d\theta\Big|\leq 2l\int_0^{2l}\Big|\int_0^{2\pi}\frac{\partial\hat{v}_n}{\partial r}(r,\theta)\,d\theta\Big|dr.
\end{eqnarray}
In order to estimate the last integral we need the following calculations
\begin{align*}\nonumber
 \Big|\int_0^{2\pi}\frac{\partial\hat{v}_n}{\partial r}(r,\theta)\,d\theta\Big|&=\frac{1}{r}\Big|\int_0^{2\pi}\frac{\partial\hat{v}_n}{\partial r}(r,\theta)r\,d\theta\Big|\\
 \nonumber &=\frac{1}{r}\Big|\int_{|\zeta|\leq r}\Delta v_n(\zeta) \,d\zeta\Big| \\ \nonumber &=\frac{1}{r}\int_{|\zeta|\leq r}\big|\tfrac{\tau_n}{|x_n|}\zeta+\tfrac{x_n}{|x_n|}\big|^{2\alpha_n}\Ks(\zeta)e^{\dis{v}_n(\zeta)} \,d\zeta \\ \nonumber
 &\leq \frac{b}{r}\int_{|\zeta|\leq r}\big|\tfrac{\tau_n}{|x_n|}\zeta+\tfrac{x_n}{|x_n|}\big|^{2\alpha_n} \,d\zeta 
 \\ \nonumber
 &\leq C r
 \end{align*}
where we used \eqref{raggibuoni2} and \eqref{raggicattivi2} to estimate the last integral. Hence,
\begin{eqnarray}\nonumber
 \int_{|\zeta|=2l}|\vn(\zeta)|\,d\sigma(\zeta)&\leq& 2l\int_0^{2l}\Big|\int_0^{2\pi}\frac{\partial\hat{v}_n}{\partial r}(s,\theta)\,d\theta\Big|ds\leq C l^{3}
\end{eqnarray}
and, by (\ref{Green4}), for every $|y|\leq l$, we have
\[
 v_n(y)\geq -C(1+l^{3}),
\]
for any fixed $l$.
\\
We notice that, for every $y$ in a compact set, the sequence of points $\tfrac{\tau_n}{|x_n|}y+\tfrac{x_n}{|x_n|}$ is converging, up to a subsequence, to a point $y_0\in\{x\big| |x|=1\}$.
By standard elliptic estimates (\cite{GT}), we can pass to subsequences $\{v_n\}$, $\{\Ks\}$ such that
\begin{eqnarray*}
 v_n\to w \qquad &\text{in}& \,\, C_{loc}^{1,\gamma}(\R^2\backslash \{y_0\})\cap W_{loc}^{2,p}(\R^2\backslash\{y_0\})\cap W_{loc}^{2,q}(\R^2)\cap C_{loc}^{0,\kappa}(\R^2), \\
 \Ks\overset{\star}{\rightharpoonup} K_0 \qquad &\text{in}& \,\, L^{\infty}_{loc}(\R^2),
\end{eqnarray*}
with $\gamma\in(0,1)$, $k\in (0,k_\infty)$ for some $k_\infty\leq1$ which depends on $\alpha_\infty$, $p\geq1$  and $q\in [1,\frac{1}{|\alpha_\infty|})$. Then $w$ weakly satisfies the following equation
\begin{equation}
 -\D w=K_0e^{\displaystyle w} \qquad \text{in}\,\,\mathbb{R}^2.
\end{equation}
From (\ref{Kbasic}) and (\ref{K1}) we deduce that $a\leq K_0\leq b$ and
\[
 \frac{\Ks(x)}{\Ks(y)}\leq {\overline\sigma}+\frac{B}{|\log (\tau_n|x-y|)|}\longrightarrow
{\overline\sigma},\qquad \text{for} \,\,n\to+\infty,\]
for every $x,y$ in a compact set. \\
This implies that for a fixed $\epsilon>0$ and for $n$ sufficiently large
\begin{eqnarray}\label{proofK0}
 \Ks(x)\leq\overline\sigma(1+\epsilon)\Ks(y),
\end{eqnarray}
for $|x-y|<\overline\rho\tau_n^{-1}$. Let us define $A_1:=\{x| K_0(x)>\esssup K_0 -\epsilon\}$ and $A_2:=\{x| K_0(x)<\essinf K_0 +\epsilon\}$ and consider $\tilde{l}$ large enough so that $A_1\cap B_{\tilde{l}}(0)$ and $A_2\cap B_{\tilde{l}}(0)$ have positive measure. Hence,
\[
 \Ks(x)\chi_{A_1}(x)\chi_{A_2}(y)\leq\overline\sigma(1+\epsilon)\Ks(y)\chi_{A_1}(x)\chi_{A_2}(y)
\]
and integrating both sides over $B_{\tilde{l}}(0)\times B_{\tilde{l}}(0)$ and using Fubini, we deduce that
\begin{align*}
 \int_{B_{\tilde l}(0)}\Big(\int_{B_{\tilde l}(0)}\Ks(x)&\chi_{A_1}(x)\,dx\Big)\chi_{A_2}(y)\,dy\leq \\ &\leq\overline\sigma(1+\epsilon)\int_{B_{\tilde l}(0)}\Big(\int_{B_{\tilde l}(0)}\Ks(y)\chi_{A_2}(y)\,dy\Big)\chi_{A_1}(x)\,dx,
\end{align*}
for $n$ sufficiently large and $|x-y|<\overline\rho\tau_n^{-1}$. \\
Passing to the limit, as $n\to+\infty$, and by using the $\star$-weak convergence of $\Ks$, we have
\begin{align*}
 \int_{B_{\tilde l}(0)}\Big(\int_{B_{\tilde l}(0)}K_0(x)&\chi_{A_1}(x)\,dx\Big)\chi_{A_2}(y)\,dy\leq
 \\
 &\leq\overline\sigma(1+\epsilon)\int_{B_{\tilde l}(0)}\Big(\int_{B_{\tilde l}(0)}K_0(y)\chi_{A_2}(y)\,dy\Big)\chi_{A_1}(x)\,dx
\end{align*}
and by using the definition of $A_1$ and $A_2$, then
\[
 (\esssup K_0-\epsilon)<\overline\sigma(1+\epsilon)(\essinf K_0+\epsilon),
\]
that is
\[
 \frac{(\esssup K_0-\epsilon)}{(\essinf K_0+\epsilon)}<\overline\sigma(1+\epsilon).
\]
Letting $\epsilon$ go to $0^+$, we have
\[
 \frac{\esssup \, K_0}{\essinf K_0}\leq \overline\sigma.
\]
Therefore, by using Theorem 1.1 (\cite{CL}), we have that
\[
 I_w=\int_{\R^2}K_0e^{\displaystyle w}\geq 4\pi\Big(1+\sqrt{\frac{ess.inf \, K_0}{ess.sup \, K_0}}\Big)\geq 4\pi\Big(1+{\frac{1}{\sqrt{\overline\sigma}}}\Big).
\]
Now let $\rho=\frac{1}{4}\overline\rho$, $L_n=\rho e^{\frac{M_n}{2}}|x_n|^{\alpha_n}=\rh\tfrac{1}{\tau_n}$ and $G(x,z)$ the Green's function of $-\D$ on $B_{\rho}(x_n)$. Then, using the fact that $G(x_n,\tau_ny+x_n)=-\tfrac{1}{2\pi}\log(|\tau_ny|)+\tfrac{1}{2\pi}\log(\rho)$ we have
\begin{align}
\nonumber
 M_n&=\int_{|z-x_n|\leq\rho}|z|^{2\alpha_n}K_n(z)e^{\dis u_n(z)}G(x_n,z)\,dz + \frac{1}{2\pi\rho}\int_{|z-x_n|=\rho}u_n(z)\,dl \\ \nonumber
 &=\int_{|y|\leq L_n}\big|\tfrac{\tau_n}{|x_n|}y+\tfrac{x_n}{|x_n|}\big|^{2\alpha_n}\Ks(y)e^{\displaystyle v_n(y)}G(x_n,\tau_ny+x_n)\,dy + p_n\\
 &=\int_{|y|\leq L_n}\Big(\tfrac{1}{2\pi}\log(\tfrac{1}{\tau_n})-\tfrac{1}{2\pi}\log\big(\tfrac{|y|}{\rho}\big)\Big)\sing\Ks(y)e^{\displaystyle v_n(y)}\,dy + p_n,  \label{stimamaxbuona}
\end{align}
where $p_n=\frac{1}{2\pi\rho}\int_{|z-x_n|=\rho}u_n(z)\,dl$.
\\

We can now prove that $I_w=4\pi\Big(1+{\frac{1}{\sqrt{\overline\sigma}}}\Big)$. Indeed, if we assume the strict inequality, we can take an $\epsilon_1>0$ and $n,l$ large enough such that $l<L_n$ and
\[
 \int_{|y|\leq l}\sing\Ks(y)e^{\displaystyle v_n(y)}\,dy\geq(1+2\epsilon_1)4\pi\Big(1+{\frac{1}{\sqrt{\overline\sigma}}}\Big).
\]
Then, by using (\ref{stimamaxbuona}),
\begin{eqnarray*}
 M_n&\geq& \int_{|y|\leq l}\Big(\tfrac{1}{2\pi}\log(\tfrac{1}{\tau_n})-\tfrac{1}{2\pi}\log\big(\tfrac{|y|}{\rho}\big)\Big)\sing\Ks(y)e^{\displaystyle v_n(y)}\,dy + p_n \\
 &\geq& \Big(1-\tfrac{\epsilon_1}{1+2\epsilon_1}\Big)\tfrac{1}{2\pi}\log(\tfrac{1}{\tau_n})\int_{|y|\leq l}\sing\Ks(y)e^{\displaystyle v_n(y)}\,dy + p_n \\
 &\geq&(1+\epsilon_1)\Big(1+{\frac{1}{\sqrt{\overline\sigma}}}\Big)(M_n+2\alpha_n\log|x_n|)+ p_n,
\end{eqnarray*}
for $n$ large enough. This, together with the fact that $p_n\geq\underset{\O}\inf u_n$, implies
\[
   \frac{1}{\sqrt{\overline\sigma}}M_n+\underset{\O}\inf u_n\leq -2\alpha_n\Big(1+{\tfrac{1}{\sqrt{\overline\sigma}}}\Big)\log|x_n|,
\]
immediately implying a contradiction to (\ref{assurdo}). So, $I_w=4\pi\big(1+\small{{\frac{1}{\sqrt{\overline\sigma}}}}\big)$ and we can use again Theorem 1.1 (\cite{CL}), to say that $w$ is a radially symmetric and strictly decreasing function with respect to $z_0\in\R^2$. In particular, by using the fact that $v_n$ is locally uniformly converging to $w$ and recalling that $0$ is a maximum point of $v_n$, we have that $z_0=0$.\hfill \\
\medskip
\\
Now we want to find $l_n\leq L_n$ such that we have
\begin{equation}\label{dis7}
 \int_{|y|\leq l_n}\sing \Ks(y)e^{\displaystyle v_n(y)}\,dy\geq 4\pi\Big(1+{\frac{1}{\sqrt{\sigma}}}\Big)-C_1|\log(\tau_n)|^{-1}
\end{equation}
and
\begin{equation}\label{dis8}
 \int_{|y|\leq l_n}\log\big(\tfrac{|y|}{\rho}\big)\sing\overline{K}_{n}(y)e^{\displaystyle v_n(y)}\,dy\leq C_2,
\end{equation}
for $C_1,C_2$ positive constants that do not depend on $n$, as $n\to+\infty$. \\
We define
\[
 l_n:=\sup\Big\{l\leq L_n\Big|\int_{|y|\leq l}\sing\overline{K}_{n}(y)e^{\displaystyle v_n(y)}\,dy\leq 4\pi\Big(1+{\tfrac{1}{\sqrt{\overline\sigma}}}\Big)\Big\}.
\]
From the fact that $I_w=4\pi\Big(1+{\frac{1}{\sqrt{\overline\sigma}}}\Big)$ and $v_n\to w$ uniformly on compact subsets, we deduce that $l_n\to+\infty$. We can choose $\epsilon$ as a positive number to be fixed later on, such that, for $n>>0$, there exist $r_{1,n}>1$ and $r_1>1$ such that $4r_{1,n}\leq \frac{l_n}{4}$ and
\[
 \int_{|y|\leq r_{1}}K_0e^{\dis w}\,dy=4\pi(1+\tfrac{1}{\sqrt{\overline\sigma}})-\epsilon,
\]
\[
 \int_{|y|\leq r_{1,n}}\sing\Ks e^{\displaystyle v_n}\,dy=4\pi(1+\tfrac{1}{\sqrt{\overline\sigma}})-\epsilon.
\]
We notice that $r_{1,n}\to r_1$, as $n\to+\infty$, and,
\begin{equation}\label{stimanello2}
\int_{\{r_{1,n}\leq|y|\leq l_n\}}\sing\Ks e^{\displaystyle v_n}\,dy\leq\epsilon.
\end{equation}

At this point, we need to split the proof into two subcases:
\begin{itemize}
 \item[(i)] it exists $\tilde\epsilon_0>0$ such that
 \begin{equation}\nonumber
 \tilde\epsilon_0 l_n\leq\tfrac{|x_n|}{\tau_n},
\end{equation}
\item [(ii)]
\begin{equation}\nonumber
 \tfrac{|x_n|}{\tau_n}=o(l_n).
\end{equation}
\end{itemize}

\subsection{Subcase (i)}
We skip the proof of this subcase because it can be worked out by a careful adaptation of the rearrangement argument used in \cite{CL}. For a detailed proof of this case, we refer the interested reader to \cite{cos}. Moreover, we point out that in the subcase (i) the ``$\sup+\inf$'' inequality \eqref{supinf} holds with the value $\sigma ={\overline\sigma}$.

\subsection{Subcase (ii)}
In this situation we have that
\begin{equation}
 \nonumber
 \tfrac{|x_n|}{\tau_n}\tfrac{1}{l_n}\to 0,\,\,\,\,\, \text{as}\,\, n\to+\infty.
\end{equation}
Therefore, let us choose a sequence $\Rns$ such that $\Rns\to\infty$ and \\$\Rns=o(\tfrac{|x_n|}{\tau_n})=o(l_n)$, \text{as} $ n\to+\infty$.
For later purposes, we choose $\Rns$ such that,
\begin{equation}
\label{Rnscondition}
 (\log(\overline{R}_n))^2=O(\log(\tfrac{|x_n|}{\tau_n})).
\end{equation}

We observe that $v_n$ is a solution of the inequality $-\D v_n\leq{b}\sing e^{\displaystyle v_n}$ in the set $\{|y|\leq L_n\}$, thus, we apply a generalization of a lemma by Suzuki (\cite{suz2}), whose details can be found in the appendix, section \ref{Suzuki}. \\
Let $r\leq\frac{1}{2}|x|$, $2r_{1,n}\leq |x|\leq \frac{l_n}{2}$ and, by using (\ref{stimanello2}), we choose an $\epsilon$ small enough such that
\begin{align}
 \label{Suzukiargument}
 v_n(x)&\leq\frac{1}{2\pi r}\int_{\p B_r(x)}v_n\,dl-2\log\Big(1-\tfrac{b}{2\beta_{\alpha,B_r(x)}}\int_{B_r(x)}\sing e^{\displaystyle v_n}\,dy\Big)_+ \\ \nonumber
 &\leq \frac{1}{2\pi r}\int_{\p B_r(x)}v_n\,dl+C\epsilon\\ \nonumber
 &=\frac{1}{\pi r^2}\int_{B_r(x)}v_n(y)\,dy+C\epsilon,
\end{align}
where $C$ is a positive constant and $\beta_{\alpha,B_r(x)}$ equals to $4\pi(1+\alpha_n)$ if $-\tfrac{x_n}{\tau_n}\in B_r(x)$ and equals to $4\pi$ if $-\tfrac{x_n}{\tau_n}\notin \overline{B_r}(x)$.\\
At this point, by fixing $r=\frac{1}{2}|x|$, if $y\in B_r(x)$, then $$\sing\geq (1+\tfrac{3}{2}\tfrac{\tau_n}{|x_n|}|x|)^{2\alpha_n}. $$Therefore, by using Jensen's inequality, we have that,
\begin{eqnarray*}
 \nonumber
 e^{\dis v_n(x)}&\leq&\frac{C^\epsilon}{\pi r^2}\int_{B_r(x)}e^{\displaystyle v_n(y)}\,dy \hspace{3cm} \\
 &\leq& \frac{C^\epsilon}{\pi r^2 {a}}\frac{1}{(1+\tfrac{3}{2}\tfrac{\tau_n}{|x_n|}|x|)^{2\alpha_n}}\int_{B_r(x)}\sing\Ks(y) e^{\displaystyle v_n(y)}\,dy \\
 &\leq& C_1C^\epsilon\epsilon|x|^{-2}(1+\tfrac{3}{2}\tfrac{\tau_n}{|x_n|}|x|)^{-2\alpha_n}.
\end{eqnarray*}
This implies that
\begin{equation}\label{stimaalto3}
 v_n(x)\leq -2\log|x|-{2\alpha_n}\log(1+\tfrac{3}{2}\tfrac{\tau_n}{|x_n|}|x|)+\log(C_1C^\epsilon\epsilon),
\end{equation}
for $2r_{1,n}\leq|x|\leq \tfrac{l_n}{2}$.
Let us define the function
\[
 \overline{v}_n(r)=\frac{1}{2\pi r}\int_{\{|y|=r\}}v_n(y)\,dl(y)=\frac{1}{2\pi }\int_{\{|x|=1\}}v_n(rx)\,dl(x).
\]
For $r_{1,n}\leq r\leq l_n$, we have that
\begin{eqnarray*}\nonumber
 \frac{d}{dr}\overline{v}_n(r)&=&\frac{1}{2\pi r}\int_{\{|y|\leq r\}}\D v_n(y)\,dy \\
 &=&-\frac{1}{2\pi r}\int_{\{|y|\leq r\}}\sing\Ks(y) e^{\displaystyle v_n(y)}\,dy \\
 &\geq&-2(1+\tfrac{1}{\sqrt{\overline\sigma}})\frac{1}{r}
\end{eqnarray*}
and this implies
\begin{equation}\label{stimab3}
 \overline{v}_n(r)\geq-2(1+\tfrac{1}{\sqrt{{\overline\sigma}}})\log r + C_2,
\end{equation}
for $r_{1,n}\leq r\leq l_n$ and a suitable constant $C_2$.\\
Now, if $r\in[4r_{1,n},\frac{l_n}{4}]$ and $\frac{r}{2}\leq|x|\leq 2r$, then by using (\ref{stimaalto3}), we have
\begin{eqnarray*}\nonumber
v_n(x)+2\log r+2\alpha_n\log(1+\tfrac{3}{2}\tfrac{\tau_n}{|x_n|}r)
\leq \log(C)+\log(\epsilon)<0,
\end{eqnarray*}
provided $\epsilon$ is small enough. Let $r_n=r\in[4r_{1,n},\frac{l_n}{4}]$ and let
\[
\tilde{v}_n(x)=v_n(rx)+2\log r+2\alpha_n\log(1+\tfrac{3}{2}\tfrac{\tau_n}{|x_n|}r), \qquad\text{in}\,\, \tfrac{1}{2}\leq|x|\leq 2,
 \]
then, $\tilde{v}_n(x)\leq \log(C)+\log(\epsilon)$ and satisfies
\[-\D \tilde{v}_n(x)=\Big(\tfrac{\big|\tfrac{\tau_n}{|x_n|}rx+\tfrac{x_n}{|x_n|}\big|}{\big(1+\tfrac{3\tau_n}{2|x_n|}r\big)}\Big)^{2\alpha_n}\Ks(rx)e^{\dis\tilde{v}_n(x)}=:f_n(x),\qquad\text{in}\,\, \tfrac{1}{2}\leq |x|\leq 2.
\]
Let us analyze the following Dirichlet problem
\[
\left\{ \begin{array}{cc}
-\D w_n= f_n &  \text{in} \, B_2\backslash B_{1\diagup2},\\
    w_n=0 &\text{on} \,\, \p(B_2\backslash B_{1\diagup2}),
\end{array}
\right.
\]
then, by noticing that $f_n\geq0$ and by using the weak maximum principle, we deduce that $w_n\geq0$ in $\overline{B_2\backslash B_{1\diagup2}}$. Morever, by standard elliptic theory (\cite{GT}), we deduce that $\|w_n\|_{L^\infty}\leq \tilde{C}\epsilon$.
Indeed, let us define \[z_n=\tfrac{\big(\tfrac{\tau_n}{|x_n|}rx+\tfrac{x_n}{|x_n|}\big)}{\big(1+\tfrac{3\tau_n}{2|x_n|}r\big)}.\]
If $r\frac{\tau_n}{|x_n|}\leq\frac{1}{3}$ or $r\frac{\tau_n}{|x_n|}\geq 3$, then $|z_n|\geq c>0$, so that $\|f_n\|_{L^\infty(B_2\backslash B_{1\diagup2})}\leq {C}\epsilon$ and by standard elliptic estimates (\cite{GT}), we have that $\|w_n\|_{L^\infty}\leq \tilde{C}\epsilon$. \\ Otherwise, if $\frac{1}{3}<r\frac{\tau_n}{|x_n|}<3$, then we split the Dirichlet problem above in the following two problems:
\[
\left\{ \begin{array}{cc}
&-\D w_{1,n}= f_n\chi_{B_\delta(-\frac{1}{r}\frac{x_n}{\tau_n})}=g_n  \qquad \text{in} \,\, B_2\backslash B_{1\diagup2},\\
      &w_{1,n}=0 \qquad\qquad\qquad\qquad \text{on} \,\, \p(B_2\backslash B_{1\diagup2}),
\end{array}
\right.
\]
and
\[
\left\{ \begin{array}{cc}
&-\D w_{2,n}= f_n\chi_{B^c_\delta(-\frac{1}{r}\frac{x_n}{\tau_n})}=d_n  \qquad \text{in} \,\, B_2\backslash B_{1\diagup2},\\
      &w_{2,n}=0 \qquad\qquad\qquad\qquad \text{on} \,\, \p(B_2\backslash B_{1\diagup2}),
\end{array}
\right.
\]
for a certain $\delta>0$. We notice that, for $n$ large enough, we can always choose $q>1$ such that $\alpha_n q+1\geq d_\infty>0$. Thus,
\[
 \|g_n\|^q_{L^q(B_2\backslash B_{1\diagup2})}=\int_{B_2\backslash B_{1\diagup2}}f_n^{q}\,\chi_{B_\delta(-\frac{x_n}{r\tau_n})}=\int_{B_\delta(-\frac{|x_n|}{r\tau_n})\cap (B_2\backslash B_{1\diagup2})}f_n^q\leq C\delta^{2\alpha_n q+2}\epsilon^{q}.
\]
By standard elliptic estimates (\cite{GT}), $\|w_{1,n}\|_{L^\infty(B_2\backslash B_{1\diagup2})}\leq C_1\epsilon^{q}$, for a certain positive constant $C_1$. \\ On the other hand, $\|d_n\|_{L^\infty(B_2\backslash B_{1\diagup2})}= \|f_n\|_{L^\infty(B^c_\delta(-\frac{1}{r}\frac{x_n}{\tau_n})\cap (B_2\backslash B_{1\diagup2}))}\leq C\delta^{2\alpha_n}\epsilon$, then by standard elliptic estimates (\cite{GT}), $\|w_{2,n}\|_{L^\infty(B_2\backslash B_{1\diagup2})}\leq C_2\epsilon$, for a certain positive constant $C_2$.
Hence, also in this case, we have $\|w_n\|_{L^\infty}\leq \tilde{C}\epsilon$. \\

Thus, the function $h_n=w_n-\tilde{v}_n$ is harmonic and $\underset{\partial B_1}{\inf} h_n\geq 0$, whence, by using the Harnack principle, there exists $\gamma_0\in(0,1)$ such that,
\[
 \gamma_0\,\underset{\partial B_1}{\sup} \, h_n\leq \underset{\partial B_1}{\inf} h_n.
\]
Since $-\tilde{v}_n\leq h_n\leq \tilde{C}\epsilon-\tilde{v}_n$, we have that
\[
 \underset{\partial B_1}{\sup} \, (-\tilde{v}_n) \leq \gamma_0^{-1}\underset{\partial B_1}{\inf} (\tilde{C}\epsilon-\tilde{v}_n).
\]
At this point, for $|x|=1$, we deduce that
\begin{align*}
 -\tilde v_n(x)&\leq -\tfrac{\gamma_0^{-1}}{2\pi}\int_{|x|=1}\tilde v_n(x)\, dl + \tilde C\gamma_0^{-1}\epsilon \\
 &\leq -\tfrac{\gamma_0^{-1}}{2\pi}\int_{|x|=1} v_n(rx)\, dl -2\gamma_0^{-1}\log r-2\gamma_0^{-1}\alpha_n\log( 1+\tfrac{3}{2}\tfrac{\tau_n}{|x_n|}r) + \tilde C\gamma_0^{-1}\epsilon \\
 &= -\gamma_0^{-1}\overline v_n(r) -2\gamma_0^{-1}\log r-2\gamma_0^{-1}\alpha_n\log( 1+\tfrac{3}{2}\tfrac{\tau_n}{|x_n|}r) + \tilde C\gamma_0^{-1}\epsilon \\
 &\leq 2\gamma_0^{-1}(1+\tfrac{1}{\sqrt{{\overline\sigma}}})\log r - \gamma_0^{-1}C_2 \\
 &\hspace{2cm}-2\gamma_0^{-1}\log r-2\gamma_0^{-1}\alpha_n\log( 1+\tfrac{3}{2}\tfrac{\tau_n}{|x_n|}r) + \tilde C\gamma_0^{-1}\epsilon \\
 &= 2\tfrac{\gamma_0^{-1}}{\sqrt{{\overline\sigma}}}\log r -2\gamma_0^{-1}\alpha_n\log( 1+\tfrac{3}{2}\tfrac{\tau_n}{|x_n|}r) - \gamma_0^{-1}C_2 +\tilde C\gamma_0^{-1}\epsilon,
\end{align*}
where we used (\ref{stimab3}). Therefore, if $|x|=1$ and $4r_{1,n}\leq r \leq \tfrac{l_n}{4}$, we see that,
\[
 v_n(rx)\geq -2(1+\tfrac{\gamma_0^{-1}}{\sqrt{\overline\sigma}})\log r - 2\alpha_n(1-\gamma_0^{-1})\log ( 1+\tfrac{3}{2}\tfrac{\tau_n}{|x_n|}r) +\gamma_0^{-1}C_2 -\tilde C\gamma_0^{-1}\epsilon
\]
and we can find a suitable constant $C_3$ such that
\[
 v_n(rx)\geq -2(1+\tfrac{\gamma_0^{-1}}{\sqrt{\overline\sigma}})\log r - 2\alpha_n(1-\gamma_0^{-1})\log ( 1+\tfrac{3}{2}\tfrac{\tau_n}{|x_n|}r)+C_3,
\]
which is equivalent to say that, for $4r_{1,n}\leq |z|\leq \frac{l_n}{4}$,
\begin{equation}\label{stimabasso3}
v_n(z)\geq -2(1+\tfrac{\gamma_0^{-1}}{\sqrt{\overline\sigma}})\log |z| - 2\alpha_n(1-\gamma_0^{-1})\log ( 1+\tfrac{3}{2}\tfrac{\tau_n}{|x_n |}|z|)+C_3.
\end{equation}
Let us set $\delta:=(1+\tfrac{\gamma_0^{-1}}{\sqrt{\overline\sigma}})^{-1}$ (we notice that $\delta<1)$ and choose $R$ such that $1\leq R\leq \frac{|x_n|}{4\tau_n}$, then $1\leq R^{\delta}\leq\frac{|x_n|}{4\tau_n}$. \\
If $4r_{1,n}\leq R^\delta=|x|\leq\frac{|x_n|}{4\tau_n}$, then,  by using (\ref{stimaalto3}) and (\ref{stimabasso3}) we conclude that, for $R\leq|y|\leq \frac{|x_n|}{\tau_n}$,
\begin{eqnarray*} \nonumber
 v_n(x)&\geq& -2\log R - 2\alpha_n(1-\gamma_0^{-1})\log ( 1+\tfrac{3}{2}\tfrac{\tau_n}{|x_n|}R^\delta) +C_3 \\ \nonumber
 &\geq& -2\log |y| - 2\alpha_n(1-\gamma_0^{-1})\log ( 1+\tfrac{3}{2}\tfrac{\tau_n}{|x_n|}|y|) +C_3 \\ \nonumber
 &\geq& -2\log |y| - 2\alpha_n\log ( 1+\tfrac{3}{2}\tfrac{\tau_n}{|x_n|}|y|) +\log (C_1C^\epsilon \epsilon)  \\
 &\geq& v_n(y),
\end{eqnarray*}
where we used the fact that $2\alpha_n\gamma_0^{-1}\log ( 1+\tfrac{3}{2}\tfrac{\tau_n}{|x_n|}|y|)\geq 2\alpha_n\gamma_0^{-1}\log (\tfrac{5}{2})$, for every $R\leq|y|\leq\frac{|x_n|}{\tau_n}$ and we choose $\epsilon<<1$ such that $\log(C_1C^\epsilon\epsilon)\leq C_3 + 2\alpha_n\gamma_0^{-1}\log (\tfrac{5}{2})$. \\
By taking the supremum over all $R\leq|y|\leq \tfrac{|x_n|}{\tau_n}$, we have proved that
\begin{equation}
 v_n(x)\geq\underset{\{R\leq|y|\leq \frac{|x_n|}{\tau_n}\}}{\sup} v_n(y),\label{11dis3}
\end{equation}
for every $4r_{1,n}\leq R^\delta= |x|\leq\frac{|x_n|}{4\tau_n}$.

Moreover, if $ 1\leq R^\delta=|x|\leq4r_{1,n}$, we also have
\begin{equation}\label{2dis}
 v_n(x)\geq \underset{\{R\leq|y|\leq \frac{|x_n|}{\tau_n}\}}{\sup} v_n(y),
\end{equation}
Indeed, 
\[
 v_n(x)\geq  \underset{\{R\leq|y|\leq T\}}{\sup} v_n(y)
\]
for a fixed $T$, with $T^\delta\geq\max\{R,4r_{1,n}\}$, otherwise the fact that $v_n$ converges uniformly on compact sets to $w$, which is a strictly decreasing function with respect to $0$, provides a contradiction. Moreover,
\[
  \underset{\{R\leq|y|\leq T\}}{\sup} v_n(y)\geq  \underset{\{R\leq|z|\leq \frac{|x_n|}{\tau_n}\}}{\sup} v_n(z)
\]
If it is not the case, by choosing $\overline z$ such that $|\overline{z}|=T^\delta$ and by using \eqref{11dis3}, we conclude that
\[
 v_n(\overline z)\leq \underset{\{R\leq|y|\leq T\}}{\sup} v_n(y) < \underset{\{R\leq|z|\leq \frac{|x_n|}{\tau_n}\}}{\sup} v_n(z)= \underset{\{T\leq|z|\leq \frac{|x_n|}{\tau_n}\}}{\sup} v_n(z)\leq v_n(\overline z)
\]
which is a contradiction.

Hence we have the following decay estimate: for every $1\leq|x|\leq \frac{|x_n|}{4\tau_n}$
\begin{equation}\label{1dis3}
 v_n(x)\geq \underset{\{\sqrt[\delta]{|x|}\leq|y|\leq\frac{|x_n|}{\tau_n}\}}{\sup} v_n(y).
\end{equation}
\\
Let us set,
\[
m_n(r):=\underset{\{|y|=r\}}{\max} v_n(y), \qquad\qquad t_{0}:=m_n(\tfrac{|x_n|}{2\tau_n}).
\]
Now, recalling that $\underset{|y|\leq l_n/2}\max \, v_n(y)=v_n(0)=0$, let us define
\[
 d\mu_n:=\singx dx, \qquad \qquad d\sigma_n:=(\singx)^{\frac{1}{2}}dl,
\]
\[
 \O_{t}^n:=\{y\big| |y|\leq \tfrac{|x_n|}{2\tau_n} \,\,\text{and} \,\, v_n(y)>t \},\qquad \qquad \xi_n(t):=\int_{\Otn}d\mu_n =\mu_n(\Otn),
\]
for any $t\in(t_{0},0)$. %where $A_n=\underset{D_{n,l_n}}{inf} v_n$.We notice that $A_n\to -\infty$, when $n\to+\infty$,
It is easy to see that $\overline{\O_t^n}\subseteq\{|z|\leq \tfrac{|x_n|}{2\tau_n} \}$ and, also that $\underset{t\to 0^-}             {\lim}\xi_n(t)=0
$ and $\underset{t\to t_{0}^+}{\lim}\xi_n(t)=\int_{{\O^n_{t_{0}}}}d\mu_n=:\xi_n(t_{0})$ . Since $v_n\in W_{loc}^{2,p}(\R^2)$ for $p>2$, as a conseguence of the Generalized Sard's Lemma (see (\cite{F})), $\p\Otn$ is a $C^1$ closed curve for a.a $t\in(\ton,0)$ and since $v_n$ satisfies (\ref{vequation2}), $\p\Otn$ has null measure for a.a $t\in(\ton,0)$. Actually, it turns out that the level sets of $v_n$ have null measure for every $t\in(\ton,0)$, see Lemma \ref{levelsetlemma} in the 
appendix. This easily implies that $\xi_n
$ is a continuous, strictly decreasing function, which is almost everywhere differentiable in $t\in(\ton,0)$. Indeed, by using the Coarea formula (see (\cite{BW})), we have that, 
$$\tfrac{d\xi_n}{dt}(t)=-\int_{\p \Otn}\singx\tfrac{1}{|\nabla v_n|}d\sigma
$$
for almost any $t\in(\ton,0)$. \\
We introduce $\vns(|x|)$, the weighted symmetric decreasing rearrangement of $v_n$ with respect to the measure $d\mu_n$, defined in the following way. Fixed $r=|x|$, then
\begin{eqnarray*}
 \vns(r)&:=&\sup\{t\in (\ton,0): \xi_n(t)>\pi r^2\} \\
 &=& \sup\{t\in (\ton,0): s_n(t)>r\},
\end{eqnarray*}
where $s_n(t)=(\pi^{-1}\xi_n(t))^{\tfrac{1}{2}}$ and this implies that
\[
 (\O_t^{n})^\ast:=\{y\big| |y|\leq \tfrac{|x_n|}{2\tau_n} \,\,\text{and} \,\, \vns(y)>t \}=B_{s(t)}(0).
\]
Firstly, we notice that $s_n(t)\leq (\tfrac{1}{2})^{\alpha_n}\tfrac{|x_n|}{2\tau_n}$ for all $t\in(\ton,0)$, see \eqref{raggibuoni2}. \\
Moreover, the function $\vns:(0,s_n(t_{0}))\to (\ton,0)$ satisfies %is the inverse of $\xi_n$ and we have that
\[
 \underset{r\to s(\ton)^-}{\lim}\vns(r)=\ton, \qquad
  \qquad \underset{r\to 0^+}{\lim}\vns(r)=0.
\]

In particular $\vns$ is continuous and strictly decreasing. Also, we can conclude that $\vns$ is locally Lipschitz in $(0,s_n(\ton))$. Indeed, $\vns(r)=\eta_n^{\star}(\pi r^2)$, where $\eta_n^{\star}(r):=\sup\{t\in (\ton,0): \xi_n(t)>r\}$ is the continuous inverse of $\xi_n$. We notice that $\eta_n^{\star}(r)$ is locally Lipschitz continuous (see Lemma 4.1 in (\cite{b})). Thus, for a fixed $\bar{r}\in(0,s_n(\ton))$ and a small neighborhood $I_{\bar{r}}$ of $\bar{r}$ we have that, for every $r\in I_{\bar{r}}$,
\begin{eqnarray}
\label{proofvnr}
 |\vns(r)-\vns(\bar{r})|=|\eta_n^{\star}(\pi r^2)-\eta_n^{\star}(\pi \bar{r}^2)|\leq c|\pi r^2-\pi s_0^2|\leq C|r-s_0|,
\end{eqnarray}
where $C$ is a positive constant that depends on $\bar{r}$.  Thus $\vns(r)$ is differentiable almost everywhere in $(0,s_n(\ton))$. \\
Let us define
\begin{equation} \label{Fn3}
 F_n(r):=\int_{\Ovn} \Ks(y)e^{\displaystyle v_n(y)}\,d\mu_n,
\end{equation}
which is a locally Lipschitz function in $(0,s_n(\ton))$. Infact, by taking $\tilde r\in(0,s_n(\ton))$ and a small neighborhood of $\tilde r$, say $I_{\tilde r}$, then for every $r\in I_{\tilde r}$, we have that,
\[
 |F_n(r)-F_n(\tilde r)|\leq\Big|\underset{{\Ovn}\backslash\O_{v_{n}^{\ast}(\tilde r)}^n}\int \Ks(y)e^{\displaystyle v_n(y)}d\mu_n(y) \Big|\leq bC|\xi_n(\vns(r))-\xi_n({v_n^{\ast}(\tilde r)})|
 \]
 \begin{eqnarray}
 \label{proofFnr}
 \leq C|\xi_n(\eta_n^\star(\pi r^2))-\xi_n({\eta_n^{\star}(\pi \tilde r^2)})|=C|r^2-\tilde r^2|
 \leq C|r-\tilde r|,
\end{eqnarray}
where $C$ is a positive constant that depends on $\tilde{r}$ and we used the fact that $\xi_n$ is the inverse of $\eta_n^{\star}$. Thus $F_n$ is almost everywhere differentiable in $(0,s_n(\ton))$.\\
Now, from (\ref{1dis3}), the following inclusions hold: 
\begin{equation}
 \label{inclusion3}
 B_{R^{\delta^2}}(0)\subseteq \O_{m_n(R^\delta)}^n \subseteq B_{R}(0).
\end{equation}
for $1\leq R\leq\frac{|x_n|}{4\tau_n}$. \\
For the first inclusion, if we take $x\in B_{R^{\delta^2}}(0)$ and set $S=R^{\delta}$, \\then $1\leq S^{\delta}\leq \frac{|x_n|}{4\tau_n}$ and we can apply (\ref{1dis3}). Namely, for $x\in B_{S^{\delta}}(0)$, we have that,
\[
 v_n(x)>\underset{|z|=S^{\delta}}\min v_n(z)\geq  \underset{\{S\leq|y|\leq \frac{|x_n|}{\tau_n}\}}{\sup} v_n(y)\geq m_n(S),
\]
where the first inequality follows from the strong minimum principle. This implies that $x\in \O_{m_n(R^\delta)}^n$. \\
For the second inclusion, if $z\in\O_{m_n(R^\delta)}^n $, then, by using again \eqref{1dis3}, we have that,
\[
 v_n(z)>\underset{|y|=R^\delta}{\max}v_n(y)\geq \underset{R\leq|y|\leq \frac{|x_n|}{\tau_n}}{\max}v_n(y).
\]
At this point, if $z\notin B_R(0)$, then
\[
\underset{R\leq|y|\leq \frac{|x_n|}{\tau_n}}{\max}v_n(y)\geq v_n(z),\]
which is a contradiction. This implies that $z\in B_R(0)$. \\

We recall that, from \eqref{raggibuoni2}, we have that
\[
 (\tfrac{3}{2})^{2\alpha_n}\pi r^2\leq \mu_n(B_r(0))\leq (\tfrac{1}{2})^{2\alpha_n}\pi r^{2}.
\]
for $r\leq \tfrac{|x_n|}{4\tau_n}$. Thus, from (\ref{inclusion3}), we have that, for $1\leq R\leq \frac{|x_n|}{4\tau_n}$,
\[
 (\tfrac{3}{2})^{2\alpha_n}\pi R^{2\delta^2}\leq\mu_n(B_{R^{\delta^2}})\leq \mu_n(\O^n_{m_n(R^\delta)})\leq \mu_n(B_R(0))\leq(\tfrac{1}{2})^{2\alpha_n}\pi R^{2}.
\]
which, by noticing that $\mu_n(\O^n_{m_n(R^\delta)})=\xi_n(m_n(R^\delta))$, implies that
\begin{equation}
 \label{raggigenerali3}
 (\tfrac{3}{2})^{\alpha_n} R^{\delta^2}\leq \Big(\pi^{-1}\xi_n\big(m_n(R^\delta)\big)\Big)^{\frac{1}{2}}\leq (\tfrac{1}{2})^{\alpha_n} R
 \end{equation}
and, by using again (\ref{inclusion3}) and the fact that $\vns$ is strictly decreasing, then for $1\leq R\leq\frac{|x_n|}{4\tau_n}$,
\begin{equation}
 \label{integrali3}
 \int_{B_R(0)}\Ks(y)e^{\displaystyle v_n(y)}\, d\mu_n\geq F_n\Big(\big(\pi^{-1}\xi_n\big(m_n(R^\delta)\big)\big)^\frac{1}{2}\Big)\geq F_n\Big((\tfrac{3}{2})^{\alpha_n} R^{\delta^2}\Big).
\end{equation}
\\
Hence, by fixing $\hat{R}_n=\Big(\pi^{-1}\xi_n\big(m_n\big((\tfrac{|x_n|}{4\tau_n})^{\frac{\delta}{2}}\big)\big)\Big)^\frac{1}{2}$ and using (\ref{raggigenerali3}) with \\ $R=(\tfrac{|x_n|}{4\tau_n})^{\frac{1}{2}}$, we notice that
\begin{equation}
\label{raggixntn}
  (\tfrac{3}{2})^{\alpha_n} (\tfrac{|x_n|}{4\tau_n})^{\frac{\delta^2}{2}}\leq\hat{R}_n\leq (\tfrac{1}{2 })^{\alpha_n}(\tfrac{|x_n|}{4\tau_n})^{\frac{1}{2}}.
\end{equation}
Obviously, $\hat R_n\to +\infty$, as $n\to+\infty$ and $\vns(r)$ is defined for all \\$r\leq \hat R_n<s(t_0)$. Moreover, by using (\ref{1dis3}), we have that,
\[
 \vns(\hat{R}_n)=m_n((\tfrac{|x_n|}{4\tau_n})^{\frac{\delta}{2}})\geq\underset{(\frac{|x_n|}{4\tau_n})^{\frac{1}{2}}\leq|x|\leq\frac{|x_n|}{\tau_n}}\max  v_n(x)\geq v_n(-\tfrac{x_n}{\tau_n}).
\]
On the other hand, let us define
\[
 s_{1,n}:=\inf\{\bar r\in (0,s_n(t_0))| -\tfrac{x_n}{\tau_n}\in\O_{\vns(r)}\,\,\forall r>\bar r\}.
\]
Thus, we deduce that $\hat{R}_n\leq s_{1,n}$, which implies, from the previous calculations,
\begin{equation}\label{s1estimate}
 s_{1,n}\geq(\tfrac{3}{2})^{\alpha_n} (\tfrac{|x_n|}{4\tau_n})^{\tfrac{\delta^2}{2}}
\end{equation}
and, of course, $s_{1,n}\to +\infty$, when $n\to+\infty$. \\

Let us define, for a.a $r\in(0,\hat R_n),$
\[
 \hat{K}_n(r):=\frac{F'_n(r)}{2\pi re^{\dis\vns(r)}}
\]
and
\[
 {a}_n:=\underset{\{|y|\leq (\frac{l_n}{4})^{\frac{1}{2}}\}}{ess.\inf} \Ks(y), \qquad {b}_n:=\underset{\{|y|\leq (\frac{l_n}{4})^{\frac{1}{2}}\}}{ess.\sup} \Ks(y).
\]
Then, by using the fact that $\xi_n(\vns)=|B_r|$, we have that for almost any $r\in (0,\hat R_n)$,
\begin{eqnarray*}
 2\pi ra_ne^{\dis\vns(r)}&\leq&\underset{h\to 0^+}\lim \frac{a_ne^{\dis\vns(r+h)}}{h}\int_{\O_{v_n^{\ast}(r+h)}\backslash\Ovn}\,d\mu_n(y) \\
 &\leq& \underset{h\to 0^+}\lim\frac{1}{h} \int_{\O_{v_n^{\ast}(r+h)}\backslash\Ovn}\Ks e^{\displaystyle v_n}\,d\mu_n(y)= F'_n(r) \\
 &\leq&\underset{h\to 0^+}\lim \frac{b_ne^{\dis\vns(r)}}{h}\int_{\O_{v_n^{\ast}(r+h)}\backslash\Ovn}\,d\mu_n(y) \\
 &\leq& 2\pi rb_ne^{\dis\vns(r)},
\end{eqnarray*}
that is, for almost any $r\in (0,\hat R_n)$
\begin{equation}
 \label{Ka3}
 a_n\leq \hat{K}_n(r)\leq b_n.
\end{equation}

Let $I_1$ be the set of those $r\in(0,\hat R_n)$ where $\frac{d}{dr}\vns(r)$ does not exist and let $I_2$ be the set of those
$r\in(0,\hat R_n)$ where $\frac{d}{dr}\vns(r)=0$. If we denote $E=\vns(I_1\cup I_2)$, then $\mathcal{H}^1(E)=0$, since $\vns$ is locally Lipschitz. Furthermore, let $I_c=(\vns)^{-1}(E_c)$, where $E_c$ is the set of critical values of $v_n$. By the Generalized Sard's Lemma (\cite{F}), we have that $\mathcal{H}^1(E_c)=0$. Now, let $I$ be the set of those $r$ such that $\vns(r)=t$ for some $t\in (t_0,0)\backslash\{E\cup E_c\}$. \\
Hence, for any $r\in I$, we can apply the Cauchy-Schwartz inequality to deduce
\begin{eqnarray*}
 \nonumber
 \Big(\int_{\partial \O^n_{\vns (r)}}\,d\sigma_n\Big)^2&\leq& \Big(\int_{\partial \O^n_{\vns (r)}}\singx\frac{1}{|\nabla v_n|}\,dl\Big)\Big(\int_{\partial \O^n_{\vns (r)}}|\nabla v_n|\,dl\Big)\\
 &\leq& \Big(-\frac{d\vns(r)}{dr}\Big)^{-1}(2\pi r)\Big(\int_{\partial \O^n_{\vns (r)}}-\frac{\partial v_n}{\partial \nu}\,dl\Big),
\end{eqnarray*}
where we used the fact that $\vns(r)=\eta^{\star}_n(\pi r^2)$, $\eta^{\star}_n$ is the inverse of $\xi_n$ and $\nu$ is the exterior normal to $\partial \O^n_{\vns (r)}$. Moreover, we have that
\[
 \int_{\partial \O^n_{\vns (r)}}|\nabla v_n|\,dl=\int_{ \O^n_{\vns (r)}}\Ks e^{\displaystyle v_n}\,d\mu_n=F_n(r),
\]
which implies, for every $s\in I$
\begin{equation}
 \label{disdiff13}
 2\pi r F_n(r)\geq \Big(\int_{\partial \O^n_{\vns (r)}}\,d\sigma_n\Big)^2\Big(-\frac{d\vns(r)}{dr}\Big).
\end{equation}
Since $v_n$ is superharmonic and by using the maximum principle we can conclude that each connected component of $\O^n_{\vns (r)}$ is simply connected. Indeed, let us assume $\O^n_{\vns (r)}$ is multiply connected and, for simplicity, $\O^n_{t}=\tilde{\O}\backslash\O^n_{0,t}$, with $t=\vns(r)$, for a certain set $\tilde{\O}$. We can assume without loss of generality that $\O^n_{0,t}$ is connected. Then, $\partial\O^n_{0,t}\subset\{v_n=t\}$ and by superharmonicity of $v_n$ we have
\[
 \underset{\overline{\O}^n_{0,t}}{\min}\,v_n=\underset{\partial\O^n_{0,t}}{\min}\,v_n\geq \underset{\{v_n=t\}}{\min}\,v_n=t,
\]
that is, $v_n\geq t$ in $\O^n_{0,t}$ which is a contradiction. This implies that, if $\O^n_{\vns (r)}$ is smooth, then each connected component of $\O^n_{\vns (r)}$ is simply connected. \\
Hence, after a traslation and a rescaling, we apply Huber inequality (see Proposition \ref{Huber.App.prop} in the appendix) and we conclude that for $r\in I\subset (0,\hat R_n)$,
\[
 \Big(\int_{\partial \O^n_{\vns (r)}}\,d\sigma_n\Big)^2\geq \beta_{\alpha_n,r}\xi_n(\vns(r)),
\]
where $\beta_{\alpha_n,r}=\beta_{\alpha_n,\O^n_{\vns (r)}}$ is equal to $4\pi(1+\alpha_n)$ if $-\frac{{x_n}}{\tau_n}\in\O^n_{\vns (r)}$ or $4\pi$ if \\$-\frac{{x_n}}{\tau_n}\notin\overline{\O^n_{\vns (r)}}$.
However, we have shown that $\hat R_n\leq s_{1,n}$, then $\beta_{\alpha_n,r}=4\pi$ for every $r\in I$.  \\
By using this inequality in (\ref{disdiff13}), we conclude that
\begin{eqnarray}
 F_n(r)&\geq& 4\pi\xi_n(\vns(r))\Big(-\frac{d\vns(r)}{dr}\Big) \frac{1}{2\pi r } \nonumber \\
 &=& 2\pi r\Big(-\frac{d\vns(r)}{dr}\Big)\label{diff4},
\end{eqnarray}
for every $r\in I$. The last inequality is always true for $r\in I_2$, so it holds for $r\in I\cup I_2$. If we consider $I_3$ the set of those $r$ such that (\ref{diff4}) does not hold, then $I_3\subset (I_1\cup I_c)\backslash I_2$. On the other hand, since $\mathcal{H}^1(I_1)=0$ and $\mathcal{H}^1(E_c)=0$, there is no possibility that $I_3$ does have positive measure. This means that (\ref{diff4}) holds for a.a. $r\in(0,\hat R_n)$. \\
At this point, from (\ref{diff4}), we deduce that for almost any $r\in (0,\hat R_n)$,
\begin{align}\nonumber
 \frac{d}{dr}\Big(  \frac{rF'_n(r)}{\hat{K}_n(r)}\Big)=\frac{d}{dr}(2\pi r^2e^{\dis\vns(r)})&=4\pi r e^{\dis\vns(r)}+2\pi r^2e^{\dis\vns(r)}\frac{d}{dr}\vns(r) \\ \nonumber
 &\geq \frac{2F'_n(r)}{\hat{K}_n(r)}-\frac{F'_n(r)}{\hat{K}_n(r)}\frac{2}{4\pi}F_n(r) \\ \nonumber
 &= \frac{2F'_n(r)}{\hat{K}_n(r)}\Big(1-\frac{F_n(r)}{4\pi}\Big) \\
 &\geq \begin{cases}
      \frac{2F'_n(r)}{a_n}\Big(1-\frac{F_n(r)}{4\pi}\Big) \qquad \text{if}\, F_n(r)>4\pi, \\
      \frac{2F'_n(r)}{b_n}\Big(1-\frac{F_n(r)}{4\pi}\Big) \qquad \text{if}\, F_n(r)\leq4\pi.\label{diffeq231}
     \end{cases}
\end{align}

At this point, from the inclusions (\ref{inclusion3}) and taking $R=(\frac{|x_n|}{4\tau_n})^\frac{1}{2}$, we deduce that,
\begin{align}
 \int_{B_{R^{\delta^2}}(0)}\Ks e^{\displaystyle v_n}\,d\mu_n\leq \int_{\O^n_{m_n(R^\delta)}}\Ks e^{\displaystyle v_n}\,d\mu_n=F_n(\hat R_n)\leq\int_{B_R(0)}\Ks e^{\displaystyle v_n}\,d\mu_n. 
 \label{s0esiste}
\end{align}
This implies that $F_n(\hat R_n)\to 4\pi(1+\frac{1}{\sqrt{{\overline\sigma}}})$, as $n\to+\infty$. Hence, by the continuity and monotonicity of $F_n$, there exists, for $n$ sufficiently large, $s_0=s_{0,n}<\hat R_n$ such that $F_n(s_0)=4\pi$.

Now, we integrate (\ref{diffeq231}) for every $r\in (s_0,\hat R_n)$ and obtain

\begin{align}
 \label{secondcase2} \nonumber
 \frac{rF'_n(r)}{\hat{K}_n(r)}&\geq \frac{2}{b_n}\int_0^{s_0}F'_n\Big(1-\frac{F_n}{4\pi}\Big)ds+\frac{2}{a_n}\int_{s_0}^{r}F'_n\Big(1-\frac{F_n}{4\pi}\Big)ds \\ \nonumber
 &= \tfrac{2}{b_n}\big(F_n(s_0)-\tfrac{F_n(s_0)^2}{8\pi}\big)+\tfrac{2}{a_n}\big(F_n(r)-\tfrac{F_n(r)^2}{8\pi}- F_n(s_0)+\tfrac{F_n(s_0)^2}{8\pi}\big) \\ \nonumber
 &= -\tfrac{1}{4\pi a_n}F_n(r)^2+\tfrac{2}{a_n}F_n(r)-\tfrac{4\pi}{a_n}(1-\tfrac{a_n}{b_n}) \\  \nonumber
 &= -\frac{1}{4\pi a_n}\Big(F_n(r)-4\pi(1-\sqrt{\tfrac{a_n}{b_n}})\Big)\Big(F_n(r)-4\pi(1+\sqrt{\tfrac{a_n}{b_n}})\Big),
\end{align}
where we used the fact that $F_n(s_0)=4\pi$. Hence we deduce the following inequality
\begin{equation}
 \label{diffeq23}
 \frac{rF'_n(r)}{\hat{K}_n(r)}\geq -\frac{1}{4\pi a_n}\Big(F_n(r)-4\pi(1-\sqrt{\tfrac{a_n}{b_n}})\Big)\Big(F_n(r)-4\pi(1+\sqrt{\tfrac{a_n}{b_n}})\Big).
\end{equation}

Once we have established (\ref{diffeq23}) for every $s_0< r < \hat R_n$, we define
\[
 \tilde{R}_n:=\sup\Big\{r\leq \hat R_n\big| F_n(r)\leq 4\pi\left(1+\sqrt{\tfrac{a_n}{b_n}}\,\right)\Big\}.
\]

Then, from (\ref{diffeq23}), it follows that
\begin{equation}
 \label{dd2}
 \frac{F'_n(r)}{F_n(r)-4\pi(1-\sqrt{\tfrac{a_n}{b_n}})}+\frac{F'_n(r)}{4\pi(1+\sqrt{\tfrac{a_n}{b_n}})-F_n(r)}\geq 2\sqrt{\frac{a_n}{b_n}}\frac{1}{r},
\end{equation}
for $s_0\leq r \leq \tilde{R}_n$.
By integrating the previous inequality,
\[
 \int_{s_0}^r\frac{F'_n(r)}{F_n(r)-4\pi(1-\sqrt{\tfrac{a_n}{b_n}})}+\int_{s_0}^r\frac{F'_n(r)}{4\pi(1+\sqrt{\tfrac{a_n}{b_n}})-F_n(r)}\geq 2\sqrt{\frac{a_n}{b_n}}\log(\tfrac{r}{s_0})
\]
\[
 \log\Bigg(\bigg(\tfrac{F_n(r)-4\pi(1-\sqrt{\tfrac{a_n}{b_n}})}{F_n(s_0)-4\pi(1-\sqrt{\tfrac{a_n}{b_n}})}\bigg)\bigg(\tfrac{4\pi(1+\sqrt{\tfrac{a_n}{b_n}})-F_n(s_0)}{4\pi(1+\sqrt{\tfrac{a_n}{b_n}})-F_n(r)}\bigg)\Bigg)\geq 2\sqrt{\frac{a_n}{b_n}}\log(\tfrac{r}{s_0})
\]
\[
 \log\bigg(\tfrac{F_n(r)-4\pi(1-\sqrt{\tfrac{a_n}{b_n}})}{4\pi(1+\sqrt{\tfrac{a_n}{b_n}})-F_n(r)}\bigg)+\log\bigg(\tfrac{4\pi(1+\sqrt{\tfrac{a_n}{b_n}})-F_n(s_0)}{F_n(s_0)-4\pi(1-\sqrt{\tfrac{a_n}{b_n}})}\bigg)\geq 2\sqrt{\frac{a_n}{b_n}}\log(\tfrac{r}{s_0})
\]
and, using the fact that $F_n(s_0)=4\pi$, we have that,
\begin{equation}
 \label{loga3}
\log\Bigg(\frac{4\pi(1+\sqrt{\tfrac{a_n}{b_n}})-F_n(r)}{F_n(r)-4\pi(1-\sqrt{\tfrac{a_n}{b_n}})}\Bigg)\leq -2\sqrt{\frac{a_n}{b_n}}\log(\tfrac{r}{s_0}).
\end{equation}
Hence, from (\ref{loga3}) we deduce that,
\begin{eqnarray}
 \nonumber
 4\pi(1+\sqrt{\tfrac{a_n}{b_n}})&\leq& F_n(r)+ \Big(\frac{r}{s_0}\Big)^{-2\sqrt{\frac{a_n}{b_n}}}\Big(F_n(r)-4\pi(1-\sqrt{\tfrac{a_n}{b_n}})\Big) \\ \label{ka2}
 &\leq& F_n(r)+ \Big(\frac{r}{s_0}\Big)^{-2\sqrt{\frac{a_n}{b_n}}}F_n(r).
\end{eqnarray}

We notice that $F_n$ and $s_0$ are uniformly bounded: indeed,
\begin{align}
 F_n(r)\leq\int_{|y|\leq\frac{l_n}{2}}\Ks e^{\displaystyle v_n}\,d\mu_n \to \int_{\R^2}K_0 e^{\displaystyle w}\,dy=4\pi(1+{\tfrac{1}{\sqrt{\overline\sigma}}})\leq 8\pi,
 \label{boundedFS}
\end{align}
as $n\to+\infty$. This shows that $F_n$ is uniformly bounded. \\
Concerning $s_0$, let assume there exists a subsequence such that $s_0\to +\infty$. Let us consider the inclusions (\ref{inclusion3}) with $R=(r_{1,n})^{\delta^{-2}}\geq 1$, for $n$ large enough. Hence, it holds that $B_{r_{1,n}}(0)\subset\O^n_{m_n((r_{1,n})^{\delta^{-1}})}$. This implies
\[
 \int_{B_{r_{1,n}}(0)}\Ks e^{\displaystyle v_n}\,d\mu_n\leq \int_{\O_{m_n((r_{1,n})^{\delta^{-1}})}}\Ks e^{\displaystyle v_n}\,d\mu_n = \int_{\O_{\vns(T)}}\Ks e^{\displaystyle v_n}\,d\mu_n,
\]
where $T=(\pi^{-1}\xi_n(m_n((r_{1,n})^{\delta^{-1}})))^{\frac{1}{2}}$. Now, using the estimate (\ref{raggigenerali3}), we have that, for $n$ large enough,
\[
T\leq(\tfrac{1}{2})^{\alpha_n}(r_{1,n})^{\delta^{-2}}\leq3(\tfrac{1}{2})^{\alpha_\infty+1}(r_1)^{\delta^{-2}}=:T_{\infty}.
\]
Therefore, using the monotonicity of $\vns$ and taking $n$ large enough such that $s_0>T_{\infty}$, we deduce
\[
 \int_{B_{r_{1,n}}(0)}\Ks e^{\displaystyle v_n}\,d\mu_n\leq \int_{\O_{\vns(s_0)}}\Ks e^{\displaystyle v_n}\,d\mu_n=F_n(s_0)
\longrightarrow 4\pi,
\]
as $n\to+\infty$. On the other hand, this fact implies a contradiction because the first integral is converging to $4\pi(1+\frac{1}{\sqrt{{\overline\sigma}}})-\epsilon$, which is greater than $4\pi$, provided $\epsilon$ small enough.
\\

Hence, there exists a positive constant $C$ for which
\begin{equation}
 \label{meglio5}
 F_n(r)\geq4\pi(1+\sqrt{\tfrac{a_n}{b_n}})-Cr^{-2\sqrt{\frac{a_n}{b_n}}},
\end{equation}
for $s_0\leq r\leq \tilde{R}_n$. Obviously (\ref{meglio5}) holds for
$\tilde{R}_n\leq r\leq\hat{R}_n$. \\
By the definition of $a_n$ and $b_n$, we have that
\[
 \frac{b_n}{a_n}\leq \underset{|y|,|z|\leq(\frac{L_n}{4})^{\frac{1}{2}}}{\sup}\frac{\Ks(y)}{\Ks(z)}\leq {\overline\sigma} + \frac{B}{|\log (\sqrt{\rho}\tau_n^{\frac{1}{2}})|}\leq {\overline\sigma} +C_1|\log(\tau_n)|^{-1}.
\]
Hence,
\[
 \sqrt{\frac{a_n}{b_n}}\geq \frac{1}{\sqrt{{\overline\sigma}}}-C_2|\log(\tau_n)|^{-1},
\]
which, combined with (\ref{meglio5}), implies that,
\begin{equation}
 \label{meglio6}
 F_n(r)\geq4\pi(1+{\tfrac{1}{\sqrt{\overline\sigma}}})-Cr^{-{\tfrac{2}{\sqrt{\overline\sigma}}}}-C_2|\log(\tau_n)|^{-1},
\end{equation}
for ${s}_0\leq r\leq \hat R_n$. The latter estimate, together with (\ref{integrali3}), implies that
\begin{eqnarray}
 \nonumber
 \int_{B_R(0)}\Ks(y)e^{\displaystyle v_n(y)}\, d\mu_n &\geq& F_n\Big((\tfrac{3}{2})^{\alpha_n}R^{\delta^2}\Big) \\
 &\geq& 4\pi(1+{\tfrac{1}{\sqrt{\overline\sigma}}})-CR^{-\Big({2\tfrac{\delta^2}{\sqrt{\overline\sigma}}}\Big)}-C_2|\log(\tau_n)|^{-1},   \label{key3}
\end{eqnarray}
for $S_0\leq R\leq \big(\frac{|x_n|}{4\tau_n}\big)^{\frac{1}{2}}$, where $S_0=\max\{1,\big((\tfrac{2}{3})^{\alpha_n}{s}_0\big)^{\delta^{-2}}\}$. Hence,
\begin{eqnarray}
 \label{mma3}
 \int_{R\leq |y|\leq l_n}\Ks(y)e^{\displaystyle v_n(y)}\, d\mu_n \leq CR^{-\Big({2\tfrac{\delta^2}{\sqrt{\overline\sigma}}}\Big)}+C_2|\log(\tau_n)|^{-1},
\end{eqnarray}
for $S_0\leq R\leq \big(\frac{|x_n|}{4\tau_n}\big)^{\frac{1}{2}}$. \\
At this point, from (\ref{Rnscondition}) and for $n$ large enough, we have that $$\Rns\leq\tfrac{1}{4}(\tfrac{3}{2})^{\alpha_n} (\tfrac{|x_n|}{4\tau_n})^{\frac{\delta^2}{2}}.$$ Moreover, by (\ref{raggixntn}),
\[
 2\Rns<\big(\tfrac{|x_n|}{4\tau_n}\big)^{\frac{1}{2}} \hspace{2cm}\text{and}\hspace{2cm} 2\Rns<\hat R_n.
\]
Hence, by using (\ref{mma3}), we deduce for every $R\in [S_0,2\overline R_n ]$,
\begin{equation}
 \label{mma4}
 \int_{ R\leq|y|\leq l_n}\sing \Ks(y)e^{\displaystyle v_n(y)}\,dy\leq CR^{-\Big({2\tfrac{\delta^2}{\sqrt{\overline\sigma}}}\Big)}+C|\log(\tau_n)|^{-1}.
\end{equation}
In particular, (\ref{mma4}) holds also for $\tilde{S_0}\leq R\leq 2\overline R_n$, with $\tilde{S_0}=\max\{S_0,r_{1,n}\}$ and we have that $\{R\leq|y|\leq l_n\}\subset\{r_{1,n}\leq|y|\leq l_n\}$. At this point, we apply again the generalization of Suzuki's lemma (see the appendix, section \ref{Suzuki}) in the ball $B_R(x)$, with $|x|=2R$. We notice that $B_R(x)\subset\{R\leq|y|\leq l_n\}$ for $n$ sufficiently large. Hence, for $2\tilde{S_0}\leq|x|\leq \overline R_n$,
\begin{eqnarray*}
 \nonumber
 v_n(x)&\leq&\frac{1}{2\pi R}\int_{\p B_R(x)}v_n\,dl-2\log\Big(1-\tfrac{b}{2\beta}\int_{B_R(x)}|\tfrac{\tau_n}{x_n}x+\tfrac{x_n}{|x_n|}|^{2\alpha_n}e^{\displaystyle v_n}\,dx\Big)_+
 \\
 &\leq& \frac{1}{2\pi R}\int_{\p B_R(x)}v_n\,dl+C\epsilon \\
 &=&\frac{1}{\pi R^2}\int_{B_R(x)}v_n(y)\,dy+C\epsilon
\end{eqnarray*}
and by applying Jensen's inequality and (\ref{mma4}), we have that,

\begin{eqnarray}
 \nonumber
 e^{\dis v_n(x)}&\leq&\frac{C}{\pi R^2}\int_{B_R(x)}e^{\displaystyle v_n(y)}\,dy \hspace{3cm} \\ \nonumber
 &\leq& \frac{C}{\pi R^2 a}\int_{B_R(x)}\Ks(y) e^{\displaystyle v_n(y)}\,d\mu_n \\ \nonumber
 &=& \frac{C}{\pi a|x|^{2}}\int_{B_R(x)}\Ks(y) e^{\displaystyle v_n(y)}\,d\mu_n \\ \label{guaio4}
 &\leq& C_3\Big[|x|^{(-2-2\frac{\delta^{2}}{\sqrt{{\overline\sigma}}})}+|\log(\tau_n)|^{-1}|x|^{-2}\Big],
\end{eqnarray}
for $2\tilde{S_0}\leq|x|\leq \overline R_n$. Here we have used the fact that, if $y\in B_R(x)$, then $\sing\geq c_0>0$.

Using this improved estimate, we prove that the following integral
\begin{eqnarray}
 \label{logfinito}
 I=\underset{\{ 2\tilde S_0 \leq |x|\leq \overline{R}_n\}}\int \log(\tfrac{|x|}{\rho})\singx \Ks e^{\displaystyle v_n}\, dx&\leq& C
\end{eqnarray}
is bounded by an uniform constant. Indeed,
\begin{eqnarray*}
 I&\leq& C \underset{\{ 2\tilde S_0 \leq |x|\leq \overline{R}_n\}}\int \log(\tfrac{|x|}{\rho})\Big[|x|^{(-2-2\frac{\delta^{2}}{\sqrt{{\overline\sigma}}})}+|\log(\tau_n)|^{-1}|x|^{-2}\Big] \,dx \\
 &\leq& C\Big(\frac{(2\tilde{S_0})^{-2\frac{\delta^2}{\sqrt{{\overline\sigma}}}}}{2\tfrac{\delta^2}{\sqrt{{\overline\sigma}}}}\log(2\tilde{S_0})+\frac{(2\tilde{S_0})^{-2\frac{\delta^2}{\sqrt{{\overline\sigma}}}}}{(2\tfrac{\delta^2}{\sqrt{{\overline\sigma}}})^2}\Big)+C|\log(\tau_n)|^{-1}(\log(\tfrac{\overline R_n}{\rho}))^2\\
 &\leq& C,
\end{eqnarray*}
where we used (\ref{Rnscondition}).

\hfill \\
Now we want to refine the estimate (\ref{mma4}) when $R$ is greater or equal to $\overline R_n$. \\
From the definition of $l_n$ and by Fatou's lemma, we have that
\begin{equation}
 \label{massapiccola}
 \underset{n\to\infty}\lim\Big( \underset{\tfrac{\Rns}{4}\leq|y|\leq 4l_n}\int\sing\Ks(y)e^{\displaystyle v_n(y)}\,dy\Big)=0,
\end{equation}
where we recall that by definition of $\rho$ and $\overline{\rho}$, we have that $4l_n\leq4\tfrac{\rho}{\tau_n}\leq \tfrac{\overline\rho}{\tau_n}\leq\tfrac{1}{2\tau_n}$. \\
In the following, we will denote
\[
 \epsilon_n:=\underset{\tfrac{\Rns}{4}\leq|y|\leq 4l_n}\int\sing\Ks(y)e^{\displaystyle v_n(y)}\,dy.
\]

As we have observed above, $v_n$ is a solution of the inequality $-\D v_n\leq{b}\sing e^{\displaystyle v_n}$ in the set $\{|y|\leq \tfrac{1}{2\tau_n}\}$, so again we apply a generalization of a lemma by Suzuki (see the appendix, section \ref{Suzuki}, for further details). \\
Let $r\leq\frac{1}{2}|x|$, $\tfrac{\Rns}{2}\leq |x|\leq {2l_n}$ and $n$ large enough, such that,
\begin{align*}
 \nonumber
 v_n(x)&\leq\frac{1}{2\pi r}\int_{\p B_r(x)}v_n\,dl-2\log\Big(1-\tfrac{b}{2\beta_{\alpha,B_r(x)}}\int_{B_r(x)}\sing e^{\displaystyle v_n}\,dy\Big)_+ \\
 &\leq \frac{1}{2\pi r}\int_{\p B_r(x)}v_n\,dl+C\epsilon_n\\
 &\leq\frac{1}{\pi r^2}\int_{B_r(x)}v_n(y)\,dy+C\epsilon_n,
\end{align*}
where $C$ is a suitable constant and $\beta_{\alpha,B_r(x)}$ equals to $4\pi(1+\alpha_n)$  if $-\tfrac{x_n}{\tau_n}\in B_r(x)$ and equals to $4\pi$ if $-\tfrac{x_n}{\tau_n}\notin \overline{B_r(x)}$.  Moreover, we have used the fact that $\overline {B_r(x)}\subset\{y:\tfrac{\Rns}{4}\leq|y|\leq 4l_n\}$. At this point, by fixing $r=\frac{1}{2}|x|$, if $y\in B_r(x)$, then $\sing\geq (1+\tfrac{3}{2}\tfrac{\tau_n}{|x_n|}|x|)^{2\alpha_n}$. Therefore, by using Jensen's inequality, we have that,
\begin{eqnarray*}
 \nonumber
 e^{\dis v_n(x)}&\leq&\frac{C^{\epsilon_n}}{\pi r^2}\int_{B_r(x)}e^{\displaystyle v_n(y)}\,dy \hspace{3cm} \\
 &\leq& \frac{C^{\epsilon_n}}{\pi r^2 {a}}\frac{1}{(1+\tfrac{3}{2}\tfrac{\tau_n}{|x_n|}|x|)^{2\alpha_n}}\int_{B_r(x)}\sing\Ks(y) e^{\displaystyle v_n(y)}\,dy \\
 &\leq& C_1C^{\epsilon_n}{\epsilon_n}|x|^{-2}(1+\tfrac{3}{2}\tfrac{\tau_n}{|x_n|}|x|)^{-2\alpha_n}.
\end{eqnarray*}
This implies that
\begin{equation}\label{stimaalto22}
 v_n(x)\leq -2\log|x|-{2\alpha_n}\log(1+\tfrac{3}{2}\tfrac{\tau_n}{|x_n|}|x|)+\log(C_1C^{\epsilon_n}{\epsilon_n}),
\end{equation}
for $\tfrac{\Rns}{2}\leq|x|\leq 2l_n$.
Now, if $r\in[\Rns,l_n]$, then by using (\ref{stimaalto22}), we have
\begin{eqnarray*}\nonumber
v_n(x)+2\log r+2\alpha_n\log(1+\tfrac{3}{2}\tfrac{\tau_n}{|x_n|}r)
\leq \log(C_1)+\log(\epsilon_n)<0,
\end{eqnarray*}
for $\frac{r}{2}\leq|x|\leq 2r$, provided $n$ is large enough. \\
Let $r_n=r\in[\Rns,l_n]$ and let
\[
\tilde{v}_n(x)=v_n(rx)+2\log r+2\alpha_n\log(1+\tfrac{3}{2}\tfrac{\tau_n}{|x_n|}r), \qquad\text{with}\,\, \tfrac{1}{2}\leq|x|\leq 2,
 \]
then, $\tilde{v}_n(x)\leq \log(C_1)+\log(\epsilon_n)$ and satisfies
\[-\D \tilde{v}_n(x)=\Big(\tfrac{\big|\tfrac{\tau_n}{|x_n|}rx+\tfrac{x_n}{|x_n|}\big|}{\big(1+\tfrac{3\tau_n}{2|x_n|}r\big)}\Big)^{2\alpha_n}\Ks(rx)e^{\dis \tilde{v}_n(x)}=:f_n(x)\qquad\text{in}\,\, \frac{1}{2}\leq |x|\leq 2.
\]
Let us analyze the following Dirichlet problem
\[
\left\{ \begin{array}{cc}
-\D w_n= f_n &  \text{in} \, B_2\backslash B_{1\diagup2},\\
    w_n=0 &\text{on} \,\, \p(B_2\backslash B_{1\diagup2}),
\end{array}
\right.
\]
then, noticing that $f_n\geq0$ and using the weak maximum principle, $w_n\geq0$ in $\overline{B_2\backslash B_{1\diagup2}}$.
Morever, by using standard elliptic estimates (\cite{GT}), we deduce that $\|w_n\|_{L^\infty}\leq \tilde{C}\epsilon_n$.
Indeed, let us define, \[z_n=\tfrac{\big|\tfrac{\tau_n}{|x_n|}rx+\tfrac{x_n}{|x_n|}\big|}{\big(1+\tfrac{3\tau_n}{2|x_n|}r\big)}.\]
If $r\frac{\tau_n}{|x_n|}\leq\frac{1}{3}$ or $r\frac{\tau_n}{|x_n|}\geq 3$, then $|z_n|\geq\frac{1}{11}$  so that $\|f_n\|_{L^\infty(B_2\backslash B_{1\diagup2})}\leq {C}\epsilon_n$ and using standard elliptic estimates (\cite{GT}) then we have $\|w_n\|_{L^\infty}\leq \tilde{C}\epsilon_n$. \\ Otherwise, if $\frac{1}{3}<r\frac{\tau_n}{|x_n|}<3$, then we split the above Dirichlet problem in the following two problems
\[
\left\{ \begin{array}{cc}
&-\D w_{1,n}= f_n\chi_{B_\delta(-\frac{1}{r}\frac{x_n}{\tau_n})}=g_n  \qquad \text{in} \,\, B_2\backslash B_{1\diagup2},\\
      &w_{1,n}=0 \qquad\qquad\qquad\qquad \text{on} \,\, \p(B_2\backslash B_{1\diagup2}),
\end{array}
\right.
\]
\[
\left\{ \begin{array}{cc}
&-\D w_{2,n}= f_n\chi_{B^c_\delta(-\frac{1}{r}\frac{x_n}{\tau_n})}=d_n  \qquad \text{in} \,\, B_2\backslash B_{1\diagup2},\\
      &w_{2,n}=0 \qquad\qquad\qquad\qquad \text{on} \,\, \p(B_2\backslash B_{1\diagup2}),
\end{array}
\right.
\]
for a certain $\delta>0$. We notice that, for $n$ large enough, we can always choose $q>1$ such that $\alpha_n q+1\geq d_\infty>0$. Thus,
\[
 \|g_n\|^q_{L^q(B_2\backslash B_{1\diagup2})}=\int_{B_2\backslash B_{1\diagup2}}f_n^{q}\,\chi_{B_\delta(-\frac{x_n}{r\tau_n})}=\int_{B_\delta(-\frac{|x_n|}{r\tau_n})\cap (B_2\backslash B_{1\diagup2})}f_n^q\leq C\delta^{2(\alpha_n q+1)}\epsilon_n^{q}.
\]
By standard elliptic estimates (\cite{GT}), $\|w_{1,n}\|_{L^\infty(B_2\backslash B_{1\diagup2})}\leq C_1\epsilon^{q}$, for a certain positive constant $C_1$. \\
On the other hand, $\|d_n\|_{L^\infty(B_2\backslash B_{1\diagup2})}= \|f_n\|_{L^\infty(B^c_\delta(-\frac{1}{r}\frac{x_n}{\tau_n})\cap (B_2\backslash B_{1\diagup2}))}\leq C\delta^{2\alpha_n}\epsilon_n$, then by standard elliptic estimates (\cite{GT}), $\|w_{2,n}\|_{L^\infty(B_2\backslash B_{1\diagup2})}\leq C_2\epsilon_n$, for a certain positive constant $C_2$.
Hence, also in this case, we have $\|w_n\|_{L^\infty}\leq \tilde{C}\epsilon_n$. \\

Then, the function $h_n=w_n-\tilde{v}_n$ is harmonic and $\underset{\partial B_1}{\inf} h_n\geq 0$, whence by using the Harnack principle, there exists $\gamma_1\in(0,1)$ such that
\[
 \underset{\partial B_1}{\gamma_1\sup} \, h_n\leq \underset{\partial B_1}{\inf} h_n.
\]
Since $-\tilde{v}_n\leq h_n\leq \tilde{C}\epsilon_n-\tilde{v}_n$, we have that
\[
 \underset{\partial B_1}{\sup} \, (-\tilde{v}_n) \leq \gamma_1^{-1}\underset{\partial B_1}{\inf} (\tilde{C}\epsilon_n-\tilde{v}_n),
\]
that is,
\begin{eqnarray*} \underset{\partial B_1}{\sup} \, (-v_n(rx)-2\log r-2\alpha_n\log(1+\tfrac{3}{2}\tfrac{\tau_n}{|x_n|}r))\leq \hspace{5cm}\\ \hspace{3cm} \leq\tilde C\gamma_1^{-1}\epsilon_n+\gamma_1^{-1}\underset{\partial B_1}{\inf} (-v_n(rx)-2\log r-2\alpha_n\log(1+\tfrac{3}{2}\tfrac{\tau_n}{|x_n|}r))
\end{eqnarray*}
\[
 -\underset{\partial B_r}{\inf}\, v_n\leq C\gamma_1^{-1}\epsilon_n-\gamma_1^{-1}\underset{\partial B_r}{\sup}\, v_n+2(1-\gamma_1^{-1})\log r+2\alpha_n(1-\gamma_1^{-1})\log(1+\tfrac{3}{2}\tfrac{\tau_n}{|x_n|}r)
\]
which implies that, for $r\in[\Rns,l_n]$, it holds
\begin{equation}
\label{xe}
\underset{\partial B_r}{\sup}\, v_n\leq C\epsilon_n+\gamma_1\underset{\partial B_r}{\inf}\, v_n-2(1-\gamma_1)\log r-2\alpha_n(1-\gamma_1)\log(1+\tfrac{3}{2}\tfrac{\tau_n}{|x_n|}r),
\end{equation}

Next, let us analyze the following three cases:\\
$$r\in[\Rns,\tfrac{|x_n|}{4\tau_n}],\,\, r\in[\tfrac{|x_n|}{4\tau_n},4\tfrac{|x_n|}{\tau_n}]\,\,\text{ and}\,\, r\in[4\tfrac{|x_n|}{\tau_n},l_n],$$\\
with the aim in particular of showing that for $r\in[\Rns,l_n]$,
\begin{equation}
 \label{pocamassa}
 \underset{\{ r\leq |x|\leq l_n\}}\int \singx \Ks e^{\displaystyle v_n}\, dx\leq C r^{-\beta},
\end{equation}
with $\beta=\min\{\tfrac{2\gamma_1}{\sqrt{\sigma_1}},\tfrac{2\gamma_1}{\sqrt{\sigma_2}}-2\alpha_n,2(1+\alpha_n)\tfrac{\gamma_1}{C_1}\}>0$, for some $\sigma_1,\sigma_2>\overline\sigma$ and $C_1>0$,
and
\begin{equation}
\label{loginfinito}
 \underset{\{ \overline{R}_n\leq |x|\leq l_n\}}\int \log(\tfrac{|x|}{\rho})\singx \Ks e^{\displaystyle v_n}\, dx\leq C.
\end{equation}
\hfill \\
Concerning the case $r\in[\Rns,\tfrac{|x_n|}{4\tau_n}]$, we apply the regular ``$\sup+\inf$'' result by Chen-Lin (\cite{CL}), which implies
\begin{equation}
\label{ChenLin}
 \underset{\partial B_r}{\inf}\, v_n\leq C-2(1+\tfrac{1}{\sqrt{\tilde\sigma}})\log r,
\end{equation}
for some $\tilde\sigma\geq{\overline\sigma}$. Indeed, we consider $r\in[\tfrac{\Rns}{2},\tfrac{|x_n|}{2\tau_n}]$ and define
\[
 \tilde{w}_n(z)=v_n(rz)+2\log r, \,\,\,\,\,\,\,\,\,\, |z|\leq 1,
\]
which satisfies
\[
 -\D \tilde{w}_n(z)=|\tfrac{\tau_n}{|x_n|}rz+\tfrac{x_n}{|x_n|}|^{2\alpha_n}\Ks(rz)e^{\dis\tilde{w}_n(z)}=\tilde{K}_n(z)e^{\tilde{w}_n(z)}, \,\,\,\,\,\,\,\,\,\, |z|\leq 1,
\]
where $\tilde{K}_n(z)\in [\tilde{a},\tilde{b}]$, for $0<\tilde{a}<\tilde{b}<+\infty$ and
\begin{eqnarray}
\nonumber
 \frac{\tilde{K}_n(x)}{\tilde{K}_n(y)}&\leq& \frac{|\tfrac{\tau_n}{|x_n|}rx+\tfrac{x_n}{|x_n|}|^{2\alpha_n}}{|\tfrac{\tau_n}{|x_n|}ry+\tfrac{x_n}{|x_n|}|^{2\alpha_n}}\frac{K_n(\tau_nrx+x_n)}{K_n(\tau_nry+x_n)} \\ \nonumber
 &\leq& \bigg|\frac{\tfrac{\tau_n}{|x_n|}rx+\tfrac{x_n}{|x_n|}}{\tfrac{\tau_n}{|x_n|}ry+\tfrac{x_n}{|x_n|}}\bigg|^{2\alpha_n}\Big[{\overline\sigma}+\frac{B}{|\log (\tau_nr|x-y|)|}\Big] \\  \nonumber
 &\leq&\bigg(1+\tfrac{\tau_n}{|x_n|}r\tfrac{|x|+|y|}{1-\tfrac{\tau_n}{|x_n|}r|x|}\bigg)^{2|\alpha_n|}\Big[{\overline\sigma}+\frac{B}{|\log (\tau_nr|x-y|)|}\Big] \\
 &\leq&\sigma_1+\frac{ B_1}{|\log (\tau_nr|x-y|)|}\leq\sigma_1+\frac{ B_1}{|\log (\rho|x-y|)|}\nonumber
 \\
 &\leq&\sigma_1+\frac{B_2}{|\log |x-y||},\nonumber
\end{eqnarray}
forall $|x|,|y|\leq 1$ and with ${\sigma_1}>{\overline\sigma}$. We recall that $\rho\leq \tfrac{1}{4}$.
Then we can apply the sharp ``$\sup+ \inf$'' inequality (\cite{CL}) with compact set $K=\{0\}$ and open set $\O=B_{1}$ and deduce that
\[
 \tfrac{1}{\sqrt{\sigma_1}}\tilde{w}_n(0)+\underset{ B_1}{\inf}\, \tilde{w}_n\leq C,
\]
which, by definition of $\tilde{w}_n$, implies (\ref{ChenLin}). At this, point, by using (\ref{ChenLin}) together with (\ref{xe}), we have that, for $r\in[\Rns,\tfrac{|x_n|}{4\tau_n}]$,
\[
 \underset{\partial B_r}{\sup}\, v_n\leq C\epsilon_n+C\gamma_1-2(1+\tfrac{\gamma_1}{\sqrt{\sigma_1}})\log r-2\alpha_n(1-\gamma_1)\log(1+\tfrac{3}{2}\tfrac{\tau_n}{|x_n|}r).
\]
Hence, for $|x|\in[\Rns,\tfrac{|x_n|}{4\tau_n}]$, we have that,
\begin{equation}
 \label{stimaforse}
 e^{\dis v_n(x)}\leq \tfrac{C|x|^{-2(1+\tfrac{\gamma_1}{\sqrt{\sigma_1}})}}{(1+\tfrac{3}{2}\tfrac{\tau_n}{|x_n|}|x|)^{2\alpha_n(1-\gamma_1)}}
\end{equation}
and this implies that, for $r\in[\Rns,\tfrac{|x_n|}{4\tau_n}]$,
\begin{equation}
 \label{massetta1}
 \underset{\{ r\leq |x|\leq \frac{|x_n|}{4\tau_n}\}}\int \singx \Ks e^{\displaystyle v_n}\, dx\leq C r^{-\tfrac{2\gamma_1}{\sqrt{\sigma_1}}}
\end{equation}
and
\begin{equation}
 \label{loggino1}
 \underset{\{ \overline R_n\leq |x|\leq \frac{|x_n|}{4\tau_n}\}}\int \log(\tfrac{|x|}{\rho})\singx \Ks e^{\displaystyle v_n}\, dx\leq C.
 \end{equation}
\\

Concerning the case $r\in[\tfrac{|x_n|}{4\tau_n},4\tfrac{|x_n|}{\tau_n}]$, let us define
\[
 \tilde{w}_n(z)=v_n(rz)+2\log r +2\alpha_n\log(1+\tfrac{3}{2}\tfrac{\tau_n}{|x_n|}r),
\]
for $z\in B_1(0) $. Without loss of generality, we can assume that $-\frac{x_n}{r\tau_n}=(-\frac{|x_n|}{r\tau_n},0)\in\R^2$. Thus,
the singular point satisfies
$$-\frac{x_n}{r\tau_n}\notin\overline {B_1(0)\setminus B_{\frac{7}{8}}((-1,0))}.$$
Therefore, we have that,
\[
-\D \tilde{w}_n(x)=\tilde{K}_n(z)e^{\dis\tilde{w}_n(x)}\qquad\text{in}\,\, B_1(0)\setminus B_{\frac{7}{8}}((-1,0)),
\]
where
\[
 \tilde{K}_n(z)=\Big(\tfrac{\big|\tfrac{\tau_n}{|x_n|}rx+\tfrac{x_n}{|x_n|}\big|}{\big(1+\tfrac{3\tau_n}{2|x_n|}r\big)}\Big)^{2\alpha_n}\Ks(rx).
\]

The function $\tilde w_n$ satisfies the hypothesis of the regular ``$\sup+\inf$'' inequality (\cite{CL}). Indeed, $\tilde{K}_n(z)\in[\tilde a,\tilde{b}]$, for some $0<\tilde a<\tilde b<+\infty$, and
\begin{eqnarray}
\nonumber
 \frac{\tilde{K}_n(x)}{\tilde{K}_n(y)}&\leq& \frac{|\tfrac{\tau_n}{|x_n|}rx+\tfrac{x_n}{|x_n|}|^{2\alpha_n}}{|\tfrac{\tau_n}{|x_n|}ry+\tfrac{x_n}{|x_n|}|^{2\alpha_n}}\frac{K_n(\tau_nrx+x_n)}{K_n(\tau_nry+x_n)} \\ \nonumber
 &\leq& \bigg|\frac{\tfrac{\tau_n}{|x_n|}rx+\tfrac{x_n}{|x_n|}}{\tfrac{\tau_n}{|x_n|}ry+\tfrac{x_n}{|x_n|}}\bigg|^{2\alpha_n}\Big[{\overline\sigma}+\frac{B}{|\log (\tau_nr|x-y|)|}\Big] \\  \label{Ktilde2}
 &\leq&\sigma_2+\frac{ B_1}{|\log (\tau_nr|x-y|)|}\leq\sigma_2+\frac{ B_1}{|\log (\rho|x-y|)|}\nonumber
 \\&\leq&\sigma_2+\frac{ B_2}{|\log (|x-y|)|},\nonumber
\end{eqnarray}
forall $|x|,|y|\in B_1(0)\setminus B_{\frac{7}{8}}((-1,0))$ and with ${\sigma_2}>{\overline\sigma}$.
\\
Then, by using the sharp sup+inf inequality (\cite{CL}) with $K=\{0\}$ and \\$\Omega=\{B_1(0)\setminus B_{\frac{7}{8}}((-1,0))\}$, we have that
\[
 \tfrac{1}{\sqrt\sigma_2}\underset{K}\sup \, \tilde w_n +\underset{B_1}\inf \, \tilde w_n\leq \tfrac{1}{\sqrt\sigma_2}\underset{K}\sup \, \tilde w_n +\underset{\Omega}\inf \, \tilde w_n \leq C,
\]
which implies that
\[
 \underset{\partial B_r}\inf \, \tilde v_n \leq C -2(1+\tfrac{1}{\sqrt{\sigma_2}})\log r-2\alpha_n(1+\tfrac{1}{\sqrt{\sigma_2}})\log(1+\tfrac{3}{2}\tfrac{\tau_n}{|x_n|}r).
\]
By using the previous inequality and (\ref{xe}), we have that,
\[
 \underset{\partial B_r}{\sup}\, v_n\leq C\epsilon_n+\gamma_1 C-2(1+\tfrac{\gamma_1}{\sqrt{\sigma_2}})\log r-2\alpha_n(1+\tfrac{\gamma_1}{\sqrt{\sigma_2}})\log(1+\tfrac{3}{2}\tfrac{\tau_n}{|x_n|}r).
\]

Hence, this implies that, for $|x|\in[\tfrac{|x_n|}{4\tau_n},4\tfrac{|x_n|}{\tau_n}]$,
\begin{equation}
 \label{stimaforse2}
 e^{\dis v_n(x)}\leq \tfrac{C|x|^{-2(1+\frac{\gamma_1}{\sqrt{\sigma_2}})}}{(1+\tfrac{3}{2}\tfrac{\tau_n}{|x_n|}|x|)^{2\alpha_n(1+\frac{\gamma_1}{\sqrt{\sigma_2}})}}.
\end{equation}
Therefore, we have that, for $r\in[\tfrac{|x_n|}{4\tau_n},4\tfrac{|x_n|}{\tau_n}]$,
\begin{equation}
 \label{massetta2}
 \underset{\{ r\leq |x|\leq 4\tfrac{|x_n|}{\tau_n}\}}\int \singx \Ks e^{\displaystyle v_n}\, dx\leq C r^{2\alpha_n-\tfrac{2\gamma_1}{\sqrt{\sigma_2}}}
\end{equation}
and
\begin{equation}
 \label{loggino2}
 \underset{\{ \tfrac{|x_n|}{4\tau_n}\leq |x|\leq 4\tfrac{|x_n|}{\tau_n}\}}\int \log(\tfrac{|x|}{\rho})\singx \Ks e^{\displaystyle v_n}\, dx\leq C.
\end{equation}

Indeed, let us prove (\ref{massetta2}) and split the set $\{ r\leq |x|\leq 4\tfrac{|x_n|}{\tau_n}\}= A\cup B$, with $A=\{ r\leq |x|\leq 4\tfrac{|x_n|}{\tau_n}\}\cap\{r\leq|x|\leq|\frac{\tau_n}{|x_n|}x+\tfrac{x_n}{|x_n|}|\}$ and  $B=\{ r\leq |x|\leq 4\tfrac{|x_n|}{\tau_n}\}\cap\{r\leq|\frac{\tau_n}{|x_n|}x+\tfrac{x_n}{|x_n|}|\leq|x|\}$. Thus,
\begin{eqnarray*}
 \underset{\{ r\leq |x|\leq 4\tfrac{|x_n|}{\tau_n}\}}\int \singx \Ks e^{\displaystyle v_n}\, dx&=&\underset{A}\int \singx \Ks e^{\displaystyle v_n}\, dx  \\
 &+& \underset{B}\int \singx \Ks e^{\displaystyle v_n}\, dx.
\end{eqnarray*}
Now, by using (\ref{stimaforse2}), we have that,
\[
 \underset{A}\int \singx \Ks e^{\displaystyle v_n}\, dx\leq C\underset{A}\int |x|^{2\alpha_n-2(1+\tfrac{\gamma_1}{\sqrt{\sigma_2}})}\, dx \leq C r^{2\alpha_n-\tfrac{2\gamma_1}{\sqrt{\sigma_2}}},
\]
while, by using (\ref{stimaforse2}) and by recalling that if $x\in B$, then $\singx\leq r^{2\alpha_n}$, we have that,
\begin{align*}
 \underset{B}\int \singx \Ks e^{\displaystyle v_n}\, dx&\leq Cr^{2\alpha_n}\underset{\{ r\leq |x|\leq 4\tfrac{|x_n|}{\tau_n}\}}\int |x|^{-2(1+\frac{\gamma_1}{\sqrt{\sigma_2}})}\, dx \\ &\leq C r^{2\alpha_n-\tfrac{2\gamma_1}{\sqrt{\sigma_2}}}.
\end{align*}

Hence (\ref{massetta2}) easily follows.\\
Concerning (\ref{loggino2}), we argue as above and define $\{ \tfrac{|x_n|}{4\tau_n}\leq |x|\leq 4\tfrac{|x_n|}{\tau_n}\}= A\cup B$, with $A=\{  \tfrac{|x_n|}{4\tau_n}\leq |x|\leq 4\tfrac{|x_n|}{\tau_n}\}\cap\{ \tfrac{|x_n|}{4\tau_n}\leq|x|\leq|\frac{\tau_n}{|x_n|}x+\tfrac{x_n}{|x_n|}|\}$ and $B=\{\tfrac{|x_n|}{4\tau_n}\leq |x|\leq 4\tfrac{|x_n|}{\tau_n}\}\cap\{ \tfrac{|x_n|}{4\tau_n}\leq|\frac{\tau_n}{|x_n|}x+\tfrac{x_n}{|x_n|}|\leq|x|\}$. Thus,
\begin{align*}
 \underset{\{  \tfrac{|x_n|}{4\tau_n}\leq |x|\leq 4\tfrac{|x_n|}{\tau_n}\}}\int \log (\tfrac{|x|}{\rho})\singx \Ks e^{\displaystyle v_n}\, dx= J_1+J_2,
 \end{align*}
 with
 \begin{align*}J_1=\underset{A}\int \log(\tfrac{|x|}{\rho}) \singx \Ks e^{\displaystyle v_n}\, dx, \\  J_2= \underset{B}\int \log(\tfrac{|x|}{\rho}) \singx \Ks e^{\displaystyle v_n}\, dx.
\end{align*}
Now, by using (\ref{stimaforse2}), we have that,
\[
 J_1\leq C\underset{A}\int \log (\tfrac{|x|}{\rho})|x|^{2\alpha_n-2(1+\tfrac{\gamma_1}{\sqrt{\sigma_2}})}\, dx \leq C,
\]
while, by using (\ref{stimaforse2}), we have that,
\begin{align*}
 J_2&\leq C\underset{B}\int \log (\tfrac{|x|}{\rho}) |\tfrac{\tau_n}{|x_n|}x+\tfrac{x_n}{|x_n|}|^{2\alpha_n-2(1+\tfrac{\gamma_1}{\sqrt{\sigma_2}})}\, dx \\ &\leq C (\tfrac{|x_n|}{4\tau_n})^{2\alpha_n-2-\tfrac{2\gamma_1}{\sqrt{\sigma_2}}}\underset{B}\int \log (\tfrac{|x|}{\rho})\,dx\leq C.
\end{align*}

Hence (\ref{loggino2}) easily follows.\\

Concerning the case $r\in[4\tfrac{|x_n|}{\tau_n},l_n]$, let us define
\[
 \tilde{w}_n(z)=v_n(\tfrac{r}{2}z-\tfrac{x_n}{\tau_n})+2\log r +2\alpha_n\log(\tfrac{\tau_n}{|x_n|}r)-2(1+\alpha_n)\log 2,
\]
for $z\in B_1(0)$. Thus,
\[
-\D \tilde{w}_n(z)=|z|^{2\alpha_n}\Ks(\tfrac{r}{2}z-\tfrac{x_n}{\tau_n})e^{\dis\tilde{w}_n(z)}\qquad\text{in}\,\, |z|< 1,
\]
which satisfies the hypothesis of the $``\sup+C\inf''$ inequality proved in (\cite{b0}). Then, if $A=\{|z|\leq \tfrac{1}{2}\}$ and $\Omega=
\{|x|< 1\}$, we have,
\begin{equation}
\label{sup+cinf}
 \underset{A}\sup \, \tilde w_n +C_1\underset{B_1}\inf \, \tilde w_n\leq C.
\end{equation}
At this point, we notice that $z_n=\tfrac{2x_n}{r\tau_n}$, so $|z_n|\in (0,\tfrac{1}{2}]$. This implies that $\underset{z\in A}\sup \,  v_n(\tfrac{r}{2}z-\tfrac{x_n}{\tau_n})=v_n(\tfrac{r}{2}z_n-\tfrac{x_n}{\tau_n})=v_n(0)=0$
and
\[
 \underset{z\in A}\sup \, \tilde w_n(z)=2\log r +2\alpha_n\log(\tfrac{\tau_n}{|x_n|}r)-2(1+\alpha_n)\log 2.
\]
On the other hand, if $y\in B_{\tfrac{r}{2}}(-\tfrac{x_n}{\tau_n})$, then $y\in B_r(0)$, which implies that
\[
 \underset{\partial B_r} \inf\,v_n(y)=\underset{B_r(0)} \inf\,v_n(y)\leq \underset{ B_{r/2}(-\frac{x_n}{\tau_n})}\inf\, v_n(y)=\underset{z\in B_1}\inf \, v_n(\tfrac{r}{2}z-\tfrac{x_n}{\tau_n})
\]
and
\[
 \underset{B_1}\inf \, \tilde w_n\geq \underset{\partial B_r} \inf\,v_n(y)+2\log r +2\alpha_n\log(\tfrac{\tau_n}{|x_n|}r)-2(1+\alpha_n)\log 2.
\]

Hence, from (\ref{sup+cinf}) and the previous calculations, we have that,
\[
 \underset{\partial B_r}\inf \, v_n \leq \tfrac{C}{C_1}+2(1+\alpha_n)(1+\tfrac{1}{C_1})\log 2 -2(1+\tfrac{1}{C_1})\log r-2\alpha_n(1+\tfrac{1}{C_1})\log(\tfrac{\tau_n}{|x_n|}r).
\]
Using the latter inequality and (\ref{xe}), we have that,
\begin{eqnarray*}
 \underset{\partial B_r}{\sup}\, v_n&\leq& C-2(1+\tfrac{\gamma_1}{C_1})\log r-2\alpha_n\log\big((1+\tfrac{3}{2}\tfrac{\tau_n}{|x_n|}r)^{1-\gamma_1}(\tfrac{\tau_n}{|x_n|}r)^{(1+\frac{1}{C_1})\gamma_1}\big) \\
 &\leq& C-2(1+\tfrac{\gamma_1}{C_1})\log r-2\alpha_n\log\big((1+\tfrac{3}{2}\tfrac{\tau_n}{|x_n|}r)(\tfrac{\tau_n}{|x_n|}r)^{\frac{\gamma_1}{C_1}}\big),
\end{eqnarray*}
where we used the fact that $-2\alpha_n\log\big((1+\tfrac{3}{2}\tfrac{\tau_n}{|x_n|}r)^{-\gamma_1}(\tfrac{\tau_n}{|x_n|}r)^{\gamma_1}\big)$
is uniformly bounded. Hence, for $|x|\in[4\tfrac{|x_n|}{\tau_n},l_n]$,
\begin{equation}
 \label{stimaforse3}
 e^{\dis v_n(x)}\leq \tfrac{C|x|^{-2(1+\frac{\gamma_1}{C_1})}}{\big((1+\tfrac{3}{2}\tfrac{\tau_n}{|x_n|}|x|)(\tfrac{\tau_n}{|x_n|}|x|)^{\frac{\gamma_1}{C_1}}\big)^{2\alpha_n}}.
\end{equation}

Therefore, for $r\in[4\tfrac{|x_n|}{\tau_n},l_n]$, we have that,
\begin{equation}
 \label{massetta3}
 \underset{\{ r\leq |x|\leq l_n\}}\int \singx \Ks e^{\displaystyle v_n}\, dx\leq C (\tfrac{|x_n|}{\tau_n})^{2\alpha_n\tfrac{\gamma_1}{C_1}}r^{-2(1+\alpha_n)\tfrac{\gamma_1}{C_1}}\leq C r^{-2(1+\alpha_n)\tfrac{\gamma_1}{C_1}}.
\end{equation}
and
\begin{equation}
 \label{loggino3}
 \underset{\{ 4\frac{|x_n|}{\tau_n}\leq |x|\leq l_n\}}\int \log(\tfrac{|x|}{\rho})\singx \Ks e^{\displaystyle v_n}\, dx\leq C
\end{equation} \\
Indeed, let us prove (\ref{massetta3}). By using (\ref{stimaforse3}) and the fact that $$\tfrac{\singx}{(1+\frac{3}{2}\tfrac{\tau_n}{|x_n|}|x|)^{2\alpha_n}}\leq C,$$ for $|x|\geq4\tfrac{|x_n|}{\tau_n}$, we have that
\begin{align*}
 \underset{\{ r\leq |x|\leq l_n\}}\int \singx \Ks e^{\displaystyle v_n}\, dx&\leq C \underset{\{ r\leq |x|\leq l_n\}}\int \tfrac{\singx|x|^{-2(1+\tfrac{\gamma_1}{C_1})}}{\big((1+\tfrac{3}{2}\tfrac{\tau_n}{|x_n|}|x|)(\tfrac{\tau_n}{|x_n|}|x|)^{\tfrac{\gamma_1}{C_1}}\big)^{2\alpha_n}} \\
 &\leq C \underset{\{ r\leq |x|\leq l_n\}}\int (\tfrac{\tau_n}{|x_n|})^{-2\alpha_n\tfrac{\gamma_1}{C_1}}|x|^{-2(1+\tfrac{\gamma_1}{C_1})-2\alpha_n\tfrac{\gamma_1}{C_1}} \\
 &= C (\tfrac{\tau_n}{|x_n|})^{-2\alpha_n\tfrac{\gamma_1}{C_1}}\underset{\{ r\leq |x|\leq l_n\}}\int |x|^{-2-2(1+\alpha_n)\tfrac{\gamma_1}{C_1}} \\
 &\leq C (\tfrac{|x_n|}{\tau_n})^{2\alpha_n\tfrac{\gamma_1}{C_1}} r^{-2(1+\alpha_n)\tfrac{\gamma_1}{C_1}} \\
 &\leq C r^{-2(1+\alpha_n)\tfrac{\gamma_1}{C_1}},
\end{align*}
which is the desired conclusion. \\
Concerning (\ref{loggino3}), we use again (\ref{stimaforse3}) to conclude that
\begin{eqnarray*}
 \underset{\{ r\leq |x|\leq l_n\}}\int \log(\tfrac{|x|}{\rho})\singx \Ks e^{\displaystyle v_n}\, dx\leq \hspace{5cm}\\ \hspace{4cm}  \leq C (\tfrac{\tau_n}{|x_n|})^{-2\alpha_n\tfrac{\gamma_1}{C_1}}\underset{\{ 4\tfrac{|x_n|}{\tau_n}\leq |x|\leq l_n\}}\int \log(\tfrac{|x|}{\rho}) |x|^{-2-2(1+\alpha_n)\tfrac{\gamma_1}{C_1}} \\ \hspace{4cm}
  \leq C ((4\tfrac{|x_n|}{\tau_n})^{-2\tfrac{\gamma_1}{C_1}}\log(\tfrac{|x_n|}{\tau_n})+ (4\tfrac{|x_n|}{\tau_n})^{-2\tfrac{\gamma_1}{C_1}}) \leq C .\hspace{1cm}
\end{eqnarray*}
 \\
At this point, by using (\ref{massetta1}),  (\ref{massetta2}) and (\ref{massetta3}), then \eqref{pocamassa} easily holds and, by using (\ref{loggino1}), (\ref{loggino2}) and (\ref{loggino3}), then also \eqref{loginfinito} holds.\\
Eventually, by combining (\ref{logfinito}) with (\ref{loginfinito}), we deduce that,
\begin{equation}\label{logok}
 \int_{|y|\leq l_n}\log\big(\tfrac{|y|}{\rho}\big)\sing\overline{K}_{n}(y)e^{\displaystyle v_n(y)}\,dy\leq C_2.
\end{equation}
At last, let us first consider the case $l_n<\tfrac{\rho}{\tau_n}$, then by definition of $l_n$ it is readily seen that,
\begin{equation}\label{dis7casobuono}
 \int_{|y|\leq l_n}\sing \Ks(y)e^{\displaystyle v_n(y)}\,dy\geq 4\pi\Big(1+{\frac{1}{\sqrt{\overline\sigma}}}\Big)-C_1|\log(\tau_n)|^{-1}.
\end{equation}

Hence, by using (\ref{stimamaxbuona}), (\ref{logok}) and (\ref{dis7casobuono}), we have that,
\begin{eqnarray}
 \nonumber
 M_n &\geq&\int_{|y|\leq l_n}\Big(\tfrac{1}{2\pi}\log(\tfrac{1}{\tau_n})-\tfrac{1}{2\pi}\log\big(\tfrac{|y|}{\rho}\big)\Big)\sing\Ks(y)e^{\displaystyle v_n(y)}\,dy + p_n \\ \nonumber
 &\geq&2(1+\tfrac{1}{\sqrt{\overline\sigma}})\log(\tfrac{1}{\tau_n})- C +\underset{\O}{\inf}\, u_n\\ \nonumber
 &=&(1+\tfrac{1}{\sqrt{\overline\sigma}})(M_n+2\alpha_n\log|x_n|)- C +\underset{\O}{\inf}\, u_n, \nonumber
\end{eqnarray}
for a positive constant $C$ and this implies that
\[
 \tfrac{1}{\sqrt{\overline\sigma}}u_n(x_n) +\underset{\O}{\inf}\, u_n\leq C-2\alpha_n(1+\tfrac{1}{\sqrt{\overline\sigma}})\log|x_n|,
\]
which is a contradiction to (\ref{assurdo}).
\hfill \\
\hfill \\
Finally, we discuss the last part, that is the case in which $l_n=\tfrac{\rho}{\tau_n}$.
Still in this situation (\ref{mma4}) holds true and, by arguing as before, we are surely able to recover again (\ref{logok}).
Conversely, we cannot easily conclude (\ref{dis7casobuono}) from the definition of $l_n$. However, by using (\ref{key3}), we see that,
\begin{align}\nonumber
 \int_{|y|\leq l_n}|\tfrac{\tau_n}{x_n}y+&\tfrac{x_n}{|x_n|}|^{2\alpha_n} \Ks(y)e^{\displaystyle v_n(y)}\,dy\geq \\ \nonumber &\geq\int_{|y|\leq{(\frac{|x_n|}{4\tau_n})^{\frac 12}}}\Ks(y)e^{\displaystyle v_n(y)}\, d\mu_n  \\
 &\geq 4\pi(1+{\tfrac{1}{\sqrt{\overline\sigma}}})-C{\tfrac{|x_n|}{\tau_n}}^{-\Big({\tfrac{\delta^2}{\sqrt{\overline\sigma}}}\Big)}-C|\log(\tau_n)|^{-1}. \label{dis7casocattivo}
\end{align}
Hence, by recalling that $\sigma>\overline\sigma$, we have that for $n$ large enough,
\begin{eqnarray}
 \label{dispeggio}
 \int_{|y|\leq l_n}\sing \Ks(y)e^{\displaystyle v_n(y)}\,dy\geq 4\pi(1+{\tfrac{1}{\sqrt{\sigma}}})-C|\log(\tau_n)|^{-1}.
\end{eqnarray}

By using (\ref{stimamaxbuona}), (\ref{logok}) and (\ref{dispeggio}) , we have that
\begin{eqnarray}
 \nonumber
 M_n &\geq&\int_{|y|\leq l_n}\Big(\tfrac{1}{2\pi}\log(\tfrac{1}{\tau_n})-\tfrac{1}{2\pi}\log\big(\tfrac{|y|}{\rho}\big)\Big)\sing\Ks(y)e^{\displaystyle v_n(y)}\,dy + p_n \\ \nonumber
 &\geq&2(1+\tfrac{1}{\sqrt{\sigma}})\log(\tfrac{1}{\tau_n})- C +\underset{\O}{\inf}\, u_n\\ \nonumber
 &=&(1+\tfrac{1}{\sqrt{\sigma}})(M_n+2\alpha_n\log|x_n|)- C +\underset{\O}{\inf}\, u_n, \nonumber
\end{eqnarray}
for a positive constant $C$ and this implies that
$$
 \tfrac{1}{\sqrt{\sigma}}u_n(x_n) +\underset{\O}{\inf}\, u_n\leq C-2\alpha_n(1+\tfrac{1}{\sqrt{\sigma}})\log|x_n|
$$
which is a contradiction to (\ref{assurdo1}).\\
This completes the proof of case (II).

 \hfill\\

\section{Case (I)}\label{Case I}
In this section we discuss the adaptation of the argument in Chen-Lin \cite{CL} required to handle case (I). We will be rather sketchy concerning those parts which can be worked out exactly as in \cite{CL} or either as in case (II). We refer the interested reader to \cite{cos} for further details.\\
By assumption, there exists $C>0$ such that
\[
 \tfrac{|x_n|}{\delta_n}\leq C.
\]
We recall that $M_n:=u_n(x_n)$, $\delta_n:=\exp\big(-\frac{M_n}{2(1+\alpha_n)}\big)$ and define $\overline{L}_n:=\frac{1}{2}\delta_n^{-1}$. At this stage, along a subsequence which we do not relabel, there exists a point $y_0\in\mathbb{R}^2$ such that $\frac{x_n}{\delta_n}\rightarrow y_0$.
Let us define
\[
 v_n(y):=u_n(\delta_n y)-u_n(x_n),
\]
where $y\in\tilde{D}_n:=\Big\{\big|\delta_n y-x_n\big|\leq \frac{1}{2}  \Big\}=\Big\{\big| y-\xndn\big|\leq \overline{L}_n  \Big\}$ and we consider $n$ sufficiently large so that $|x_n|\leq\frac{1}{2}$ and $\delta_n y\in A$, whenever $y\in\tilde{D}_n$. \\  The function $v_n$ satisfies a Liouville type equation,
\begin{equation}\label{vequation}
  -\D v_n(y)=|y|^{2\alpha_n}\overline{K}_n(y)e^{\displaystyle v_n(y)},
\end{equation}
\[
 v_n(y)\leq 0=v_n(\txndn),
\]

with $y\in \tilde{D}_n$ and where $\Ks(y)=K_n(\delta_ny)$. \\
Now, by using the Green representation formula and by arguing as in \eqref{Green4}, for every $|y-\txndn|\leq l$, we have that,
\[
 v_n(y)\geq -C(1+l^{3+2\alpha_n}),
\]
for any fixed $l$.
\\
By standard elliptic estimates (\cite{GT}), we can pass to subsequences $\{v_n\}$, $\{\Ks\}$ such that
\begin{eqnarray*}
 v_n\to w \qquad &\text{in}& \,\, C_{loc}^{1,\gamma}(\R^2\backslash \{0\})\cap W_{loc}^{2,p}(\R^2\backslash\{0\})\cap W_{loc}^{2,q}(\R^2)\cap C_{loc}^{0,\kappa}(\R^2), \\
 \Ks\overset{\star}{\rightharpoonup} K_0 \qquad &\text{in}& \,\, L^{\infty}_{loc}(\R^2),
\end{eqnarray*}
with $\gamma\in(0,1)$, $k\in (0,k_\infty)$ for some $k_\infty\leq1$ which depends on $\alpha_\infty$, $p\geq1$  and $q\in [1,\frac{1}{|\alpha_\infty|})$. Then $w$ weakly satisfies the following equation
\begin{equation}
 -\D w=|y|^{2\alpha_\infty}K_0e^{\displaystyle w} \qquad \text{in}\,\,\mathbb{R}^2.
\end{equation}
From (\ref{Kbasic}) and (\ref{K1}) we deduce that $a\leq K_0\leq b$ and
\[
 \frac{\Ks(x)}{\Ks(y)}\leq \overline\sigma + \frac{B}{|\log|\delta_n(x-y)||}\longrightarrow
\overline\sigma,\qquad \text{for} \,\,n\to+\infty,\]
for every $x,y$ in a compact set and, also in this case, we have that,
\[
 \frac{\esssup \, K_0}{\essinf K_0}\leq \overline\sigma.
\]
Therefore, by using Theorem 1.5 (\cite{b}), we deduce that,
\[
 I_w=\int_{\R^2}|y|^{2\alpha_\infty}K_0e^{\displaystyle w}\geq 4\pi\Big(1+\alpha_\infty\Big)\Big(1+\sqrt{\tfrac{ess.inf \, K_0}{ess.sup \, K_0}}\Big)\geq 4\pi\Big(1+\alpha_\infty\Big)\Big(1+{\tfrac{1}{\sqrt{\overline\sigma}}}\Big).
\]
Now let $\rho=\frac{1}{2}\overline\rho$, $L_n=\rho e^{\frac{M_n}{2(1+\alpha_n)}}=\rh{\delta^{-1}_n}$ and $G(x,z)$ the Green's function of $-\D$ on $B_{\rho}(x_n)$. Then, if $D_{n,L_n}=\{|y-\xndn|\leq L_n\}$, we have
\begin{eqnarray}
\nonumber
 M_n&=&\int_{|z-x_n|\leq\rho}|z|^{2\alpha_n}K_n(z)e^{\dis u_n(z)}G(x_n,z)\,dz + \frac{1}{2\pi\rho}\int_{|z-x_n|=\rho}u_n(z)\,dl \\ \nonumber
 &=&\int_{D_{n,L_n}}|y|^{2\alpha_n}
 \overline{K}_{n}(y)e^{\displaystyle v_n(y)}G(x_n,\delta_ny)\,dy + \frac{1}{2\pi\rho}\int_{|z-x_n|=\rho}u_n(z)\,dl \\
 &=&\int_{D_{n,L_n}}\Big(\tfrac{M_n}{4\pi(1+\alpha_n)}-\tfrac{1}{2\pi}\log\big(\tfrac{|y-\xndn|}{\rho}\big)\Big)|y|^{2\alpha_n}\overline{K}_{n}(y)e^{\displaystyle v_n(y)}\,dy + p_n,  \label{stimamax}
\end{eqnarray}
where $p_n=\frac{1}{2\pi\rho}\int_{|z-x_n|=\rho}u_n(z)\,dl$. \\
Actually, by using \eqref{stimamax} and arguing as in the previous section, we can prove that the total curvature $I_w$ is minimal, namely $$I_w=4\pi\Big(1+\alpha_\infty\Big)\Big(1+{\frac{1}{\sqrt{\overline\sigma}}}\Big).$$
Thus, we use again Theorem 1.5 (\cite{b}) to say that $w$ is a radially symmetric and strictly decreasing function with respect to $0$. In particular, by using the fact that $v_n$ is locally uniformly converging to $w$ and recalling that $\xndn$ is a maximum point of $v_n$ and $\xndn\to y_0$, we have that $y_0=0$.\hfill
\medskip

Now we want to find $l_n\leq L_n$ such that, putting ${D}_{n,l_n}=\{|y-\xndn|\leq l_n\}$, we have
\begin{equation}\label{dis1}
 \int_{{D}_{n,l_n}}|y|^{2\alpha_n}\overline{K}_{n}(y)e^{\displaystyle v_n(y)}\,dy\geq 4\pi\Big(1+\alpha_n\Big)\Big(1+{\tfrac{1}{\sqrt{\overline\sigma}}}\Big)-C_1M_n^{-1}
\end{equation}
and
\begin{equation}\label{dis2}
 \int_{{D}_{n,l_n}}\log\big(\tfrac{|y-\xndn|}{\rho}\big)|y|^{2\alpha_n}\overline{K}_{n}(y)e^{\displaystyle v_n(y)}\,dy\leq C_2,
\end{equation}
for $C_1,C_2$ positive constants that do not depend on $n$, as $n\to+\infty$. \\
We define,
\[
 l_n:=\sup\Big\{l\leq L_n\Big|\int_{{D}_{n,l}}|y|^{2\alpha_n}\overline{K}_{n}(y)e^{\displaystyle v_n(y)}\,dy\leq 4\pi\Big(1+\alpha_n\Big)\Big(1+{\tfrac{1}{\sqrt{\overline\sigma}}}\Big)\Big\},
\]
where ${D}_{n,l}=\{|y-\xndn|\leq l\}$.  From the fact that $I_w=4\pi\Big(1+\alpha_\infty\Big)\Big(1+{\frac{1}{\sqrt{\overline\sigma}}}\Big)$ and $v_n\to w$ uniformly on the compacts, we deduce that $l_n\to+\infty$. \\
We can choose $\epsilon$ as a positive number to be fixed later on, such that, for $n>>0$, there exist $r_{1,n}>1$ and $r_1>1$ such that $4r_{1,n}\leq \frac{l_n}{4}$ and
\[
 \int_{B_{r_1}(0)}|y|^{2\alpha_\infty}K_0e^{\displaystyle w}\,dy=4\pi(1+\alpha_\infty)(1+\tfrac{1}{\sqrt{{\overline\sigma}}})-\epsilon,
\]
\[
 \int_{\{y:|y-\xndn|\leq r_{1,n}\}}|y|^{2\alpha_n}\Ks e^{\displaystyle v_n}\,dy=4\pi(1+\alpha_n)(1+\tfrac{1}{\sqrt{{\overline\sigma}}})-\epsilon.
\]
We notice that $r_{1,n}\to r_1$ and, as $n\to+\infty$ ,
\begin{equation}\label{stimanello}
\int_{\{y:r_{1,n}\leq|y-\xndn|\leq l_n\}}|y|^{2\alpha_n}\Ks e^{\displaystyle v_n}\,dy\leq\epsilon.
\end{equation}

At this stage, as we did in \eqref{Suzukiargument}, by using \eqref{stimanello}, Lemma \ref{Suzukilemma}  and the Harnack principle, we deduce the following estimate: for every $1\leq|x-\txndn|\leq \frac{l_n}{4}$,
\begin{equation}\label{3dis}
 v_n(x)\geq \underset{\{\sqrt[\delta]{|x-\xndn|}\leq|y-\xndn|\leq \frac{l_n}{2}\}}{\sup} v_n(y).
\end{equation}
\\
Set
\[
m_n(r):=\underset{\{|y-\xndn|=r\}}{\max} v_n(y), \qquad\qquad t_{0}:=m_n(\tfrac{l_n}{2}).
\]

Now, recalling that $\underset{D_{n,l_n/2}}\max \, v_n=v_n(\txndn)=0$, let us define
\[
 d\mu_n:=|x|^{2\alpha_n}dx, \qquad \qquad d\sigma_n:=|x|^{\alpha_n} dl,
\]
\[
 \O_{t}^n:=\{y\big| |y-\txndn|\leq \tfrac{l_n}{2} \,\,\text{and} \,\, v_n(y)>t \},\qquad \qquad \xi_n(t):=\int_{\Otn}d\mu_n,
\]
for any $t\in(t_{0},0)$. %where $A_n=\underset{D_{n,l_n}}{inf} v_n$.We notice that $A_n\to -\infty$, when $n\to+\infty$,
It is easy to see that $\overline{\O_t^n}\subseteq\{z\big| |z-\txndn|\leq \frac{l_n}{2} \}$ and, also that $\underset{t\to 0^-}             {\lim}\xi_n(t)=0
$ and $\underset{t\to t_{0}^+}{\lim}\xi_n(t)=\int_{{\O^n_{t_{0}}}}d\mu_n=:\xi_n(t_{0})$ . Since $v_n\in W_{loc}^{2,p}(\R^2\backslash\{0\})$ for $p>2$, as a consequence of the Generalized Sard's Lemma (see (\cite{F})), $\p\Otn$ is a $C^1$ closed curve for a.a $t\in(\ton,0)$ and since $v_n$ satisfies (\ref{vequation}), $\p\Otn$ has null measure for a.a $t\in(\ton,0)$. Actually, it turns out that the level sets of $v_n$ have null measure for every $t\in(\ton,0)$, see Lemma \ref{levelsetlemma} in
the appendix for further details. This easily implies that $\xi_n
$ is a continuous, strictly decreasing function, which is almost everywhere differentiable in $(\ton,0)$. Indeed, by using the Coarea formula (see (\cite{BW})), it holds $$\tfrac{d\xi_n}{dt}(t)=-\int_{\p \Otn}\tfrac{|x|^{2\alpha_n}}{|\nabla v_n|}dl,$$ for almost any $t\in(\ton,0)$. \\
We introduce $\vns(|x-\xndn|)$, the weighted symmetric decreasing rearrangement of $v_n$ with respect to the measure $d\mu_n$, defined in the following way.
For a fixed $r=|x-\txndn|$, we have,
\begin{eqnarray*}
 \vns(r)&:=&\sup\{t\in (\ton,0): \xi_n(t)>\pi r^2\} \\
 &=& \sup\{t\in (\ton,0): s_n(t)>r\},
\end{eqnarray*}
where $s_n(t)=(\pi^{-1}\xi_n(t))^{\tfrac{1}{2}}$ and this implies that
\[
 (\O_t^{n})^\ast:=\{y\big| |y-\txndn|\leq \tfrac{l_n}{2} \,\,\text{and} \,\, \vns(y)>t \}=B_{s_n(t)}(\txndn).
\]
Firstly, we notice that $s_n(t)\leq \sqrt{\tfrac{2}{1+\alpha_n}}(\tfrac{l_n}{2})^{1+\alpha_n}$ for all $t\in(\ton,0)$, see the calculations in \eqref{misurapalle1} . \\
Moreover, the function $\vns:(0,s_n(t_{0}))\to (\ton,0)$ satisfies %is the inverse of $\xi_n$ and we have that
\[
 \underset{r\to s_n(\ton)^-}{\lim}\vns(r)=\ton ,\qquad
  \qquad \underset{r\to 0^+}{\lim}\vns(r)=0.
\]
In particular $\vns$ is continuous and strictly decreasing. Moreover, we can conclude that $\vns$ is locally Lipschitz and almost everywhere differentiable  in $(0,s_n(\ton))$.
At this point, let us define
\begin{equation} \label{Fn}
 F_n(r):=\int_{\Ovn} \Ks(y)e^{\displaystyle v_n(y)}d\mu_n(y),
\end{equation}
which is a locally Lipschitz function in $(0,s_n(\ton))$.  Thus $F_n$ is almost everywhere differentiable in $(0,s_n(\ton))$.
Now, from (\ref{3dis}) the following inclusions hold,
\begin{equation}
 \label{inclusion}
 B_{R^{\delta^2}}(\txndn)\subseteq \O_{m_n(R^\delta)}^n \subseteq B_{R}(\txndn).
\end{equation}
for $1\leq R\leq\frac{l_n}{4}$. We skip the proof of this fact since the argument  is similar to the one used in \eqref{inclusion3}.
Let us set, for every open set $\omega$,
\[
\mu_n(\omega):=\int_{\omega}\,d\mu_n,
\]
then $\mu_n(B_r(0))=\tfrac{\pi}{\alpha_n+1}r^{2\alpha_n+2}$. At this point,
\[B_r(\txndn)=\Big(B_r(0)\cup \big(B_r(\txndn)\backslash B_r(0)\big)\Big)\backslash\Big(B_r(0)\backslash B_r(\txndn)\Big)\] and
$$\mu_n\big(B_r(\txndn)\backslash B_r(0)\big)=r^{2\alpha_n+2}o(1)=\mu_n\big(B_r(0)\backslash B_r(\txndn)\big),$$ as $n\to+\infty$. Thus, we have that
\begin{equation}\label{misurapalle1}
 \mu_n(B_r(\txndn)):=r^{2\alpha_n+2}(\tfrac{\pi}{\alpha_n+1}+o(1)),
\end{equation}
as $n\to+\infty$.
From (\ref{inclusion}), we deduce for $1\leq R\leq \frac{l_n}{4}$,
\[
 \mu_n\big(B_{R^{\delta^2}}(\txndn)   \big)\leq \mu_n\big(\O^n_{m_n(R^\delta)}\big)=\xi_n(m_n(R^\delta))\leq \mu_n\big(B_{R}(\txndn)   \big),
\]
that is,
\[
 R^{\delta^2(2+2\alpha_n)}\Big(\tfrac{\pi}{1+\alpha_n}+o(1)\Big)\leq \xi_n\big(m_n(R^\delta)\big)\leq R^{2+2\alpha_n}\Big(\tfrac{\pi}{1+\alpha_n}+o(1)\Big),
\]
as $n\to+\infty.$ \\
Therefore there exist two constants $c_n^1,c_n^2$, infinitesimal as $n\to+\infty$, such that,
\begin{equation}
 \label{raggigenerali}
 R^{\delta^2(1+\alpha_n)}\Big(\tfrac{1}{\sqrt{1+\alpha_n}}+c_n^1\Big)\leq \Big(\pi^{-1}\xi_n\big(m_n(R^\delta)\big)\Big)^{\frac{1}{2}}\leq R^{1+\alpha_n}\Big(\tfrac{1}{\sqrt{1+\alpha_n}}+c_n^2\Big).
 \end{equation}
By using again (\ref{inclusion}) and the fact that $\vns$ is strictly decreasing, then for $1\leq R\leq\frac{l_n}{4}$,
\begin{equation}
 \label{integrali}
 \int_{B_R(\txndn)}\Ks(y)e^{\displaystyle v_n(y)}\, d\mu_n\geq F_n\Big(\big(\pi^{-1}\xi_n\big(m_n(R^\delta)\big)\big)^\frac{1}{2}\Big)\geq F_n\Big(\tfrac{1}{2\sqrt{1+\alpha_n}}R^{\delta^2(1+\alpha_n)}\Big),
\end{equation}
having chosen $n>>0$ such that $c_n^1>-\frac{1}{2}\geq-\frac{1}{2\sqrt{1+\alpha_n}}$. \\
Hence, fixing $R_n=\Big(\pi^{-1}\xi_n\big(m_n\big(({\frac{l_n}{4}})^{\frac{\delta}{2}}\big)\big)\Big)^\frac{1}{2}$, we have that,
\begin{equation}
 \label{raggi}
 \tfrac{1}{2\sqrt{1+\alpha_n}}\Big(\frac{l_n}{4}\Big)^{\frac{\delta^2}{2}(1+\alpha_n)}\leq R_n\leq \Big(\frac{l_n}{4}\Big)^{\frac{1}{2}(1+\alpha_n)}\Big(\tfrac{1}{\sqrt{1+\alpha_n}}+c_n^2\Big).
\end{equation}
Obviously, $R_n\to +\infty$, as $n\to+\infty$, and $\vns(r)$ is defined for every $r\leq R_n<s_n(t_0)$. \\
Let us define, for a.a $r\in(0,R_n),$
\[
 \hat{K}_n(r):=\frac{F'_n(r)}{2\pi re^{\dis\vns(r)}}
\]
and
\[
 a_n:=\underset{\{|y-\xndn|\leq (\frac{l_n}{4})^{\frac{1}{2}}\}}{ess.\inf} \Ks(y), \qquad b_n:=\underset{\{|y-\xndn|\leq (\frac{l_n}{4})^{\frac{1}{2}}\}}{ess.\sup} \Ks(y).
\]
Then, by using the fact that $\xi_n(\vns)=|B_r|$, we have that for almost any $r\in (0,R_n)$,
\begin{equation}
 \label{Ka}
 a_n\leq \hat{K}_n(r)\leq b_n.
\end{equation}

Let $I_1$ be the set of those $r\in(0,R_n)$ where $\frac{d}{dr}\vns(r)$ does not exist and let $I_2$ be the set of
those $r\in(0,R_n)$ where $\frac{d}{dr}\vns(r)=0$. If we denote $E=\vns(I_1\cup I_2)$, then $\mathcal{H}^1(E)=0$, since $\vns$ is locally Lipschitz. Furthermore, let $I_c=(\vns)^{-1}(E_c)$, where $E_c$ is the set of critical values of $v_n$. By the Generalized Sard's Lemma (\cite{F}), we have that $\mathcal{H}^1(E_c)=0$. Now, let $I$ be the set of those $r$ such that $\vns(r)=t$ for some $t\in (t_0,0)\backslash\{E\cup E_c\}$. \\
Hence, for any $r\in I$, we can apply the Cauchy-Schwartz inequality to deduce
\begin{eqnarray*}
 \nonumber
 \Big(\int_{\partial \O^n_{\vns (r)}}\,d\sigma_n\Big)^2&\leq& \Big(\int_{\partial \O^n_{\vns (r)}}\frac{|x|^{2\alpha_n}}{|\nabla v_n|}\,dl\Big)\Big(\int_{\partial \O^n_{\vns (r)}}|\nabla v_n|\,dl\Big)\\
 &\leq& \Big(-\frac{d\vns(r)}{dr}\Big)^{-1}(2\pi r)\Big(\int_{\partial \O^n_{\vns (r)}}-\frac{\partial v_n}{\partial \nu}\,dl\Big),
\end{eqnarray*}
where we used the fact that $\vns(r)=\eta^{\star}_n(\pi r^2)$, $\eta^{\star}_n$ is the inverse of $\xi_n$ and $\nu$ is the exterior normal to $\partial \O^n_{\vns (r)}$. Moreover, we have that,
\[
 \int_{\partial \O^n_{\vns (r)}}|\nabla v_n|\,dl=\int_{ \O^n_{\vns (r)}}\Ks e^{\displaystyle v_n}\,d\mu_n=F_n(r),
\]
which implies, for every $s\in I$,
\begin{equation}
 \label{disdiff1}
 2\pi r F_n(r)\geq \Big(\int_{\partial \O^n_{\vns (r)}}\,d\sigma_n\Big)^2\Big(-\frac{d\vns(r)}{dr}\Big).
\end{equation}
Since $v_n$ is superharmonic and by using the maximum principle we can deduce that each connected component of $\O^n_{\vns (r)}$ is simply connected. \\
Hence, we apply Huber's inequality (see Proposition \ref{Huber.App.prop} in the appendix) and conclude that
\[
 \Big(\int_{\partial \O^n_{\vns (r)}}\,d\sigma_n\Big)^2\geq \beta_{\alpha_n,r}\xi_n(\vns(r)),
\]
where $\beta_{\alpha_n,r}=\beta_{\alpha_n,\O^n_{\vns (r)}}$ is equal to $4\pi(1+\alpha)$ if $0\in\O^n_{\vns (r)}$ or $4\pi$ if $0\notin\overline\O^n_{\vns (r)}$. \\
By using this inequality in (\ref{disdiff1}), we conclude that
\begin{eqnarray}
 F_n(r)&\geq& \beta_{\alpha_n,r}\xi_n(\vns(r))\Big(-\frac{d\vns(r)}{dr}\Big) \frac{1}{2\pi r } \nonumber \\
 &=& \beta_{\alpha_n,r}\frac{r}{2}\Big(-\frac{d\vns(r)}{dr}\Big),\label{diff2}
\end{eqnarray}
for every $r\in I$. The latter inequality is always true for $r\in I_2$, so it holds for $r\in I\cup I_2$. If we define $I_3$ to be the set of those $r$ for which (\ref{diff2}) does not hold, then $I_3\subset (I_1\cup I_c)\backslash I_2$. On the other hand, since $\mathcal{H}^1(I_1)=0$ and $\mathcal{H}^1(E_c)=0$, there is no possibility that $I_3$ does have positive measure. This means that (\ref{diff2}) holds for a.a. $r\in(0,R_n)$. \\
At this point, from (\ref{diff2}), we deduce that for almost any $r\in (0,R_n)$,
\begin{align}\nonumber
 \frac{d}{dr}\Big(  \frac{rF'_n(r)}{\hat{K}_n(r)}\Big)=\frac{d}{dr}(2\pi r^2e^{\dis\vns(r)})&=4\pi r e^{\dis\vns(r)}+2\pi r^2e^{\dis\vns(r)}\frac{d}{dr}\vns(r) \\ \nonumber
 &\geq \frac{2F'_n(r)}{\hat{K}_n(r)}-\frac{F'_n(r)}{\hat{K}_n(r)}\frac{2}{\beta_{\alpha_n,r}}F_n(r) \\ \label{diffeq}
 &= \frac{2F'_n(r)}{\hat{K}_n(r)}\Big(1-\frac{F_n(r)}{\beta_{\alpha_n,r}}\Big). %\\
 %&\geq \begin{cases}
      %\frac{2F'_n(r)}{a_n}\Big(1-\frac{F_n(r)}{\beta_{\alpha_n,r}}\Big) \qquad \text{if}\, F_n(r)>\beta_{\alpha_n,r} \\
      %\frac{2F'_n(r)}{b_n}\Big(1-\frac{F_n(r)}{\beta_{\alpha_n,r}}\Big) \qquad \text{if}\, F_n(r)\leq\beta_{\alpha_n,r}
     %\end{cases}
\end{align}
Let us set
\begin{equation}
 \label{s1}
 s_1=s_{1,n}:=\inf\{\tilde r>0\,\big| \,0\in\O^n_{\vns (r)}, \forall r>\tilde r\}.
\end{equation}

At this point, from the inclusions (\ref{inclusion}) and taking $R=(\frac{l_n}{4})^\frac{1}{2}$, we deduce that,
\begin{align}\nonumber
 \int_{B_{R^{\delta^2}}(\xndn)}\Ks e^{\displaystyle v_n}\,d\mu_n\leq \int_{\O^n_{m_n(R^\delta)}}\Ks e^{\displaystyle v_n}\,d\mu_n=F_n(R_n)\leq\int_{B_R(\xndn)}\Ks e^{\displaystyle v_n}\,d\mu_n. \\
 \label{s0esiste}
\end{align}
This implies that $F_n(R_n)\to 4\pi(1+\alpha_\infty)(1+\frac{1}{\sqrt{{\overline\sigma}}})$, as $n\to+\infty$. Hence, by the continuity and monotonicity of $F_n$, there exists, for $n$ sufficiently large, $s_0=s_{0,n}<R_n$ such that $F_n(s_0)=4\pi(1+\alpha_n)$.  \\
Now, to analyze  (\ref{diffeq}), we split the discussion in two cases: firstly, if $s_0>s_1$, then for every $r\in (s_0,R_n)$ we have
\begin{align}
 \nonumber
 \frac{rF'_n(r)}{\hat{K}_n(r)}&\geq \frac{2}{b_n}\int_0^{s_1}F'_n\Big(1-\frac{F_n}{4\pi}\Big)ds+\frac{2}{b_n}\int_{s_1}^{s_0}F'_n\Big(1-\frac{F_n}{4\pi(1+\alpha_n)}\Big)ds \\ \label{firstcase}
 &+\frac{2}{a_n}\int_{s_0}^{r}F'_n\Big(1-\frac{F_n}{4\pi(1+\alpha_n)}\Big)ds.
\end{align}
On the other hand, if $s_1\geq s_0$, then for every $r\in (s_0,R_n)$,
\begin{align}
 \label{secondcase}
 \frac{rF'_n(r)}{\hat{K}_n(r)}&\geq \frac{2}{b_n}\int_0^{s_0}F'_n\Big(1-\frac{F_n}{4\pi}\Big)ds+\int_{s_0}^{r}\frac{2F'_n}{\hat{K}_n}\Big(1-\frac{F_n}{\beta_{\alpha_n,s}}\Big)ds.
\end{align}

Both (\ref{firstcase}) and (\ref{secondcase}) imply the following crucial inequality,
\begin{equation}
 \label{diffeq2}
 \tfrac{rF'_n(r)}{\hat{K}_n(r)}\geq -\tfrac{1}{4\pi(1+\alpha_n)a_n}\big(F_n(r)-4\pi(1+\alpha_n)(1-{\scriptstyle\sqrt{\tfrac{a_n}{b_n}}})\big)\big(F_n(r)-4\pi(1+\alpha_n)(1+{\scriptstyle\sqrt{\tfrac{a_n}{b_n}}})\big).
\end{equation}
We skip the proof of this fact and refer to the argument used in Section \ref{Case (II)} to prove \eqref{diffeq23}.\\
Once we have established (\ref{diffeq2}) for every $s_0\leq r \leq R_n$, we define
\[
 \tilde{R}_n:=\sup\Big\{r\leq R_n\big| F_n(r)\leq 4\pi(1+\alpha_n)(1+\sqrt{\tfrac{a_n}{b_n}})\Big\}.
\]
Then, from (\ref{diffeq2}), it follows that
\begin{equation}
 \label{dd}
 \frac{F'_n(r)}{F_n(r)-4\pi(1+\alpha_n)(1-\sqrt{\tfrac{a_n}{b_n}})}+\frac{F'_n(r)}{4\pi(1+\alpha_n)(1+\sqrt{\tfrac{a_n}{b_n}})-F_n(r)}\geq 2\sqrt{\frac{a_n}{b_n}}\frac{1}{r},
\end{equation}
for $s_0\leq r \leq \tilde{R}_n$. By integrating the previous inequality,
\[
 \log\bigg(\tfrac{F_n(r)-4\pi(1+\alpha_n)(1-\sqrt{\tfrac{a_n}{b_n}})}{4\pi(1+\alpha_n)(1+\sqrt{\tfrac{a_n}{b_n}})-F_n(r)}\bigg)+\log\bigg(\tfrac{4\pi(1+\alpha_n)(1+\sqrt{\tfrac{a_n}{b_n}})-F_n(s_0)}{F_n(s_0)-4\pi(1+\alpha_n)(1-\sqrt{\tfrac{a_n}{b_n}})}\bigg)\geq 2\sqrt{\frac{a_n}{b_n}}\log(\tfrac{r}{s_0})
\]
and, using the fact that $F_n(s_0)=4\pi(1+\alpha_n)$, we have that
\begin{equation}
 \label{loga}
\log\Bigg(\frac{4\pi(1+\alpha_n)(1+\sqrt{\tfrac{a_n}{b_n}})-F_n(r)}{F_n(r)-4\pi(1+\alpha_n)(1-\sqrt{\tfrac{a_n}{b_n}})}\Bigg)\leq -2\sqrt{\frac{a_n}{b_n}}\log(\tfrac{r}{s_0}).
\end{equation}
Thus, from (\ref{loga}) we deduce that,
\begin{eqnarray}
 \nonumber
 4\pi(1+\alpha_n)(1+\sqrt{\tfrac{a_n}{b_n}})&\leq& F_n(r)+ \Big(\tfrac{r}{s_0}\Big)^{-2\sqrt{\frac{a_n}{b_n}}}\scriptstyle\Big(F_n(r)-4\pi(1+\alpha_n)(1-\sqrt{\tfrac{a_n}{b_n}})\Big) \\ \label{ka}
 &\leq& F_n(r)+ \Big(\frac{r}{s_0}\Big)^{-2\sqrt{\frac{a_n}{b_n}}}F_n(r).
\end{eqnarray}
We notice that $F_n$ and $s_0$ are uniformly bounded and we refer to \eqref{boundedFS} in section \ref{Case (II)} for a similar proof of these facts.
\\
Hence, there exists a positive constant $C$ such that
\begin{equation}
 \label{meglio}
 F_n(r)\geq4\pi(1+\alpha_n)(1+\sqrt{\frac{a_n}{b_n}})-Cr^{-2\sqrt{\frac{a_n}{b_n}}},
\end{equation}
for $s_0\leq r\leq \tilde{R}_n$. Obviously, (\ref{meglio}) is always true for
$\tilde{R}_n\leq r\leq{R}_n$. \\
By the definition of $a_n$ and $b_n$, we have that,
\[
 \frac{b_n}{a_n}\leq \underset{|y-\xndn|,|z-\xndn|\leq(\frac{l_n}{4})^{\frac{1}{2}}}{\sup}\frac{\Ks(y)}{\Ks(z)}\leq {\overline\sigma} + \frac{c}{|\log (l_n^{\frac{1}{2}}e^{-\frac{M_n}{2}})|}\leq \overline\sigma +C_1M_n^{-1}.
\]
Hence we have,
\[
 \sqrt{\frac{a_n}{b_n}}\geq \frac{1}{\sqrt{{\overline\sigma}}}-CM_n^{-1},
\]
which, combined with (\ref{meglio}), implies that,
\begin{equation}
 \label{meglio2}
 F_n(r)\geq4\pi(1+\alpha_n)(1+{\tfrac{1}{\sqrt{\overline\sigma}}})-Cr^{-{\tfrac{2}{\sqrt{\overline\sigma}}}}-CM_n^{-1},
\end{equation}
for $s_0\leq r\leq R_n$. The latter estimate, together with (\ref{integrali}), implies that
\begin{align}
 \nonumber
 \int_{B_R(\xndn)}\Ks(y)e^{\displaystyle v_n(y)}\, d\mu_n &\geq F_n\Big(\tfrac{1}{2\sqrt{1+\alpha_n}}R^{\delta^2(1+\alpha_n)}\Big)\geq \\
 \geq 4\pi(1+\alpha_n)&(1+{\tfrac{1}{\sqrt{\overline\sigma}}})-C(1+\alpha_n)^{\frac{1}{\sqrt{{\overline\sigma}}}}R^{-\Big({2(1+\alpha_n)\tfrac{\delta^2}{\sqrt{\overline\sigma}}}\Big)}-CM_n^{-1}, \label{key}
\end{align}
for $S_0\leq R\leq \big(\frac{l_n}{4}\big)^{\frac{1}{2}}$, where $S_0=\max\{1,(2\sqrt{1+\alpha_n }s_0)^{(1+\alpha_n)^{-1}\delta^{-2}}\}$. Hence,
\begin{eqnarray}
 \int_{R\leq |y-\txndn|\leq l_n}\Ks(y)e^{\displaystyle v_n(y)}\, d\mu_n \leq C(1+\alpha_n)^{\frac{1}{\sqrt{{\overline\sigma}}}}R^{-\Big({2(1+\alpha_n)\tfrac{\delta^2}{\sqrt{\overline\sigma}}}\Big)}+CM_n^{-1},\label{mma}
\end{eqnarray}
for $S_0\leq R\leq \big(\tfrac{l_n}{4}\big)^{\frac{1}{2}}$. \\
In particular, \eqref{mma} holds for $\tilde{S_0}\leq R\leq \big(\frac{l_n}{4}\big)^{\frac{1}{2}}$, with $\tilde{S_0}=\max\{S_0,r_{1,n}\}$ and we have that $\{R\leq|y-\xndn|\leq l_n\}\subset\{r_{1,n}\leq|y-\xndn|\leq l_n\}$. At this point, we apply again Suzuki's lemma (see the appendix, section \ref{Suzuki}) in the ball $B_R(x)$, with $|x-\xndn|=2R$. We notice that $B_R(x)\subset\{R\leq|y-\xndn|\leq l_n\}$ for $n$ sufficiently large. Hence, for $2\tilde{S_0}\leq|x-\xndn|\leq (l_n)^{\frac{1}{2}}$,
\begin{eqnarray*}
 \nonumber
 v_n(x)&\leq&\frac{1}{2\pi R}\int_{\p B_R(x)}v_n\,dl-2\log\Big(1-\frac{b}{8\pi}\int_{B_R(x)}|y|^{2\alpha_n}e^{\displaystyle v_n}\,dx\Big)_+ \\
 &\leq& \frac{1}{2\pi R}\int_{\p B_R(x)}v_n\,dl+\log 4 \\
 &=&\frac{1}{\pi R^2}\int_{B_R(x)}v_n(y)\,dy+\log 4
\end{eqnarray*}
and by applying Jensen's inequality and (\ref{mma}), we have that,
\begin{eqnarray}
 \nonumber
 e^{\displaystyle v_n(x)}&\leq&\frac{4}{\pi R^2}\int_{B_R(x)}e^{\displaystyle v_n(y)}\,dy \hspace{3cm} \\ \nonumber
 &\leq& \frac{4}{\pi R^2 a}\frac{1}{2^{2\alpha_n}|x-\xndn|^{2\alpha_n}}\int_{B_R(x)}|y|^{2\alpha_n}\Ks(y) e^{\displaystyle v_n(y)}\,dy \\ \nonumber
 &=& \frac{16}{\pi a2^{2\alpha_n}|x-\xndn|^{2(1+\alpha_n)}}\int_{B_R(x)}\Ks(y) e^{\displaystyle v_n(y)}\,d\mu_n \\ \label{guaio}
 &\leq& C_3\Big[\big|x-\tfrac{x_n}{\delta_n}\big|^{(1+\alpha_n)(-2-2\frac{\delta^{2}}{\sqrt{{\overline\sigma}}})}+M_n^{-1}\big|x-\tfrac{x_n}{\delta_n}\big|^{-2(1+\alpha_n)}\Big],
\end{eqnarray}
for $2\tilde{S_0}\leq|x-\xndn|\leq (l_n)^{\frac{1}{2}}$.
Now we choose $l_n^\star$ satisfying $(\log l_n^\star)^2=\log l_n$ and, for $n$ sufficiently large, we assume that $2\tilde{S_0}\leq l_n^\star\leq (\frac{l_n}{4})^{\frac{1}{2}}$. \\
Hence, (\ref{mma}) implies that,
\begin{eqnarray}
 \int_{l_n^\star\leq |y-\txndn|\leq l_n}\Ks(y)e^{\displaystyle v_n(y)}\, d\mu_n \leq C(1+\alpha_n)^{\frac{1}{\sqrt{\overline\sigma}}}(l_n^\star)^{-\Big({2(1+\alpha_n)\tfrac{\delta^2}{\sqrt{\overline\sigma}}}\Big)}+CM_n^{-1} .  \label{I3}
\end{eqnarray}
At this point, we estimate the following integral
\begin{eqnarray}
 \label{log}
 \int_{|y-\xndn|\leq l_n}\log\big(\tfrac{|y-\xndn|}{\rho}\big)|y|^{2\alpha_n}\Ks e^{\displaystyle v_n}\, dy= I_1+I_2+I_3,
\end{eqnarray}
where
\[
 I_1=\int_{|y-\xndn|\leq 2\tilde{S_0}}\log\big(\tfrac{|y-\xndn|}{\rho}\big)|y|^{2\alpha_n}\Ks e^{\displaystyle v_n}\, dy,
\]
\[
 I_2=\int_{2\tilde{S_0}<|y-\xndn|\leq l_n^\star}\log\big(\tfrac{|y-\xndn|}{\rho}\big)|y|^{2\alpha_n}\Ks e^{\displaystyle v_n}\, dy,
\]
\[
 I_3=\int_{l_n^\star<|y-\xndn|\leq l_n }\log\big(\tfrac{|y-\xndn|}{\rho}\big)|y|^{2\alpha_n}\Ks e^{\displaystyle v_n}\, dy.
\]
For the first integral, using the fact that $v_n\leq0$, it yields,
\begin{align} \nonumber
 I_1&\leq\log\big(\tfrac{2\tilde{S_0}}{\rho}\big)b\int_{|y-\xndn|\leq 2\tilde{S_0}}|y|^{2\alpha_n}\, dy \\ \nonumber
 &\leq\log\big(\tfrac{2\tilde{S_0}}{\rho}\big)b\int_{\{|y-\xndn|\leq 2\tilde{S_0}\}\cap\{|y|>|y-\xndn|\}}|y|^{2\alpha_n}\, dy+ \\ \nonumber
 &\hspace{2cm}+\log\big(\tfrac{2\tilde{S_0}}{\rho}\big)b\int_{\{|y-\xndn|\leq 2\tilde{S_0}\}\cap\{|y|\leq|y-\xndn|\}}|y|^{2\alpha_n}\, dy
 \\ \nonumber
 &\leq \log\big(\tfrac{2\tilde{S_0}}{\rho}\big)2b\int_{|y-\xndn|\leq 2\tilde{S_0}}|y-\txndn|^{2\alpha_n} \,dy \\ \nonumber
 &\leq  \tfrac{4b}{2+2\alpha_n} \tilde{S_0}^{2\alpha_n+2}\log\big(\tfrac{2\tilde{S_0}}{\rho}\big)\leq C,
\end{align}
where $C$ is uniform in $n$ because $\tilde{S_0}$ does not depend on $n$. \\
Let us estimate $I_2$. By using (\ref{guaio}) and the fact that if $2\tilde{S_0}<|y-\xndn|\leq l_n^\star$, then, for n large enough,  $|y|^{2\alpha_n}\leq\frac{1}{2^{2\alpha_n}}|y-\xndn|^{2\alpha_n}$,
\begin{align*} \nonumber
 I_2&\leq bC_3 \int_{2\tilde{S_0}<|y-\xndn|\leq l_n^\star}\log\big(\tfrac{|y-\xndn|}{\rho}\big)|y|^{2\alpha_n}\big|y-\tfrac{x_n}{\delta_n}\big|^{(1+\alpha_n)(-2-2\frac{\delta^{2}}{\sqrt{\overline\sigma}})}\,dy+ \\  \nonumber
 &\, \hspace{1,8cm}+CM_n^{-1}\int_{2\tilde{S_0}<|y-\xndn|\leq l_n^\star}\log\big(\tfrac{|y-\xndn|}{\rho}\big)|y|^{2\alpha_n}\big|y-\tfrac{x_n}{\delta_n}\big|^{-2(1+\alpha_n)}\, dy  \\ \nonumber
 &\leq  C \int_{2\tilde{S_0}<|y-\xndn|\leq l_n^\star}\log\big(\tfrac{|y-\xndn|}{\rho}\big)\big|y-\tfrac{x_n}{\delta_n}\big|^{(-2-2(1+\alpha_n)\frac{\delta^{2}}{\sqrt{\overline\sigma}})}\,dy+ \\  \nonumber
 &\, \hspace{1,8cm}+CM_n^{-1}\int_{2\tilde{S_0}<|y-\xndn|\leq l_n^\star}\log\big(\tfrac{|y-\xndn|}{\rho}\big)\big|y-\tfrac{x_n}{\delta_n}\big|^{-2}\, dy \\ \nonumber
 &\leq C \int_{2\tilde{S_0}}^{l_n^\star}\log\big(\tfrac{r}{\rho}\big) r^{(-1-2(1+\alpha_n)\frac{\delta^{2}}{\sqrt{\overline\sigma}})}\,dr +CM_n^{-1}\int_{2\tilde{S_0}}^{l_n^\star}\log\big(\tfrac{r}{\rho}\big)r^{-1}\, dr \\ \nonumber
 &= C \rho^{-2(1+\alpha_n)\frac{\delta^{2}}{\sqrt{\overline\sigma}}}\int_{\frac{2\tilde{S_0}}{\rho}}^{\frac{l_n^\star}{\rho}}\log (r) \, r^{(-1-2(1+\alpha_n)\frac{\delta^{2}}{\sqrt{\overline\sigma}})}\,dr +CM_n^{-1}\int_{\frac{2\tilde{S_0}}{\rho}}^{\frac{l_n^\star}{\rho}}\log(r)\,r^{-1}\, dr \\ \nonumber
 &\leq  C \Big(\tfrac{(2\tilde{S_0})^{-2(1+\alpha_n)\frac{\delta^2}{\sqrt{\overline\sigma}}}}{2(1+\alpha_n)\tfrac{\delta^2}{\sqrt{\overline\sigma}}}\log(2\tilde{S_0})+\tfrac{(2\tilde{S_0})^{-2(1+\alpha_n)\frac{\delta^2}{\sqrt{\overline\sigma}}}}{(2(1+\alpha_n)\tfrac{\delta^2}{\sqrt{\overline\sigma}})^2}\Big)+CM_n^{-1}(\log(\tfrac{l_n^{\star}}{\rho}))^2
 \end{align*}
 \begin{align*}
 &\leq \tilde{C}(1+M_n^{-1}(\log(\tfrac{l_n^{\star}}{\rho}))^2)
 \leq \tilde{\tilde{C}}(1+M_n^{-1}\log {l_n})\leq C_4,
\end{align*}
where we used the fact that $$M_n^{-1}\log {l_n}\leq M_n^{-1}\log {L_n}=M_n^{-1}\log{(\rho e^{\frac{M_n}{2(1+\alpha_n)}})}\leq c.$$
Finally, we deal with $I_3$, by using (\ref{I3}). Hence we have,
\begin{eqnarray}
\nonumber
 I_3&\leq& \log(\tfrac{l_n}{\rho})\int_{l_n^\star<|y-\xndn|\leq l_n }|y|^{2\alpha_n}\Ks e^{\displaystyle v_n}\, dy \\ \nonumber
 &\leq& C\log(l_n)\Big(C_1(1+\alpha_n)^{\frac{1}{\sqrt{\overline\sigma}}}(l_n^\star)^{-\Big({2(1+\alpha_n)\tfrac{\delta^2}{\sqrt{\overline\sigma}}}\Big)}+CM_n^{-1}\Big) \\
 &\leq& C_5.
\end{eqnarray}
Therefore,
\[
 \int_{|y-\xndn|\leq l_n}\log\big(\tfrac{|y-\xndn|}{\rho}\big)|y|^{2\alpha_n}\Ks e^{\displaystyle v_n}\, dy\leq C,
\]
for a positive constant which does not depend on $n$ and this establishes (\ref{dis2}). \\
We conclude this part and deduce (\ref{dis1}). Clearly, as far as $l_n<L_n$, then  (\ref{dis1}) is true by the definition of $l_n$. In the case $l_n=L_n$, by choosing $R=(\frac{l_n}{4})^{\frac{1}{2}}$ in (\ref{key}), we have that,
\[
 \int_{|y-\xndn|\leq L_n}\Ks(y)e^{\displaystyle v_n(y)}\, d\mu_n \geq 4\pi(1+\alpha_n)(1+{\tfrac{1}{\sqrt{\overline\sigma}}})-C(L_n)^{-{\frac{(1+\alpha_n)\delta^2}{\sqrt{\overline\sigma}}}}-CM_n^{-1}
\]
and by recalling that $L_n=\rho e^{\frac{M_n}{2(1+\alpha_n)}}$, we estimate the right hand side as in (\ref{dis1}). Therefore we have proved both (\ref{dis1}) and (\ref{dis2}).\hfill \\
\medskip

We are ready to obtain the contradiction to (\ref{assurdo}). Indeed, by (\ref{stimamax}), (\ref{dis1}) and (\ref{dis2}) we conclude that, for $n$ large enough,
\begin{eqnarray}
 \nonumber
 M_n &\geq& \int_{|y-\xndn|\leq l_n}\Big(\tfrac{M_n}{4\pi(1+\alpha_n)}-\tfrac{1}{2\pi}\log\big(\tfrac{|y-\xndn|}{\rho}\big)\Big)|y|^{2\alpha_n}\overline{K}_{n}(y)e^{\displaystyle v_n(y)}\,dy + p_n \\ \nonumber
 &\geq&(1+\tfrac{1}{\sqrt{\overline\sigma}})M_n- C +\underset{\O}{\inf}\, u_n,
\end{eqnarray}
for a positive constant $C$, eventually implying that,
\[
 \tfrac{1}{\sqrt{\overline\sigma}}u_n(x_n) +\underset{\O}{\inf} \,u_n\leq C,
\]
which is a contradiction to (\ref{assurdo}). This concludes the proof of case (I).
\hfill\\

\section{Remarks and open problem}\label{Remarks and}

\begin{remark}\label{remark1}
 The proof actually shows that in case (I)  the ``$\sup+\inf$'' inequality \eqref{supinf} holds with the value $\sigma ={\overline\sigma}$. Indeed, let us define $\mathcal{U}$ to be the set of all solutions of (\ref{maineq}) with potential $K$ satisfying (\ref{Kbasic}), (\ref{K1}) and $\alpha\in (-1,0)$. Let $A$ be a compact set in $\Omega$ and, for $u\in \mathcal{U}$, let
 \[
 M_{u}:=\underset{A} \sup\, u, \hspace{1cm} \delta_u:= e^{-\frac{M_u}{2(1+\alpha)}},
 \]
 \[
A_u:=\{x\in A\,|\, u(x)=M_u \},
 \]
 \[
  m_u:=\underset{x\in A_u}\sup |x|,
\hspace{1cm}
R_u:=\tfrac{m_u}{\delta_u}.
\]
 Now, let us define $\mathcal{F}_C:=\{u\in \mathcal{U}\,|\, R_u\leq C \}$. With these definitions, the proof of case (I) shows
 that for any $C>0$ and for any $u\in \mathcal{F}_C $,
 \[
  \tfrac{1}{\sqrt{\overline\sigma}}\underset{A}\sup\, u +\underset{\O}{\inf} \, u\leq \tilde C,
 \]
for a certain $\tilde C$, which does not depend on $u$.

 \end{remark}
 
\begin{remark} \label{remark2}
 Similarly, the proof actually shows that in subcase (i)  the ``$\sup+\inf$'' inequality \eqref{supinf} holds with the value $\sigma ={\overline\sigma}$.
 Indeed, by using the same definitions of Remark \ref{remark1}, for $u\in\mathcal{U}$, let us consider $\hat x_u\in A_u$ such that $|\hat x_u|=m_u$. Then, let us define
\[
\tau_u:=\tfrac{(\delta_u)^{1+\alpha}}{(m_u)^\alpha},
\]
\[
 l_u:=\sup\Big\{l\leq\tfrac{\rho}{\tau_u}\,|\, \small\int_{|x-\hat x_u|\leq l\tau_u} |x|^{2\alpha}K(x)e^{\dis u(x)}\,dx\leq 4\pi(1+\tfrac{1}{\sqrt{\overline\sigma}}) \Big\},
\]
\[
 \hat R_u:=\tfrac{\tau_u}{m_u}l_u.
\]
 Now, let us define $\hat{\mathcal{F}}_C:=\{u\in \mathcal{U}\,|\, \hat R_u\leq C \}$. Hence, the proof of subcase (i) shows that for any $C>0$ and $u\in \hat{\mathcal{F}}_C $,
 \[
  \tfrac{1}{\sqrt{\overline\sigma}}\underset{A}\sup\, u +\underset{\O}{\inf} \, u\leq \tilde C,
 \]
for a certain $\tilde C$, which does not depend on $u$.
\end{remark}

\begin{remark}\label{goodclass}
Based on the estimate \eqref{dis7casocattivo}, it makes sense to ask whether or not the following quantity,
\begin{equation}
 \label{formaindeterminata}
 \left(\frac{|x_n|}{\tau_n}\right)^{-\kappa}|\log(\tau_n)|,
\end{equation}
is bounded as far as $\kappa\in(0,1)$. Indeed, letting $\kappa=\tfrac{\delta^2}{\sqrt{\overline\sigma}}$, if  \eqref{formaindeterminata} was bounded by a uniform constant, then we would deduce from \eqref{dis7casocattivo} that \eqref{dis7casobuono} holds for a certain constant $C$ and the contradiction would arise as above.\\
In particular, there are two types of natural conditions which would ensure that a contradiction would arise at this stage.
The first one is that \eqref{dispeggio} would be satisfied with $\sigma=\overline{\sigma}$, that is, the "missing mass"
$4\pi(1+{\tfrac{1}{\sqrt{\overline\sigma}}})-\int_{|y|\leq l_n}\sing \Ks(y)e^{\displaystyle v_n(y)}\,dy$ should be of order at most $|\log(\tau_n)|^{-1}$.
More exactly, the proof of the second subcase (ii) shows that for every $u\in\mathcal{U}$ (see Remark \ref{remark2}) for which
 $l_u$ satisfies,
 \[
  \small\int_{|x-\hat x_u|\leq l_u\tau_u} |x|^{2\alpha}K(x)e^{\dis u(x)}\,dx\geq 4\pi(1+\tfrac{1}{\sqrt{\overline\sigma}}),
 \]
then the ``$\sup+\inf$'' inequality holds with $\sigma=\overline\sigma$.

\begin{comment} This is the case if for example we assume that there exist $\beta_1>2$, $0<\beta_2<\tfrac{\kappa}{2}(\beta_1-2)$ and $\tilde C>0$ such that
\begin{equation}
 \label{condizionesharp}
 -\beta_1\log(|x_n|)\leq w_n(x_n)\leq \tilde {C}|x_n|^{-\beta_2},
\end{equation}
for $n$ large enough, where $w_n(x)=u_n(x)+2\alpha_n\log(|x_n|)$. Indeed, we notice, by these assumptions, that
\[
 (e^{w_n(x_n)}|x_n|^2)^{\frac{\kappa}{2}}\geq|x_n|^{-\frac{\kappa}{2}(\beta_1-2)}\geq|x_n|^{-\beta_2}\geq \tilde C^{-1}w_n(x_n)
\]
that is
\[
 \left(\frac{|x_n|}{\tau_n}\right)^{-\kappa}|\log(\tau_n)|\leq \frac{w_n(x_n)}{2(e^{w_n(x_n)}|x_n|^2)^{\frac{\kappa}{2}}}\leq \frac{\tilde C}{2}
\]

We notice that \eqref{condizionesharp} is equivalent to require, for $n$ large enough, the following
\begin{equation}
 \label{condizionesharp2}
 -(\beta_1+2\alpha_n)\log(|x_n|)\leq u_n(x_n)\leq 2\tilde {C}|x_n|^{-\beta_2}
\end{equation}
The l.h.s of the previous condition is actually a stronger assumption on the maximum of $u_n$, since from the hypothesis of the second case, namely $\tfrac{|x_n|}{\delta_n}\to\infty$, we only are able to deduce that, for $n$ large, $u_n(x_n)\geq-(2+2\alpha_n)\log|x_n|$.
\medskip
We notice that
\[
 \left(\frac{|x_n|}{\tau_n}\right)^{-\kappa}|\log(\tau_n)|=\frac{\log(e^{\frac{u_n(x_n)}{2}}|x_n|^{\alpha_n})}{(e^{\frac{u_n(x_n)}{2}}|x_n|^{1+\alpha_n})}
\]
\end{comment}
However, a stronger but interesting sufficient condition can be described as follows.
From the hypothesis of case (II), namely $\tfrac{|x_n|}{\delta_n}\to\infty$, we are just able to deduce that,
\begin{equation}\label{azs}u_n(x_n)=-2(1+\alpha_n)\log(|x_n|)+c_n,
\end{equation}
 for some unknown $c_n\to +\ii$. Let us assume a little bit more, that is,
\begin{equation}
 \label{conditionsharp}
 u_n(x_n)\geq-2(1+\alpha_n)(1+\epsilon_0)\log(|x_n|),
\end{equation}
for some $\epsilon_0>0$, which, in turn, by elementary arguments, is equivalent to
$$
\left(\frac{|x_n|}{\tau_n}\right)\geq (\tau_n)^{-\epsilon_1},
$$
for some $\epsilon_1>0$.
%\begin{eqnarray*}
% \frac{|x_n|}{\delta_n}\geq |x_n|^{-\epsilon_0}\Longleftrightarrow \frac{|x_n|}{\tau_n}\geq |x_n|^{-(1+\alpha_n)\epsilon_0} \Longleftrightarrow& \\ \Longleftrightarrow\left(\frac{|x_n|}{\tau_n}\right)^{1-\epsilon_n}\geq |x_n|^{-\epsilon_n}&\Longleftrightarrow \left(\frac{|x_n|}{\tau_n}\right)\geq (\tau_n)^{-\epsilon_n}
%\end{eqnarray*}
%where $\epsilon_n=\tfrac{(1+\alpha_n)\epsilon_0}{1+(1+\alpha_n)\epsilon_0}$and $\epsilon_n\in(0,1)$.
Now, by using the latter condition, we would deduce that,
\[
 \frac{|\log(\tau_n)|}{\left(\frac{|x_n|}{\tau_n}\right)^{\kappa}}\leq \frac{1}{\epsilon_1} \frac{|\log(\frac{|x_n|}{\tau_n})|}{\left(\frac{|x_n|}{\tau_n}\right)^{\kappa}}\leq \tfrac{1}{\epsilon_1}C,
\]
for $n$ sufficiently large. In view of \eqref{azs}, \eqref{conditionsharp}, this fact suggests that the class of solutions for which the ``$\sup+\inf$'' inequality  with $\sigma=\overline\sigma$ could be possibly not satisfied is in fact rather thin. \\
%Indeed, condition \eqref{conditionsharp} does not hold for those solutions $u_n$ which verifies, for $n $ large enough,
%\begin{equation}\label{conditionfail}
% -2(1+\alpha_n)\log(|x_n|)\leq u_n(x_n)\leq-2(1+\alpha_n)(1+\omega_n)\log(|x_n|)
%\end{equation}
%where $\omega_n$ is an infinitesimal sequence which satisfies $\omega_n^{-1}=o(|\log(|x_n|)|)$.

\end{remark}
\medskip
\newtheorem{OpenProblem}{Open Problem}

\begin{OpenProblem} \hfill \\
 Is it true that, under the assumptions of Theorem \ref{teo1.intro}, the ``$\sup+\inf$'' inequality is satisfied  with $\sigma=\overline\sigma$?
\end{OpenProblem}
 This would be a full generalization of the result of Chen-Lin (\cite{CL}) to the case of conical singularities with bounded potentials satisfying \eqref{Kbasic}, \eqref{K1}.

\hfill\\

\section{An explicit example and a two-dimensional ``$\sup \times \inf$'' inequality}\label{Example}
We want to present an explicit example about the sharp ``$\sup+\inf$'' inequality. \\
Let $\O=B_1$ and $A\subset\O$ be a compact set which contains $0$. Let $a<b$ be two positive constants and $\alpha\in(-1,0]$. Let consider the following sequence of functions

\begin{equation}
     u_n(z)=\begin{cases}
        \log\Big(\frac{8(1+\alpha)^2b^{-1}n^{2(1+\alpha)}}{(1+n^{2(1+\alpha)}|z|^{2(1+\alpha)})^2}\Big), & \text{if } |z|<\tfrac{1}{n}, \\
        \log\Big(\frac{8(1+\alpha)^2b^{-1}n^{2(1+\alpha)\sqrt{\frac{a}{b}}}|z|^{2(1+\alpha)(\sqrt{\small\frac{a}{b}}-1)}}{\big(1+n^{2(1+\alpha)\sqrt{\frac{a}{b}}}|z|^{2(1+\alpha)\sqrt{\frac{a}{b}}}\big)^2}\Big), & \text{if } |z|\in[\tfrac{1}{n},1],
    \end{cases}
\end{equation}
which are solutions of the following equation
\[
 -\D u_n(z)=|z|^{2\alpha}K_n(z)e^{\dis u_n(z)}, \,\,\,\,\,\, \text{in}\,|z|\leq1,
\]
where
 \begin{equation}
     K_n(z)=\begin{cases}
        b, & \text{if } |z|<\tfrac{1}{n}, \\
        a, & \text{if } |z|\in[\tfrac{1}{n},1].
    \end{cases}
\end{equation}
The sequence $K_n$ satisfies both (\ref{Kbasic}) and (\ref{K1}). Indeed $a\leq K_n\leq b$ and
\[
 \frac{K_n(x)}{K_n(y)}\leq\frac{b}{a}=:{\overline\sigma},
\]
for every $|x|,|y|\leq 1$. Therefore, we remark that \eqref{K1} holds with ${\overline\sigma}>1$ and $B=0$. At this point, by using the fact that $u_n(x)=u_n(|x|)$ is radial and decreasing in $|x|$, we deduce that $\underset{A}{\sup}\, u_n=u_n(0)$ and $\underset{B_1}{\inf}\, u_n$ is attained at $|z|=1$. Thus,
\begin{align} \nonumber
 \frac{1}{\sqrt{{\overline\sigma}}}\underset{A}{\sup}\, u_n+\underset{B_1}{\inf}\, u_n&= \sqrt{\frac{a}{b}}u_n(0)+\underset{B_1}{\inf}\, u_n \\ \nonumber
 &= \log\Big(\frac{\big(8(1+\alpha)^2b^{-1}\big)^{(\sqrt{\frac{a}{b}}+1)}n^{4(1+\alpha)\sqrt{\frac{a}{b}}}}{\big(1+n^{2(1+\alpha)\sqrt{\frac{a}{b}}}\big)^2}\Big) \\ \nonumber
 &\leq \log\Big(\big(8(1+\alpha)^2b^{-1}\big)^{(\sqrt{\frac{a}{b}}+1)}\Big) \\ \nonumber
 &=(\sqrt{\frac{a}{b}}+1)\log\big(8(1+\alpha)^2b^{-1}\big)=C,
\end{align}
where $C$ is a constant which does not depend on $n$.
\hfill \\
Moreover, let us define
\[
 v_n(z)=u_n(\tfrac{z}{n})-u_n(0)=\begin{cases}
        \log\Big(\frac{1}{(1+|z|^{2(1+\alpha)})^2}\Big) ,& \text{if } |z|<1,\\
        \log\Big(\frac{|z|^{2(1+\alpha)(\sqrt{\small\frac{a}{b}}-1)}}{\big(1+|z|^{2(1+\alpha)\sqrt{\frac{a}{b}}}\big)^2}\Big), & \text{if } |z|\in[1,n].
    \end{cases}
\]
As $n\to+\infty$, we have that,
\[
 v_n(z)\longrightarrow U_{\alpha,a,b}(z)-\log(4(1+\alpha)b^{-1}),
\]
where 
\begin{equation}\label{specificfunction}
 U_{\alpha,a,b}(x)=\begin{cases}
        \log\Big(\frac{8(1+\alpha)^2b^{-1}}{1+|x|^{2(1+\alpha)})^2}\Big), & |x|<1,\\
        \log\Big(\frac{8(1+\alpha)^2b^{-1}|x|^{2(1+\alpha)(\sqrt{\small\frac{a}{b}}-1)}}{\big(1+|x|^{2(1+\alpha)\sqrt{\frac{a}{b}}}\big)^2}\Big), & |x|\in[1,+\infty),
    \end{cases}
\end{equation}

is the function which realizes the optimal total curvature (see for details \cite{b}). \\
At last, we remark that in the case $a=b=1$ and $\alpha=0$, we obtain the same constant, $\log(64)$, which was found by Shafrir in \cite{sh}.

\hfill\\

\medskip

Now, we shortly discuss a geometric application of the ``$\sup+\inf$'' inequality.

\begin{definition}
Let $S$ be a Riemann surface and $P_0\in S$ be an interior point.
 A metric $\tilde{g}_S$ on $S\setminus\{P_0\}$ is said to have a conical singularity of order $\alpha\in(-1,0)$ at $P_0$ if  there exist local coordinates $z(P)\in\O\subset\mathbb{C}$ and $u\in C^0(\O)\cap W_{loc}^{2,p}(\O\setminus\{0\})$, for any $p\in[1,+\infty)$, such that $z(P_0)=0$ and
 \[
  g_s(z)=|z|^{2\alpha}e^{\displaystyle u(z)}|dz|^2, \,\,\,\,\,\, z\in\O,
 \]
 where $g_s$ is the local expression of $\tilde g_s$.
 Then, let us set $\rho_S(z)=|z|^{2\alpha}e^{\displaystyle u(z)}$ the singular (local) conformal factor and $\rho_{S,0}(z)=e^{\dis u(z)}$ its regular part.
\end{definition}
These sort of singular metrics naturally arise in the framework of singular surfaces, see \cite{brej} and references therein.
\hfill\\
Now let us consider a Riemann surface $S$ and a metric $g_S$ for which the Gaussian curvature is any function satisfying \eqref{Kbasic} and \eqref{K1}.
We observe that, for any compact $A\Subset\O$, for which $0\in A$,
\[
\Big(\underset{A}\sup\, \rho_{S}\Big)\Big(\underset{\O}\inf \,\rho_{S}\Big)=+\infty.
\]

However, in suitable local coordinates, we have (Example 1, \cite{brej}),
 \[
  -\D u=|z|^{2\alpha}K(z)e^{\displaystyle u(z)}, \,\,\,\, z\in\O,
 \]
 with $u\in W^{2,q}(\O)\cap W_{loc}^{2,p}(\O\setminus\{0\})$ for any $q\in [1,\tfrac{1}{|\alpha|})$, for any $p\in [1,+\infty) $. By a straightforward application of Theorem \ref{teo1.intro}, we deduce a relevant property of the regular part of the metric, $\rho_{S,0}$.

\begin{theorem}
 For any relatively compact subset $A\Subset\O$ for which $0\in A$, and for every $\sigma>\overline\sigma$, there exists a constant $C>1$ such that
 \[
  \Big(\underset{A}\sup\, \rho_{S,0}\Big)\Big(\underset{\O}\inf \,\rho_{S,0}\Big)^{\sqrt{\sigma}}\leq C.
 \]

\end{theorem}

This inequality can be seen as a two-dimensional singular version of the $\sup \times\inf$ inequalities
which were first established in dimension $N\geq 3$ in the context of the Yamabe problem (\cite{Schoen}), see
also the work of Li and Zhang (\cite{lz}) and references quoted therein.

\hfill\\

\section{Appendix}\label{Appendixsection}

\subsection{A Suzuki-type lemma}\label{Suzuki}

Here we discuss a generalization of a lemma from \cite{suz2}.
\begin{lemma}\label{Suzukilemma}
Let $\O$ be an open bounded domain in $\R^2$ and let $w$ be a solution of
\beq\label{eqSuz}
\qquad -\D w \leq \lambda |x|^{2\a} \e{w}\qquad \text{in} \,\, \O,
\eeq
 with $\a\in (-1,0]$, $\lambda>0$. \\ Suppose that $0\in\O$ and $w\in C_{loc}^{1,\gamma}(\O\backslash \{0\})\cap W^{2,p}_{loc}(\O\backslash\{0\})\cap W^{2,q}_{loc}(\O)$, with $p\geq 1, q\in[1,\frac{1}{|\al|})$ and $\gamma\in (0,1)$. \\
 Then
\begin{equation} \label{dis}
	w(x) \leq \frac{1}{2\pi r}\int\limits_{\partial B_r(x)} w \, ds - 2\log\biggl\{1-\frac{\lambda}{2\beta_{\alpha,\brx}}\int\limits_{B_r(x)} |x|^{2\alpha}e^{\displaystyle w}\, dx\biggl\}_{+}
\end{equation}

holds for $B_r(x)\subset\subset\O$, where $\{\cdot\}_+=\max\{\cdot,0\}$ and
\begin{equation*}
\beta_{\alpha,\omega} = \begin{cases}
4\pi(1+\alpha) \quad&\text{if $(0,0)\in\brx$},\\
4\pi \quad&\text{if $(0,0)\in \O\backslash \overline{\brx}$}.
\end{cases}
\end{equation*}
\end{lemma}

\begin{proof}\hfill
By the Sobolev Embedding Theorem, $w\in C^{0,\kappa}(\ov \brx)$, for $\kappa\in(0,1)$ that depends on $\alpha$.
Let us set,
$$
f:=-\Delta w-\lambda |x|^{2\a} \e{w}\leq 0 \mbox{ in }\brx,
$$
By the regularity of $w$, $f\in L^q(B_r(x))$, so we can consider the function $h_-$, that is the unique solution of
$-\D h_-=f \mbox{ in }\brx,\, h_-=0 \mbox{ on } \pa\brx$. Next, let $h_0$ be the harmonic lifting
of $w$ on $\p \brx$, that is $\D h_0=0$ in $\brx$, $h_0=w$ on $\p\brx$. Since $w\in C_{loc}^{1,\gamma}(\O\backslash \{0\})\cap W^{2,p}_{loc}(\O\backslash\{0\})\cap W^{2,q}_{loc}(\O)$, then, by standard elliptic
theory (\cite{GT}), $h_-$ and $h_0$ are unique, $h_-$ is a subharmonic function of class $ h_-\in C_{loc}^{1,\gamma}(\brx\backslash \{0\})\cap W^{2,p}_{loc}(\brx\backslash\{0\})\cap W^{2,q}_{loc}(\brx)$
and $h_0\in C^2(\brx)\cap C^{0}(\ov\brx)$.
To simplify the notations let us set,
$$
h=h_0+h_-\mbox{ in }\brx.
$$
At this point, we define $u=w-h$, which satisfies,
\beq\label{ineqSuz}
-\D u= \lambda |x|^{2\a}e^{\displaystyle h}e^{\displaystyle u} \mbox{ in }\brx,\quad u=0 \mbox{ on } \p\brx.
\eeq
Clearly $u\in C_{loc}^{1,\gamma}(\brx\backslash \{0\})\cap W^{2,p}_{loc}(\brx\backslash\{0\})\cap W^{2,q}_{loc}(\brx)$, $u\geq 0$ in $\brx$ and we define,
$$
\o(t)=\{x\in\brx\,:\,u >t\}, \quad \gamma(t)=\{x\in\brx\,:\,u=t\}, \quad t\in [0, t_+],
$$
where $t_+=\max\limits_{\ov \brx} u$, and

$$
m(t)=\int\limits_{\o(t)}\lambda |x|^{2\a}e^{\displaystyle h}e^{\displaystyle u}dx,\qquad \mu(t)=\int\limits_{\o(t)}\lambda |x|^{2\a}e^{\displaystyle h}dx.
$$
Since $u$ satisfies \eqref{ineqSuz}, then by Lemma \ref{levelsetlemma} the level sets have vanishing two dimensional area $|\gamma(t)|=0$ for any $t\in [0,t_+]$. In particular, the arguments in  Lemma \ref{levelsetlemma} show that $\{x\in\brx|\grad u(x)=0\}\cap u^{-1}([0,t_+])$ is of measure zero and by the co-area formula in (\cite{BW}) we have that $m(t)$ and $\mu(t)$ are absolutely continuous in $[0,t_+]$.

By the co-area formula and the Sard Lemma we have,
\begin{equation}\label{diffmSuz}
-m^{'}(t)=\int\limits_{\gamma(t)}\frac{\lambda |x|^{2\a}e^{\dis h}e^{ \dis u}}{|\nabla u |}dl=
e^{\dis t}\int\limits_{\gamma(t)}\frac{\lambda |x|^{2\a}e^{ \dis h}}{|\nabla u |}dl=
e^{\dis t}(-\mu^{'}(t)),
\end{equation}
for a.a. $t\in [0, t_+]$, and, in view of \rife{ineqSuz},
\begin{equation}\label{diffemSuz}
m(t)=-\int\limits_{\o(t)}\D u=\int\limits_{\gamma(t)}|\nabla u|,
\end{equation}
for a.a. $t\in [0, t_+]$. By the Schwarz inequality and using a generalization of Huber's inequality (see Proposition \ref{Huber.App.prop} in the appendix) we find that,
$$
-m^{'}(t)m(t)=
\int\limits_{\gamma(t)}\frac{\lambda |x|^{2\a}e^{\displaystyle h}e^{\displaystyle u}}{|\nabla u |}dl\int\limits_{\gamma(t)}|\nabla u|dl=
e^{\dis t}\int\limits_{\gamma(t)}\frac{\lambda |x|^{2\a}e^{\displaystyle h}}{|\nabla u |}dl\int\limits_{\gamma(t)}|\nabla u|dl\geq
$$
\beq\label{hub2Suz}
e^{\dis t}\left(\int\limits_{\gamma(t)}{\lambda^{\frac{1}{2}} |x|^{\a}e^{h/2}} dl\right)^2\geq \beta_{\alpha,\omega(t)} e^{\dis t}\mu(t),
\mbox{ for a.a. } t\in [0, t_+],
\eeq
where $\beta_{\alpha,\omega(t)}$ is equal to $4\pi(1+\alpha)$ if $(0,0)\in\omega(t)$ or $(0,0)$ belongs to the interior part of the bounded components of $\R^2\backslash\omega(t)$ and it equals $4\pi$ in the other cases. %Actually, we remark that, in the last inequality, since $h$ is subharmonic, we have used a generalization of a classical isoperimetric inequality due to Huber, which was firstly proved in the case of open and simply connected domains. \\
By simple calculations we can check that the inequality still holds for a general $\omega(t)$.
Therefore we conclude that,
\begin{equation}\label{diff1Suz}
\frac{1}{2\beta_{\alpha,\omega(t)}}(m^2(t))^{'}+e^{\dis t}\mu(t)\leq 0,\mbox{ for a.a. } t\in [0, t_+].
\end{equation}

In particular, because of \eqref{diffmSuz}, we conclude that, for a.a. $ t\in [0, t_+]$,
\begin{equation}
 \label{disb}
\left(\frac{1}{2\beta_{\alpha,\omega(t)})}m^2(t)-m(t)+e^{\dis t}\mu(t)\right)^{'}=\frac{1}{2\beta_{\alpha,\omega(t)}}(m^2(t))^{'}+e^{\dis t}\mu(t)\leq 0.
\end{equation}

However, as mentioned above, the quantity in the parentheses in the l.h.s. of this inequality is continuous and absolute continuous in $[0, t_+]$, and then we also conclude that, for any $t\in[0, t_+]$
\[
 m(t)-e^{\dis t}\mu(t)\leq \frac{m^2(t)}{2\beta_{\alpha,\omega(t)}},
\]
where we used also the fact that $ m(t_+)=0$ and $ \mu(t_+)=0.$
We rewrite the previous inequality as follows,
\begin{equation}\label{formSuz}
 \mu(t)\geq \frac{m^2(t)}{e^{\dis t}}\bigg(\frac{1}{m(t)}-\frac{1}{2\beta_{\alpha,\omega(t)})}\bigg)
\end{equation}
and define
\[
 j(t)=\frac{\mu(t)}{m(t)}-\frac{\mu(t)}{2\beta_{\alpha,\omega(t)}}.
\]
This function is absolutely continuous in $[0,t]$, for all $t<t_+$, and we notice that

 \[j'(t)=\mu'(t)\biggl(\frac{1}{m(t)}-\frac{1}{2\beta_{\alpha,\omega(t)}}\biggl)-\mu(t)\frac{m'(t)}{m^2(t)}\geq\]

\begin{equation}
 \label{dis3}
 \geq \frac{\mu'(t)e^{\dis t}\mu(t)}{m^2(t)}-\mu(t)\frac{m'(t)}{m^2(t)}=0,
\end{equation}
for a.a. $t\in[0,t_+]$, where we used relations (\ref{diffmSuz}) and (\ref{formSuz}).
Hence, using again (\ref{formSuz}) we can conclude that
\[
 \lim\limits_{t\uparrow t_+} j(t)=\lim\limits_{t\uparrow t_+} \frac{\mu'(t)}{m'(t)}=\frac{1}{e^{\dis t_+}}\geq j(t)=\mu(t)\biggl(\frac{1}{m(t)}-\frac{1}{2\beta_{\alpha,\omega(t)}}\biggl)\geq
\]
\begin{equation}
 \label{dis4}
 \geq\frac{m^2(t)}{e^{\dis t}}\biggl(\frac{1}{m(t)}-\frac{1}{2\beta_{\alpha,\omega(t)}}\biggl)^2_{+},
\end{equation}
for every $t\in[0,t_+]$.
Substituing $t=0$ and remembering that $\omega(0)=\brx$ we have that
\begin{equation}
\label{dis5}
 \frac{1}{e^{\dis t_+}}\geq \biggl(1-\frac{m(0)}{2\beta_{\alpha,\brx}}\biggl)^2_{+}\geq \biggl(1-\frac{\int_{\brx} \lambda|x|^{2\alpha}e^{\displaystyle w}\, dx}{2\beta_{\alpha,\brx}}\biggl)^2_{+}.
\end{equation}

Then, we conclude that,
\begin{equation}
 \label{dis6}
  \biggl(1-\frac{\int_{\brx} \lambda|x|^{2\alpha}e^{\displaystyle w}\, dx}{2\beta_{\alpha,\brx}}\biggl)^{-2}_{+}\geq e^{\dis t_+}\geq e^{\dis u(x)}=e^{\dis w(x)}e^{\dis -h(x)}.
\end{equation}
Hence, by using the subharmonicity of $h$ in $\brx$ we have that,
\begin{eqnarray*}
 w(x) &\leq& h(x)-2\log \biggl\{ 1-\frac{\lambda}{2\beta_{\alpha,\brx}}\int\limits_{B_r(x)} |x|^{2\alpha}e^{\displaystyle w}\, dx\biggl\}_{+} \\
 &\leq& \frac{1}{|\partial\brx|}\int\limits_{\partial\brx}h\, dl -2\log \biggl\{ 1-\frac{\lambda}{2\beta_{\alpha,\brx}}\int\limits_{B_r(x)} |x|^{2\alpha}e^{\displaystyle w}\, dx\biggl\}_{+} \\
 &=& \frac{1}{|\partial\brx|}\int\limits_{\partial\brx}w\, dl -2\log \biggl\{ 1-\frac{\lambda}{2\beta_{\alpha,\brx}}\int\limits_{B_r(x)} |x|^{2\alpha}e^{\displaystyle w}\, dx\biggl\}_{+}
\end{eqnarray*}
and the assertion is proved. \\
\end{proof}

\subsection{Huber inequality}
We state a corollary of the Huber inequality (\cite{hu}), which is suitable to our applications, see also Theorem 5.2, \cite{b}.
We say that $h$ is subharmonic in $\O$, if $h\in W_{loc}^{2,q}(\O)$ for some $q>1$ and $-\D h(x)\leq 0$ for almost any $x\in\O$.

\begin{proposition}\label{Huber.App.prop} \hfill \\
Let $\O\subset\R^2$ be an open, bounded and smooth domain, and let $dl$ denote the arc-lenght on $\p \O$. Let $\Phi$ be a conformal map of $\O$ onto the unit ball $|\xi|<1$ and $k$ a real constant. Let $h$ be a subharmonic function in $\O$, $V_0(x)=|x|^{2\alpha}$, with $\alpha\in(-1,0]$, and assume that $(0,0)\notin\p\O$. Then: either $(0,0)$ belongs to the interior of the (possible multiply connected) bounded component of $\R^2\setminus\O$ and then
\[
 \left(\int_{\p\O}{(e^{\dis h}V_0})^{\frac{1}{2}}\,dl\right)^{2}\geq 4\pi(1+\alpha)\int_{\O}e^{\dis h}V_0\,dx
\]
or
\begin{equation}\label{Huber.App}
 \left(\int_{\p\O}(e^{\dis h}V_0)^{\frac{1}{2}}\,dl\right)^{2}\geq \beta_{\alpha,\O}\int_{\O}e^{\dis h}V_0\,dx,
\end{equation}

where
\[
 \beta_{\alpha,\O}=\left\{\begin{array}{cc}
 4\pi(1+\alpha), &\text{if}\,\,(0,0)\in\O, \\
 4\pi, &\text{if}\,\,(0,0)\notin\O.
 \end{array}
 \right.
\]
The equality holds in \eqref{Huber.App} if and only if $\O$ is simply connected and: either $(0,0)\in\O$ and $V_0(x_1+ix_2)=e^{\dis k}|\Phi'(z)\Phi^\alpha(z)|^2$, $\forall z=x_1+ix_2\in\O$ and $\Phi(0)=0$, or $ (0,0)\notin\overline\O$ and $V_0(x_1+ix_2)=e^{\dis k}|\Phi'(z)|^2$, $\forall z=x_1+ix_2\in\O$. \\
In particular, if $\O$ is not simply connected, then the inequalities are always strict.

\end{proposition}

\subsection{On the measure of level sets of solutions of Liouville-type equations}
We prove a lemma about the measure of the level sets of solutions of (possibly singular) Liouville type equations.

\begin{lemma}\label{levelsetlemma}
 Let $\alpha\in(-1,0]$ and $v$ be a solution of
 \begin{equation} \label{equazione}
  -\D v(z)=|z|^{2\alpha}K(z) e^{\displaystyle v(z)} \qquad \text{in} \,\,\, \O\subset\mathbb{R}^2,
 \end{equation}
 \[
   0<a\leq K(z)\leq b<+\infty,
 \]

with $v\in C_{loc}^{1,\gamma}(\O\backslash \{0\})\cap W_{loc}^{2,p}(\O\backslash\{0\})\cap W_{loc}^{2,q}(\O)\cap C_{loc}^{0,\kappa}(\O)$, where $\gamma\in(0,1)$, $k\in (0,k_0)$ for some $k_0\leq1$ which depends on $\alpha$, $p\geq1$  and $q\in [1,\frac{1}{|\alpha|})$.  \\
Let $t\in[\underset{\O}{\inf}\, v, \underset{\O}{\sup} \, v]$ and set $\Gamma_t=\{z\in\O\,| \, v(x)=t\}$  , then
\[
 |\Gamma_t|=0.
\]
\end{lemma}
\begin{proof}
First of all, we will show that $G=\{x\in \O\,|\, \nabla v(x)=0\}$ has zero measure. \\ Indeed, let us consider $E=\{-\Delta v \,\text{is well defined}\,\}$. Since $v\in W_{loc}^{2,q}(\R^2)$, then $|\O\setminus E|=0$. Moreover, for every $z\in\tilde{F}=E\backslash \{0\}$, $0<\D v(z)<+\infty$ and $|\O\setminus \tilde{F}|=0$. \\
Let consider a relatively compact set $U\subset\subset \O\backslash\{0\}$ and a sequence of functions $\{v_n\}\subset C_{loc}^{\infty}(\O)\cap W_{loc}^{2,p}(\O\backslash\{0\})$ such that
\[
 v_n\longrightarrow v \qquad \text{in}\,\, C^{1,\gamma}(U)\cap W^{2,p}(U).
\]
Then $D^2v_n\longrightarrow D^2 v$ for almost any $x\in U$. We will show that $|G|=0$, by showing that $|G\cap U|=0$ and using the arbitrariness of $U$. \\
Let $F=\tilde{F}\cap\{x\in U\,|\, D^2v_n\,\, \text{pointwise converges to}\,\, D^2v\}$. \\
Le $K$ be a relatively compact set in $F\cap G\cap U$ and let $z_0\in K$. Then $-\D v(z_0)>0$, so we can suppose, without loss of generality, that $\partial^2_{x}v(z_0)<0$ and, for $n>>0$, that $\partial^2_{x}v_n(z_0)<0$. \\
If $\partial_xv_n(z_0)\neq 0$, then there exists a radius $r_0$ such that  $$\{z\in K|\partial_{x}v_n(z)=0\}\cap B_{r_0}(z_0)=\varnothing.$$ Otherwise, by the Implicit Function Theorem, there exists a radius $r_0$ such that the set $$\{z\in K|\partial_{x}v_n(z)=0\}\cap B_{r_0}(z_0)$$ is a $C^1-$graph. In both cases, we deduce that $$|\{z\in K|\partial_{x}v_n(z)=0\}\cap B_{r_0}(z_0)|=0.$$ Moreover, $K\subset\underset{z_0\in K}\bigcup B_{r_0}(z_0)$, so by compactness $K\subset\underset{i\in I}\bigcup B_{r_i}(z_i)$, for a finite set $I$. Then,
\[
 K\cap \{z\in U|\partial_{x}v_n(z)=0\}\subset \underset{i\in I}\bigcup \Big( B_{r_i}(z_i)\cap\{z\in U|\partial_{x}v_n(z)=0\}\Big)
\]
and this implies that $$|K\cap\{z\in U|\nabla v_n(z)=0\}|\leq |K\cap\{z\in U|\partial_{x} v_n(z)=0\}|=0.$$ Using the dominated convergence theorem and the uniformly convergence of $v_n$ to $v$ in $C^1_{loc}$,
\[
 |K\cap\{z\in U|\nabla v(z)=0\}|\leq |K\cap\{z\in U|\nabla v_n(z)\leq\delta\}|\to 0,
\]
when $\delta\to0^+$. For the arbitrariness of $K$, we have that $|F\cap G\cap U|=0$ and, by observing that $|\O\setminus F|=0$, we obtain $|G\cap U|=0$.\\ \hfill \\
Now, by using the coarea formula (\cite{BW}), for $t\in (\underset{\O}{\inf}\, v, \underset{\O}{\sup} \, v)$, we have
\begin{align*}
 |\Gamma_t|&\leq\int_{\O} \chi_{\{t-\delta\leq v\leq t+\delta\}}(x)\,dx \\&=|\{\nabla v(x)=0\}\cap v^{-1}(t-\delta,t+\delta)|+\int_{-\delta}^{\delta}\,ds\,\big(\underset{\{v(x)=s\}}\int\tfrac{1}{|\nabla v(x)|}\,d\sigma(x)\big).
\end{align*}
Thus, as $\delta\to 0^+$, we obtain the desired conclusion whenever $t\in(\underset{\O}{\inf}\, v, \underset{\O}{\sup} \, v)$. \\
Otherwise, if $\hat t=\underset{\O}{\sup}\, v$, we have
\begin{align*}
 |\Gamma_{\hat t}|&\leq\int_{\O} \chi_{\{\hat t-\delta\leq v\leq \hat t\}}(x)\,dx \\
 &=|\{\nabla v(x)=0\}\cap v^{-1}(\hat t-\delta,\hat t)|+\int_{-\delta}^{0}\,ds\,\big(\underset{\{v(x)=s\}}\int\tfrac{1}{|\nabla v(x)|}\,d\sigma(x)\big).
\end{align*}
Thus, as $\delta\to0^+$, we obtain the desired conclusion in this case as well. Clearly the case $t=\underset{\O}{\inf}\, v$ follows in the same way.
\end{proof}

\maketitle

\
\section*{Acknowledgements}
We would like to express our sincere gratitude to Prof. Daniele Bartolucci for the illuminating and inspiring discussions concerning the topics of the paper.

\end{document}